\documentclass[reqno]{amsart}
\usepackage{amsmath}
\usepackage{amsfonts}

\newtheorem{theorem}{Theorem}
\theoremstyle{plain}

\newtheorem{corollary}{Corollary}

\newtheorem{definition}{Definition}

\newtheorem{lemma}{Lemma}

\newtheorem{proposition}{Proposition}
\newtheorem{remark}{Remark}

\DeclareMathOperator{\Div}{div}
\numberwithin{equation}{section}
 \numberwithin{theorem}{section}
 \numberwithin{proposition}{section}
 \numberwithin{remark}{section}
 \numberwithin{definition}{section}
 \numberwithin{lemma}{section}
 \numberwithin{corollary}{section}
 \numberwithin{example}{section}
 \numberwithin{claim}{section}
\input{tcilatex}

\begin{document}
\title[Nonlocal Stokes-Vlasov system]{Nonlocal Stokes-Vlasov system:
Existence and deterministic homogenization results}
\author{Gabriel Nguetseng}
\address{G. Nguetseng, Department of Mathematics, University of Yaounde 1,
P.O. Box 812, Yaounde, Cameroon}
\email{nguetseng@yahoo.fr, nguetseng@uy1.uninet.cm}
\author{Celestin Wafo Soh}
\address{C. Wafo Soh, Department of Mathematics and Statistical Sciences,
College of Science, Engineering and Technology, Jackson State University,
JSU Box 17610, 1400 J R Lynch St., Jackson, MS 39217, USA}
\email{celestin.wafo\_soh@jsums.edu}
\author{Jean Louis Woukeng}
\address{J. L. Woukeng, Department of Mathematics and Computer Science,
University of Dschang, P.O. Box 67, Dschang, Cameroon}
\email{jwoukeng@yahoo.fr}
\date{June, 2014}

\begin{abstract}
Our work deals with the systematic study of the coupling between the
nonlocal Stokes system and the Vlasov equation. The coupling is due to a
drag force generated by the fluid-particles interaction. We establish the
existence of global weak solutions for the nonlocal Stokes-Vlasov system in
dimensions two and three without resorting to assumptions on higher-order
velocity moments of the initial distribution of particles. We then study by
the means of the sigma-convergence method, the asymptotic behavior in the
general deterministic framework, of the sequence of solutions to the
nonlocal Stokes-Vlasov system. In guise of illustration, we provide several
physical applications of the homogenization result including periodic,
almost-periodic and weakly almost-periodic settings.
\end{abstract}

\maketitle

\section{Introduction}

This paper is concerned with the rigorous asymptotic analysis of a system of
integro-differential equations modeling the evolution of a cloud of
particles immersed in an incompressible viscous fluid. We neglect
particle-particle collisions in such a way that at the microscale level,
particles' distribution, $f_{\varepsilon }$, satisfies the Vlasov equation 
\begin{equation}
\frac{\partial f_{\varepsilon }}{\partial t}+\varepsilon v\cdot \nabla
f_{\varepsilon }+\Div_{v}\left( (\boldsymbol{u}_{\varepsilon
}-v)f_{\varepsilon }\right) =0\mbox{ in }Q\times \mathbb{R}^{N}  \label{1.1}
\end{equation}%
in which the non-dimensional small parameter $\varepsilon >0$ represents the
scale of inhomogeneities, $\boldsymbol{u}_{\varepsilon }(t,x)$ is fluid's
velocity at time $t$ and position $x$, $f_{\varepsilon }(t,x,v)dv$ is
roughly the odd of finding a particle with velocity $v$ near $x$ at time $t$%
, $Q=(0,T)\times \Omega $, $\Omega \subset \mathbb{\ R}^{N}$ ($N=2,3$) is a
bounded domain with smooth boundary , $T$ is a given positive real number
representing the final time, the operator $\nabla $ (resp. $\Div_{v}$)
denotes the gradient operator with respect to $x\in \Omega $ (resp. the
divergence operator in $\mathbb{R}^{N}$ with respect to $v\in \mathbb{R}^{N}$%
). We posit that the cloud of particles is highly diluted in such a way that
we may assume that the density of the fluid is constant. Thus, the particles
evolve in a Newtonian fluid governed by the Stokes system. The viscoelastic
constitutive law associated to the momentum balance and the fluid-particles
interaction give rise to the following Stokes system: 
\begin{equation}
\frac{\partial \boldsymbol{u}_{\varepsilon }}{\partial t}-\Div\left(
A_{0}^{\varepsilon }\nabla \boldsymbol{u}_{\varepsilon
}+\int_{0}^{t}A_{1}^{\varepsilon }(t-\tau ,x)\nabla \boldsymbol{u}%
_{\varepsilon }(\tau ,x)d\tau \right) +\nabla p_{\varepsilon }=-\int_{%
\mathbb{R}^{N}}(\boldsymbol{u}_{\varepsilon }-v)f_{\varepsilon }dv\mbox{ in }%
Q,  \label{1.2}
\end{equation}%
\begin{equation}
\Div\boldsymbol{u}_{\varepsilon }=0\mbox{ in }Q,  \label{1.3}
\end{equation}%
where $p_{\varepsilon }$ is pressure and the oscillating viscosities $%
A_{0}^{\varepsilon }$ and $A_{1}^{\varepsilon }$ are defined by $%
A_{i}^{\varepsilon }(t,x)=A_{i}\left( t,x,\frac{t}{\varepsilon },\frac{x}{%
\varepsilon }\right) $ ($(t,x)\in Q$ and $i=0,1$), with the $A_{i}$s
constrained as follows:

\begin{itemize}
\item[(\textbf{A1})] $A_{i}\in \mathcal{C}(\overline{Q};L^{\infty }(\mathbb{R%
}_{y,\tau }^{N+1})^{N^{2}})$ are symmetric matrices with $A_{0}$ satisfying
the following condition: 
\begin{equation*}
A_{0}\xi \cdot \xi \geq \alpha \left\vert \xi \right\vert ^{2}%
\mbox{ for all
}\xi \in \mathbb{R}^{N}\mbox{ and a.e. in }\overline{Q}\times \mathbb{R}%
_{y,\tau }^{N+1}
\end{equation*}%
with $\alpha >0$ a given constant not depending on $x,t,y,\tau $ and $\xi $.
\end{itemize}

The system (\ref{1.1})-(\ref{1.3}) is supplemented with the initial data 
\begin{equation}
\boldsymbol{u}_{\varepsilon }(0,x)=\boldsymbol{u}^{0}(x),\ f_{\varepsilon
}(0,x,v)=f^{0}(x,v),\ x\in \Omega ,v\in \mathbb{R}^{N},  \label{1.4}
\end{equation}%
and the boundary conditions 
\begin{equation}
\boldsymbol{u}_{\varepsilon }=0\mbox{ on }\partial \Omega \mbox{ and }%
f_{\varepsilon }(t,x,v)=f_{\varepsilon }(t,x,v^{\ast })\mbox{ for }x\in
\partial \Omega \mbox{ with }v\cdot \nu (x)<0,  \label{1.5}
\end{equation}%
where $v^{\ast }=v-2(v\cdot \nu (x))\nu (x)$ is the specular velocity, $\nu
(x)$ is the outward normal to $\Omega $ at $x\in \partial \Omega $ and the
functions $\boldsymbol{u}^{0}$ and $f^{0}$ are chosen as follows:

\begin{itemize}
\item[(\textbf{A2})] $\boldsymbol{u}^{0}\in L^{2}(\Omega )^{N}$ with $\Div%
\boldsymbol{u}^{0}=0$, $f^{0}\geq 0$, $f^{0}\in L^{\infty }(\Omega \times 
\mathbb{R}^{N})\cap L^{1}(\Omega \times \mathbb{R}^{N})$.
\end{itemize}

It is opportune to stress that we have not imposed the constraint $|v|^5f^0
\in L^1(\Omega \times \mathbb{R}^N)$ as suggested by Yu \cite{JMPA}. Indeed,
we are going to see that the Lemma 2.1 of Hamdache \cite{Hamdache} renders
such assumption superfluous provided appropriate regularization and
truncation are performed. In particular, the truncation of the initial
distribution of particles in the $v$-direction relieves us from the
assumption on moments. Contrary to contemporary approaches, ours permeates
initial distribution of the form $\alpha(x)/ (1 + |v|^2)$, where $\alpha \ge
0$ and $\alpha \in L^{\infty}(\Omega)$.

The system (\ref{1.1})-(\ref{1.5}) arises in several applications comprising
the modeling of reactive flows of sprays \cite{Amsden, Oran}, atmospheric
pollution modeling \cite{Greenberg}, and waste water treatment \cite{Gokhale}%
. When there is no particle evolving in the fluid (i.e. when $f_{\varepsilon
}\equiv 0$) the asymptotic analysis of Eqs. (\ref{1.1})-(\ref{1.5}) reduces
to the study of the asymptotics of (\ref{1.2})-(\ref{1.5}) (with of course $%
f_{\varepsilon }\equiv 0$ therein), which has been very recently undertaken
by Woukeng \cite{M2AS} in the almost periodic framework. %It is important
%to note that if in equation (\ref{1.2}) we replace the gradient $\nabla
%\boldsymbol{u}_{\varepsilon }$ by its symmetric part $\frac{1}{2}(\nabla
%\boldsymbol{u}_{\varepsilon }+\nabla ^{T}\boldsymbol{u}_{\varepsilon })$,
%then thanks to the Korn's inequality, the mathematical analysis does not
%change, although the model becomes in this case, physical.

There is a fairly extensive literature concerned with the asymptotic
analysis of the Vlasov equations coupled with other equations: the works 
\cite{Crous, Frenod1, Frenod} deal with the periodic homogenization of the
Vlasov equations, the main thrust of the papers \cite{Frenod, Jiang} is the
asymptotic analysis of the coupling Vlasov-Poisson system, Mellet et al. 
\cite{Mellet2} is concerned with the homogenization of the coupling
Vlasov-Fokker-Planck/Compressible Navier-Stokes system, and Goudon et al. 
\cite{Goudon1, Goudon2} treats the asymptotic behavior of the coupling
Vlasov-Navier-Stokes equations.

In this work, our objective is twofold: 1) we state and prove an existence
result for the system (\ref{1.1})-(\ref{1.5}) without any assumption on the $%
v$-moments of the initial condition $f^{0}$; 2) we carry out the
homogenization of (\ref{1.1})-(\ref{1.5}) under suitable structural
assumptions on the coefficients of the operators involved in (\ref{1.2}). 
%Precisely, we investigate the
%asymptotic behavior of the sequence of solutions to (\ref{1.1})-(\ref{1.5})
%as $\varepsilon \rightarrow 0$ under reasonably general assumptions on the coefficients $%
%A_{i} $ ($i=0,1$).
These assumptions cover a wide set of concrete behaviors such as the
classical periodicity assumption, the almost periodicity hypothesis, weakly
almost periodicity hypothesis and much more. In order to achieve our goal,
we shall use the concept of \emph{sigma-convergence} \cite{Hom1, NA} which
is roughly a formulation of the well-known two-scale convergence method \cite%
{NG} in the context of \emph{algebras with mean value} \cite{Jikov, NA,
Zhikov4}. This is the so-called \emph{deterministic homogenization} theory
which includes the periodic homogenization theory as a special case. As far
as we know, our results are new in the context of general deterministic
homogenization since the available results deal with either periodic
homogenization \cite{Crous, Frenod1, Frenod, Jiang}\ or rely on the concept
of \emph{relative entropy} \cite{Goudon1, Goudon2, Mellet2, Poupaud1,
Poupaud2, Yau}.

The remainder of this paper is structured as follows. In Section 2, we state
and outline the proof of an existence result for our $\varepsilon $-problem.
We also derive some a priori estimates that will be useful in next sections.
Section 3 deals with the concept of $\Sigma $-convergence and its relation
with convolution. We first recall some useful tools related to algebras with
mean value and define convolution over the spectrum of an algebra with mean
value. In Section 4, we state and prove the main homogenization result. In
Section 5, we give some concrete situations in which the result of Section 4
is valid. Finally, we summarize our findings in Section 6.

In the sequel, unless otherwise specified, the field of scalars acting on
vector spaces is the set of real numbers and scalar functions are
real-valued. If $X$ and $F$ respectively denote a locally compact space and
a Banach space, then we respectively write $\mathcal{C}(X;F)$ and \textrm{BUC%
}$(X;F)$ for continuous mappings of $X$ into $F$ and bounded uniformly
continuous mappings of $X$ into $F$. We shall always assume that \textrm{BUC}%
$(X;F)$ is equipped with the supremum norm $\left\Vert u\right\Vert _{\infty
}=\sup_{x\in X}\left\Vert u(x)\right\Vert $ in which $\left\Vert \cdot
\right\Vert $ stands for the norm of $F$. In the notations for functions
space, we shall omit the codomain when it is $\mathbb{R}$. To wit, $\mathcal{%
C}(X)$ will stand for $\mathcal{C}(X;\mathbb{R})$ and \textrm{BUC}$(X)$ will
be a shorthand notation for \textrm{BUC}$(X;\mathbb{R})$. Likewise, the
usual Lebesgue spaces $L^{p}(X;\mathbb{R})$ and $L_{\mbox{loc}}^{p}(X;%
\mathbb{R})$ where $X$ is equipped with a positive Radon measure, are
respectively abbreviated $L^{p}(X)$ and $L_{\mbox{loc}}^{p}(X)$. Finally, it
will always be assumed that the Euclidian space $\mathbb{R}^{N}\;(N\geq 1)$
and its open sets are each endowed with Lebesgue measure $dy=dy_{1}\ldots
dy_{N}$.

\section{Existence result and basic a priori estimates}

In this part, we focus on the existence of solutions to our $\varepsilon$%
-problem. We shall define a regularized problem, solve it and show that the
limit (in a sense to be specified) of the solution of this regularized
problem solves our $\varepsilon$-problem. In order to implement our program,
we shall establish some \emph{a priori} estimates in some classical
functional spaces we introduce below. These estimates will be used in
compactness arguments at various stages of this work.

The main classical spaces involved in the mathematical study of
incompressible fluid flows are spaces connected to kinetic energy, entropy,
the boundary conditions and the conservation of mass. These spaces will be
denoted $V$ and $H$ and they may be respectively constructed as closure of 
%We begin this section by defining the notion of weak solution we shall deal
%with in this work. Prior to that, we introduce the following usual spaces: $%
$\mathcal{V}=\{\boldsymbol{\varphi }\in \mathcal{C}_{0}^{\infty }(\Omega
)^{N}:\Div\boldsymbol{\varphi }=0\}$ in $H^{1}(\Omega )^{N}$ ($H^{1}(\Omega )
$ the usual Sobolev space on $\Omega $), and $L^{2}(\Omega )^{N}$. Since $%
\partial \Omega $ is smooth, we have that $V=\{\boldsymbol{u}\in
H_{0}^{1}(\Omega )^{N}:\Div\boldsymbol{u}=0\}$ and $H=\{\boldsymbol{u}\in
L^{2}(\Omega )^{N}:\Div\boldsymbol{u}=0$ and $\boldsymbol{u}\cdot \nu =0$ on 
$\partial \Omega \}$, $\nu $ being the outward unit vector normal to $%
\partial \Omega $. We denote by $\left( \cdot ,\cdot \right) $ the inner
product in $H$, and by $\cdot $ the scalar product in $\mathbb{R}^{N}$. The
associated norm in $\mathbb{R}^{N}$ is denoted by $\left\vert \cdot
\right\vert $. All duality pairing are denoted by $\left\langle \cdot ,\cdot
\right\rangle $ without referring to spaces involved. Such spaces will be
understood from the context. We set 
\begin{equation*}
\Sigma ^{\pm }=\{(x,v)\in \partial \Omega \times \mathbb{R}^{N}:\pm v\cdot
\nu (x)>0\}.
\end{equation*}

With the functional framework fixed, we can now specify the type of
solutions we will be seeking.

\begin{definition}
\label{d2.1}\emph{A pair }$(\boldsymbol{u}_{\varepsilon },f_{\varepsilon })$%
\emph{\ (for fixed }$\varepsilon >0$\emph{) is called a }weak solution\emph{%
\ to the system (\ref{1.1})-(\ref{1.5}) if the following conditions are
satisfied:}

\begin{itemize}
\item $\boldsymbol{u}_{\varepsilon }\in L^{\infty }(0,T;H)\cap
L^{2}(0,T;V)\cap \mathcal{C}([0,T];V^{\prime });$

\item $f_{\varepsilon }(t,x,v)\geq 0$\emph{\ for any }$(t,x,v)\in Q\times 
\mathbb{R}^{N};$

\item $f_{\varepsilon }\in L^{\infty }(0,T;L^{\infty }(\Omega \times \mathbb{%
R}^{N})\cap L^{1}(\Omega \times \mathbb{R}^{N}));$

\item $f_{\varepsilon }\left\vert v\right\vert ^{2}\in L^{\infty
}(0,T;L^{1}(\Omega \times \mathbb{R}^{N}));$

\item \emph{for all }$\phi \in \mathcal{C}^{1}([0,T]\times \Omega \times 
\mathbb{R}^{N})$\emph{\ with compact support in }$v$\emph{, such that }$\phi
(T,\cdot ,\cdot )=0$\emph{\ and }$\phi (t,x,v)=\phi (t,x,v^{\ast })$\emph{\
on }$\left( 0,T\right) \times \Sigma ^{+}$\emph{, we have }%
\begin{equation}
\int_{Q\times \mathbb{R}^{N}}f_{\varepsilon }\left( \frac{\partial \phi }{%
\partial t}+\varepsilon v\cdot \nabla \phi +(\boldsymbol{u}_{\varepsilon
}-v)\cdot \nabla _{v}\phi \right) dxdvdt+\int_{\Omega \times \mathbb{R}%
^{N}}f^{0}\phi (0,x,v)dxdv=0;  \label{2.1}
\end{equation}

\item \emph{for all }$\psi \in \mathcal{C}^{1}([0,T]\times \Omega )^{N}$%
\emph{\ with }$\Div\psi =0$\emph{\ and }$\psi (T,\cdot )=0,$%
\begin{eqnarray}
\int_{Q}\left( -\boldsymbol{u}_{\varepsilon }\cdot \frac{\partial \psi }{%
\partial t}+(A_{0}^{\varepsilon }\nabla \boldsymbol{u}_{\varepsilon
}+A_{1}^{\varepsilon }\ast \nabla \boldsymbol{u}_{\varepsilon })\cdot \nabla
\psi \right) \;dxdt &=&-\int_{Q\times \mathbb{R}^{N}}f_{\varepsilon }(%
\boldsymbol{u}_{\varepsilon }-v)\cdot \psi \;dxdtdv  \label{2.2} \\
&&+\int_{\Omega }\boldsymbol{u}^{0}\cdot \psi (0,x)\;dx.  \notag
\end{eqnarray}
\end{itemize}
\end{definition}

In Eq. (\ref{2.2}) $A_{1}^{\varepsilon }\ast \nabla \boldsymbol{u}%
_{\varepsilon }$ stands for the function defined by 
\begin{equation*}
(A_{1}^{\varepsilon }\ast \nabla \boldsymbol{u}_{\varepsilon
})(t,x)=\int_{0}^{t}A_{1}^{\varepsilon }(t-\tau ,x)\nabla \boldsymbol{u}%
_{\varepsilon }(\tau ,x)\;d\tau \mbox{ whenever } (t,x)\in Q.
\end{equation*}%
The main result of this section is summarized in the following theorem.

\begin{theorem}
\label{t2.1}Under assumption \emph{(\textbf{A1})-(\textbf{A2})} and for any
fixed $\varepsilon >0$, there exists a weak solution $(\boldsymbol{u}%
_{\varepsilon },f_{\varepsilon })$ of \emph{(\ref{1.1})-(\ref{1.5})} in the
sense of Definition \emph{\ref{d2.1}}. There also exists a $p_{\varepsilon
}\in L^{2}(0,T;L^{2}(\Omega )/\mathbb{R})$ such that \emph{(\ref{1.2})} is
satisfied.
\end{theorem}

The proof of Theorem \ref{t2.1} will be done in several steps described in
the subsections that follow. The general idea is loosely to regularize our
problem, solve the regularized problem and take the limit of its solution to
obtain a solution of our problem.

\subsection{Regularization and truncation}

We start by fixing notations that will be useful in the sequel. Let $(\theta
_{\lambda })_{\lambda >0}$ be a mollifying sequence in $x$ i.e. its terms
belong to $\mathcal{C}_{0}^{\infty }(\mathbb{R}_{x}^{N})$ (the space of
compactly supported and smooth functions) such that $\theta _{\lambda
}(x)=\lambda ^{-N}\theta (x/\lambda )$ for all $x\in \mathbb{R}^{N}$ and $%
\lambda >0$ , where $\theta \in \mathcal{C}_{0}^{\infty }(\mathbb{R}%
_{x}^{N}) $, $0\leq \theta \leq 1$, $\mbox{supp}(\theta )\subset
B(0,1)=\{x\in \mathbb{R}^{N}:\left\vert x\right\vert \leq 1\}$ and $\int_{%
\mathbb{R}^{N}}\theta (x)dx=1$. A regularizing sequence in $(x,v)\in \mathbb{%
R}_{x}^{N}\times \mathbb{R}_{v}^{N}$ will be denoted $(\Theta _{\lambda
})_{\lambda >0}$. We shall use the same notation for convolution in $x$ and $%
(x,v)$. The context shall indicate which convolution is used. For
vector-valued functions, convolution is done componentwise. We also consider 
$\gamma \in \mathcal{C}_{0}^{\infty }(\mathbb{R}_{v}^{N})$ such that $%
\mbox{supp}(\gamma )\subset B(0,2)$, $\gamma (v)=1$ for all $v\in B(0,1)$
and $0\leq \gamma \leq 1$. We define the truncating sequence $(\gamma
_{\lambda })_{\lambda >0}$ by $\gamma _{\lambda }(v)=\gamma (\lambda v)$. It
can be verified that $\gamma _{\lambda }(v)\rightarrow 1$ as $\lambda
\rightarrow 0$.

In the sequel, a function defined on $\Omega $ will be extended by zero
outside of $\Omega $ before convolution although we shall keep the same
notation for the extended function. Furthermore, especially in inequalities
involving $\lambda $, we shall assume throughout that $0<\lambda \leq 1$.
The latter assumption is used when needed to obtained uniform estimates in $%
\lambda $.

%For
%$\boldsymbol{w}\in L^2(0,T, V)$, we define $\hat{\boldsymbol{w}}_{\lambda}$ as follows:
% $\hat{\boldsymbol{w}}_{\lambda}(., x) = \boldsymbol{w}(., x)$ for
% $x\in \Omega_{\lambda} = \{ x\in \Omega: d(x, \partial \Omega) >\lambda\}$
%  and  $\hat{ \boldsymbol{w}}_{\lambda}(., x) = 0$ for
%   $x\in \Omega -\Omega_{\lambda}$. We still call $\hat{\boldsymbol{ w}}_{\lambda}$ the extension of
%   $\hat{\boldsymbol{w}}_{\lambda}$ by $0$ outside of $\Omega$ and we  set
%\begin{equation}
%    \boldsymbol{w}_{\lambda} = \hat{\boldsymbol{ w}}_{\lambda}* \theta_{{\lambda \over 2}}
%    , \label{pr1}
% \end{equation}
% where $*$ stands for convolution with respect to $x$. It can be verified that $\boldsymbol{w}_{\lambda}$ is smooth in $x$, $\boldsymbol{w}_{\lambda} = 0$ on $\partial \Omega$ and $\mbox{div } \boldsymbol{w}_{\lambda}=0 $ on $\mathbb{R}^N_x$.
%
%Now, if we fix $\boldsymbol{w} \in L^2(0,T; V)$, then

Let $\boldsymbol{w}\in L^{2}(0,T;V)$. 
% such that $\partial \boldsymbol{w}/ \partial t \in L^2(0,T; V')$.
The regularized system associated to our $\varepsilon $-problem takes the
following form:

\begin{equation}
\frac{\partial f_{\varepsilon,\lambda }}{\partial t}+\varepsilon v\cdot
\nabla f_{\varepsilon,\lambda }+\Div_{v}\left( (\boldsymbol{w}%
*\theta_{\lambda} -v)f_{\varepsilon,\lambda }\right) =0\mbox{ in }Q\times 
\mathbb{R}^{N}  \label{r1}
\end{equation}

\begin{equation}
\frac{\partial \boldsymbol{u}_{\varepsilon ,\lambda }}{\partial t}-\Div%
\left( A_{0}^{\varepsilon }\nabla \boldsymbol{u}_{\varepsilon ,\lambda
}+\int_{0}^{t}A_{1}^{\varepsilon }(t-\tau ,x)\nabla \boldsymbol{w}(\tau
,x)d\tau \right) +\nabla p_{\varepsilon ,\lambda }=-\int_{\mathbb{R}^{N}}(%
\boldsymbol{w}\ast \theta _{\lambda }-v)\gamma _{\lambda }(v)f_{\varepsilon
,\lambda }dv\mbox{ in }Q,  \label{r2}
\end{equation}%
\begin{equation}
\Div\boldsymbol{u}_{\varepsilon ,\lambda }=0\mbox{ in }Q,  \label{r3}
\end{equation}%
The system (\ref{r1})-(\ref{r3}) is supplemented with the following initial
and boundary conditions. 
\begin{equation}
\begin{array}{l}
\mbox{a) }\boldsymbol{u}_{\varepsilon ,\lambda }(0,x)=u^{0}(x); \\ 
\mbox{b) }f_{\varepsilon ,\lambda }(0,x,v):=f_{\lambda }^{0}(x,v)=\gamma
_{\lambda }(v)(f^{0}\ast \Theta _{\lambda })(x,v),\ x\in \Omega ,v\in 
\mathbb{R}^{N},%
\end{array}
\label{r4}
\end{equation}%
and the boundary conditions 
\begin{equation}
\begin{array}{l}
\mbox{a) }\boldsymbol{u}_{\varepsilon ,\lambda }=0\mbox{ on }\partial \Omega
; \\ 
\mbox{b) }f_{\varepsilon ,\lambda }(t,x,v)=f_{\varepsilon ,\lambda
}(t,x,v^{\ast })\mbox{ for }x\in \partial \Omega \mbox{ with }v\cdot \nu
(x)<0.%
\end{array}
\label{r5}
\end{equation}

\subsection{Existence, regularity and estimates of ${f_{\protect\varepsilon ,%
\protect\lambda }}$}

In this part, we focus on Eq. (\ref{r1}) coupled with the initial condition (%
\ref{r4} b) and the boundary condition (\ref{r5} b). In what follows, we use 
$C$ as a generic name for positive constants independent of both $%
\varepsilon $ and $\lambda $. In all the estimates, we suppose when the need
arises that $\varepsilon $ and $\lambda $ are sufficiently small. We shall
assume throughout that $f^{0}\in L^{p}(\Omega \times \mathbb{R}^{n})\cap
L^{\infty }(\Omega \times \mathbb{R}^{n})$, $p\geq 1$, unless we mention
otherwise.

Since $|f_{\lambda }^{0}|\leq |f^{0}\ast \Theta _{\lambda }|$, we have 
\begin{equation*}
\left\Vert f_{\lambda }^{0}\right\Vert _{L^{p}(\Omega \times \mathbb{R}%
^{N})}\leq \left\Vert f^{0}\ast \Theta _{\lambda }\right\Vert _{L^{p}(\Omega
\times \mathbb{R}^{N})}\leq \left\Vert f^{0}\right\Vert _{L^{p}(\Omega
\times \mathbb{R}^{N})}\left\Vert \Theta _{\lambda }\right\Vert
_{L^{1}(\Omega \times \mathbb{R}^{N})}\leq \left\Vert f^{0}\right\Vert
_{L^{p}(\Omega \times \mathbb{R}^{N})}
\end{equation*}%
in such a way that 
\begin{equation}
\left\Vert f_{\lambda }^{0}\right\Vert _{L^{p}(\Omega \times \mathbb{R}%
^{N})}\leq \left\Vert f^{0}\right\Vert _{L^{p}(\Omega \times \mathbb{R}%
^{N})}.  \label{r6}
\end{equation}

By Theorem 4 of Mischler \cite{SM}, we infer that $f_{\varepsilon ,\lambda }$
uniquely exists and belongs to $L^{\infty }(0,T,L^{\infty }(\Omega \times 
\mathbb{R}^{N})\cap L^{p}(\Omega \times \mathbb{R}^{N}))$. Since both $%
f_{\lambda }^{0}$ and the coefficients of Eq. (\ref{r1}) are $\mathcal{C}%
^{\infty }$, we can deduce using the method of characteristics that $%
f_{\varepsilon ,\lambda }$ is nonnegative and belongs to $\mathcal{C}%
^{1}((0,T)\times \Omega \times \mathbb{R}^{N})$. Following \cite[p. 54]%
{Hamdache}, we have 
\begin{equation}
\frac{d}{dt}\left( e^{Nt}\int_{\Omega \times \mathbb{R}^{N}}\left(
e^{-Nt}f_{\varepsilon ,\lambda }\right) ^{p}dxdv\right) =0.  \label{r7}
\end{equation}%
Then, by integrating both sides of Eq (\ref{r7}) from $0$ to $t$, we obtain 
\begin{equation}
\left\Vert f_{\varepsilon ,\lambda }(t)\right\Vert _{L^{p}(\Omega \times 
\mathbb{R}^{N})}=\mbox{e}^{Nt\left( 1-\frac{1}{p}\right) }\left\Vert
f_{\lambda }^{0}\right\Vert _{L^{p}(\Omega \times \mathbb{R}^{N})}
\label{r8}
\end{equation}%
Thus, by using the inequality (\ref{r6}), we arrive at the estimate 
\begin{equation}
\left\Vert f_{\varepsilon ,\lambda }\right\Vert _{L^{\infty
}(0,T;L^{p}(\Omega \times \mathbb{R}^{N}))}\leq C(N,T,p)\left\Vert
f^{0}\right\Vert _{L^{p}(\Omega \times \mathbb{R}^{N})}.  \label{r9}
\end{equation}

Now, we turn our attention to the estimates of $v$-moments of $%
f_{\varepsilon ,\lambda }$. Lemma 2.1 of Hamdache \cite{Hamdache} will be
our workhorse. Let us first observe that $f_{\lambda }^{0}\in L^{\infty
}(\Omega \times \mathbb{R}^{N})\cap L^{1}(\Omega \times \mathbb{R}^{N})$
since $f_{\lambda }^{0}$ is compactly supported and $f^{0}\in L^{\infty
}(\Omega \times \mathbb{R}^{N})\cap L^{p}(\Omega \times \mathbb{R}^{N})$, $%
p\geq 1$ be fixed. Furthermore, we assert that for any $m\geq 1$, 
\begin{equation}
\int_{\Omega \times \mathbb{R}^{N}}\left\vert v\right\vert ^{m}f_{\lambda
}^{0}dxdv\leq C(N,p)\left\Vert f^{0}\right\Vert _{L^{p}(\Omega \times 
\mathbb{R}^{N})}.  \label{r10}
\end{equation}%
Indeed, the following inequalities 
\begin{eqnarray}
\int_{\Omega \times \mathbb{R}^{N}}\left\vert v\right\vert ^{m}f_{\lambda
}^{0}dxdv &\leq &\int_{\Omega \times \{\left\vert v\right\vert \leq
2\}}\left\vert v\right\vert ^{m}f_{\lambda }^{0}\,dxdv+\int_{\Omega \times
\{\left\vert v\right\vert >2\}}\left\vert v\right\vert ^{m}f_{\lambda
}^{0}dxdv=\int_{\Omega \times \{\left\vert v\right\vert \leq 2\}}\left\vert
v\right\vert ^{m}f_{\lambda }^{0}dxdv  \notag \\
&\leq &\int_{\Omega \times \{\left\vert v\right\vert \leq 2\}}f_{\lambda
}^{0}dxdv\leq \left\vert \Omega \times \{\left\vert v\right\vert \leq
2\}\right\vert ^{1/q}\left\Vert f_{\lambda }^{0}\right\Vert _{L^{p}(\Omega
\times \mathbb{R}^{N}}  \notag \\
&\leq &C(N,p)\,\left\Vert f^{0}\right\Vert _{L^{p}(\Omega \times \mathbb{R}%
^{N})}\mbox{ (see (\ref{r6})),}  \label{r11}
\end{eqnarray}%
where $q$ is the conjugate exponent of $p$, are true. In view of the
continuous embedding $H_{0}^{1}(\Omega )\hookrightarrow L^{r}(\Omega )$ for
any $1\leq r\leq 6$, and since $N\leq 3$, we may choose $m\geq 1$ such that $%
N+m\leq 6$ (for example $1\leq m\leq 4$ if $N=2$, and $1\leq m\leq 3$ for $%
N=3$). Thus, for such an $m$, the fact that $\boldsymbol{w}\in
L^{2}(0,T;L^{N+m}(\Omega )^{N})$ steems from both the above continuous
embedding and $\boldsymbol{w}\in L^{2}(0,T;V)$. With this in mind and taking
into account Eq. (\ref{r10}), we see that we are within the hypotheses of 
\cite[Lemma 2.1]{Hamdache}. Hence, the following estimate holds: 
\begin{equation}
\int_{\Omega \times \mathbb{R}^{N}}\left\vert v\right\vert
^{m}f_{\varepsilon ,\lambda }\,dxdv\leq C(N,T)\left[ \left( \int_{\Omega
\times \mathbb{R}^{N}}\left\vert v\right\vert ^{m}f_{\lambda
}^{0}dxdv\right) ^{\frac{1}{N+m}}+(\left\Vert f_{\lambda }^{0}\right\Vert
_{L^{\infty }(\Omega \times \mathbb{R}^{N})}+1)\left\Vert \boldsymbol{w}\ast
\theta _{\lambda }\right\Vert _{L^{2}(0,T;L^{N+m}(\Omega )^{N})}\right]
^{N+m}.  \label{r12}
\end{equation}%
By employing the inequality (\ref{r10}) in conjunction with the estimate 
\begin{equation}
\left\Vert \boldsymbol{w}(t,\cdot )\ast \theta _{\lambda }\right\Vert
_{L^{N+m}(\Omega )^{N}}\leq \left\Vert \boldsymbol{w}(t,\cdot )\right\Vert
_{L^{N+m}(\Omega )^{N}}\left\Vert \theta _{\lambda }\right\Vert
_{L^{1}(\Omega )}\leq \left\Vert \boldsymbol{w}(t,\cdot )\right\Vert
_{L^{N+m}(\Omega )^{N}},  \label{r13}
\end{equation}%
and noting that $\left\Vert f_{\lambda }^{0}\right\Vert _{L^{\infty }(\Omega
\times \mathbb{R}^{N})}\leq \left\Vert f^{0}\right\Vert _{L^{\infty }(\Omega
\times \mathbb{R}^{N})}$, we arrive at the inequality 
\begin{equation*}
\int_{\Omega \times \mathbb{R}^{N}}\left\vert v\right\vert
^{m}f_{\varepsilon ,\lambda }dxdv\leq C(N,m,p,T)\left[ \left\Vert
f^{0}\right\Vert _{L^{p}(\Omega \times \mathbb{R}^{N})}^{\frac{1}{N+m}%
}+(\left\Vert f^{0}\right\Vert _{L^{\infty }(\Omega \times \mathbb{R}%
^{N})}+1)\left\Vert \boldsymbol{w}\right\Vert _{L^{2}(0,T;L^{N+m}(\Omega
)^{N})}\right] ^{N+m}
\end{equation*}%
for any $m\geq 1$ satisfying $N+m\leq 6$. By employing the Sobolev embedding 
$H_{0}^{1}(\Omega )\hookrightarrow L^{N+m}(\Omega )$, the latter inequality
leads to 
\begin{equation}
\int_{\Omega \times \mathbb{R}^{N}}\left\vert v\right\vert
^{m}f_{\varepsilon ,\lambda }dxdv\leq C\left[ \left\Vert f^{0}\right\Vert
_{L^{p}(\Omega \times \mathbb{R}^{N})}^{\frac{1}{N+m}}+(\left\Vert
f^{0}\right\Vert _{L^{\infty }(\Omega \times \mathbb{R}^{N})}+1)\left\Vert 
\boldsymbol{w}\right\Vert _{L^{2}(0,T;V)}\right] ^{N+m}  \label{r14}
\end{equation}%
for any $m\geq 1$ satisfying $N+m\leq 6$, where $C=C(N,m,p,T,\Omega )$.

\begin{remark}
\label{r2.0}\emph{We shall discover in the sequel that the estimate (\ref%
{r14}) makes assumptions on higher-order }$v$\emph{-moments of the initial
distribution redundant. It is opportune to emphasize that besides Lemma 2.1
of Hamdache \cite{Hamdache}, both regularization and truncation have played
a fundamental role in deriving the inequality (\ref{r14}).}
\end{remark}

Next, we provide estimates which show among other things that the force
field 
\begin{equation}
\boldsymbol{F}_{\varepsilon ,\lambda }=\boldsymbol{G}_{\varepsilon ,\lambda
}+\boldsymbol{H}_{\varepsilon ,\lambda },  \label{rn14}
\end{equation}%
where 
\begin{equation}
\boldsymbol{G}_{\varepsilon ,\lambda }(t,x)=-\int_{\mathbb{R}^{N}}(%
\boldsymbol{w}\ast \theta _{\lambda }-v)\gamma _{\lambda }(v)f_{\varepsilon
,\lambda }dv  \label{r14'}
\end{equation}%
and%
\begin{equation}
\boldsymbol{H}_{\varepsilon ,\lambda }(t,x)=\Div\left(
\int_{0}^{t}A_{1}^{\varepsilon }(t-\tau ,x)\nabla \boldsymbol{w}(\tau
,x)d\tau \right) ,  \label{r14''}
\end{equation}%
belongs to $L^{2}(0,T;H^{-1}(\Omega )^{N})$. So, let $\boldsymbol{\Phi }\in 
\mathcal{C}_{0}^{\infty }(\Omega )^{N}$. For almost all $t\in \lbrack 0,T]$
we have 
\begin{eqnarray}
\left\vert \left\langle \boldsymbol{G}_{\varepsilon ,\lambda }(t,x),%
\boldsymbol{\Phi }\right\rangle \right\vert &\leq &\int_{\Omega \times
\{\left\vert v\right\vert \leq 2\}}(1+\left\vert \boldsymbol{w}\ast \theta
_{\lambda }\right\vert )f_{\varepsilon ,\lambda }\left\vert \boldsymbol{\Phi 
}\right\vert dxdv  \notag \\
&\leq &C(N)\left( 1+\left\Vert \boldsymbol{w}(t,\cdot )\right\Vert
_{L^{2}(\Omega )^{N}}\right) \left\Vert f_{\varepsilon ,\lambda }(t,\cdot
)\right\Vert _{L^{\infty }(\Omega \times \mathbb{R}^{N})}\left\Vert 
\boldsymbol{\Phi }(t,\cdot )\right\Vert _{L^{2}(\Omega )^{N}}.  \label{r15}
\end{eqnarray}%
Thus, for almost all $t\in \lbrack 0,T]$ we have 
\begin{equation}
\left\Vert \boldsymbol{G}_{\varepsilon ,\lambda }(t,\cdot )\right\Vert
_{L^{2}(\Omega )^{N}}\leq C(N)\left( 1+\left\Vert \boldsymbol{w}(t,\cdot
)\right\Vert _{L^{2}(\Omega )^{N}}\right) \left\Vert f_{\varepsilon ,\lambda
}(t,\cdot )\right\Vert _{L^{\infty }(\Omega \times \mathbb{R}^{N})}.
\label{r16}
\end{equation}%
Since $f_{\lambda }^{0}\in L^{\infty }(\Omega \times \mathbb{R}^{N})$, by
the maximum principle applied to the transport equation, we have 
\begin{equation}
\left\Vert f_{\varepsilon ,\lambda }\right\Vert _{L^{\infty }(0,T;L^{\infty
}(\Omega \times \mathbb{R}^{N}))}\leq C(N,T)\left\Vert f_{\lambda
}^{0}\right\Vert _{L^{\infty }(\Omega \times \mathbb{R}^{N})}\leq
C(N,T)\left\Vert f^{0}\right\Vert _{L^{\infty }(\Omega \times \mathbb{R}%
^{N})}.  \label{r17}
\end{equation}%
Therefore, using Eq. (\ref{r17}) in Eq. (\ref{r16}), we obtain 
\begin{equation}
\left\Vert \boldsymbol{G}_{\varepsilon ,\lambda }(t,\cdot )\right\Vert
_{L^{2}(\Omega )^{N}}\leq C(N,T)\left\Vert f^{0}\right\Vert _{L^{\infty
}(\Omega \times \mathbb{R}^{N})}\left( 1+\left\Vert \boldsymbol{w}(t,\cdot
)\right\Vert _{L^{2}(\Omega )^{N}}\right)  \label{r18}
\end{equation}%
for almost all $t\in \lbrack 0,T]$. Thus, since $\boldsymbol{w}\in
L^{2}(0,T;L^{2}(\Omega )^{N})$, so is $\boldsymbol{G}_{\varepsilon ,\lambda
} $.

Now, we turn our attention to $\boldsymbol{H}_{\varepsilon ,\lambda }$: 
\begin{eqnarray}
\left\vert \left\langle \boldsymbol{H}_{\varepsilon ,\lambda }(t,.),%
\boldsymbol{\Phi }\right\rangle \right\vert &\leq &N^{3}\left\Vert
A_{1}\right\Vert _{L^{\infty }(Q\times \mathbb{R}^{N+1})^{N^{2}}}%
\int_{0}^{t}\int_{\Omega }\left\vert \nabla \boldsymbol{w}\right\vert
\left\vert \nabla \boldsymbol{\Phi }\right\vert dxd\tau  \notag \\
&\leq &N^{3}\left\Vert A_{1}\right\Vert _{L^{\infty }(Q\times \mathbb{R}%
^{N+1})^{N^{2}}}T\left\Vert \nabla \boldsymbol{w}\right\Vert
_{L^{2}(0,T;L^{2}(\Omega )^{N^{2}})}\left\Vert \nabla \boldsymbol{\Phi }%
\right\Vert _{L^{2}(\Omega )^{N^{2}}}.  \label{r19}
\end{eqnarray}%
Therefore, $\boldsymbol{H}_{\varepsilon ,\lambda }\in
L^{2}(0,T;H^{-1}(\Omega )^{N})$ and 
\begin{equation}
\left\Vert \boldsymbol{H}_{\varepsilon ,\lambda }\right\Vert
_{L^{2}(0,T;H^{-1}(\Omega )^{N})}\leq N^{3}\left\Vert A_{1}\right\Vert
_{L^{\infty }(Q\times \mathbb{R}^{N+1})^{N^{2}}}T^{3/2}\left\Vert \nabla 
\boldsymbol{w}\right\Vert _{L^{2}(0,T;L^{2}(\Omega )^{N^{2}})}.  \label{r20}
\end{equation}

\subsection{Existence of $(\boldsymbol{u}_{\protect\varepsilon ,\protect%
\lambda },p_{\protect\varepsilon ,\protect\lambda })$ and further estimates}

We look for $\boldsymbol{u}_{\varepsilon ,\lambda }\in L^{2}(0,T;V)$ such
that $\partial \boldsymbol{u}_{\varepsilon }/\partial t\in
L^{2}(0,T;V^{\prime })$ and, for almost all $t\in \lbrack 0,T]$ and all $%
\boldsymbol{\Phi }\in V$, 
\begin{eqnarray}
\left\langle {\frac{d\boldsymbol{u}_{\varepsilon ,\lambda }}{dt}},%
\boldsymbol{\Phi }\right\rangle +a_{\varepsilon }(t;\boldsymbol{u}%
_{\varepsilon ,\lambda },\boldsymbol{\Phi }) &=&\langle \boldsymbol{F}%
_{\varepsilon ,\lambda }(t),\boldsymbol{\Phi }\rangle ,\quad  \label{r21} \\
\boldsymbol{u}_{\varepsilon ,\lambda }(0) &=&\boldsymbol{u}_{0}\in H,
\label{r22}
\end{eqnarray}%
where 
\begin{equation}
a_{\varepsilon }(t;\boldsymbol{u},\boldsymbol{\Phi })=\int_{\Omega
}A_{0}^{\varepsilon }(t,x)\nabla \boldsymbol{u}(x)\cdot \nabla \boldsymbol{%
\Phi }(x)dx  \label{r22'}
\end{equation}%
and $\boldsymbol{F}_{\varepsilon ,\lambda }$ is defined in Eq. (\ref{rn14}).
Note that by standard arguments \textbf{\cite{Lions1}}, the problem (\ref%
{r21})-(\ref{r22}) makes sense. By direct computations using assumptions
made on the $A_{i}$s, one arrives at the following properties of $%
a_{\varepsilon }$.

\begin{itemize}
\item The function $t\mapsto a_{\varepsilon}(t; \boldsymbol{u}, \boldsymbol{%
\Phi})$ is mesurable for all $\boldsymbol{u}, \boldsymbol{\Phi} \in V$.

\item For almost every $t\in \lbrack 0,T]$ and for all $\boldsymbol{u},%
\boldsymbol{\Phi }\in V$, 
\begin{equation}
\left\vert a_{\varepsilon }(t;\boldsymbol{u},\boldsymbol{\Phi })\right\vert
\leq N^{3}\left\Vert A_{0}\right\Vert _{L^{\infty }(Q\times \mathbb{R}%
^{N+1})^{N^{2}}}\left\Vert \boldsymbol{u}\right\Vert _{V}\left\Vert 
\boldsymbol{\Phi }\right\Vert _{V}:=M\left\Vert \boldsymbol{u}\right\Vert
_{V}\left\Vert \boldsymbol{\Phi }\right\Vert _{V}.  \label{r23}
\end{equation}

\item For almost every $t\in \lbrack 0,T]$ and for all $\boldsymbol{v}\in V$%
, 
\begin{equation}
\left\vert a_{\varepsilon }(t;\boldsymbol{v},\boldsymbol{v})\right\vert \geq
\alpha \left\Vert \boldsymbol{v}\right\Vert _{V}^{2}.  \label{r24}
\end{equation}
\end{itemize}

Thus, by Lions' theorem \cite{Lions1}, there is a unique $\boldsymbol{u}%
_{\varepsilon ,\lambda }\in L^{2}(0,T;V)\cap \mathcal{C}(0,T;H)$ satisfying
Eqs (\ref{r21})-(\ref{r22}). Since $\boldsymbol{F}_{\varepsilon ,\lambda
}\in L^{2}(0,T;H^{-1}(\Omega )^{N})$ and $\boldsymbol{w}\in L^{2}(0,T,V)$,
it is a simple matter to check that $\frac{\partial \boldsymbol{u}%
_{\varepsilon ,\lambda }}{\partial t}-\Div\left( A_{0}^{\varepsilon }\nabla 
\boldsymbol{u}_{\varepsilon ,\lambda }\right) -\boldsymbol{F}_{\varepsilon
,\lambda }\in H^{-1}(\Omega )^{N}\subset \mathcal{D}^{\prime }(\Omega )^{N}$
($\mathcal{D}^{\prime }(\Omega )$ the usual space of distributions on $%
\Omega $) for almost all $t\in \lbrack 0,T]$. Thus, thanks to Eq. (\ref{r21}%
) and Propositions 1.1 and 1.2 of \cite{Temam}, there is a unique $%
p_{\varepsilon ,\lambda }(t)\in L^{2}(\Omega )/\mathbb{R}$ such that Eq. (%
\ref{r2}) holds in the sense of distributions and 
\begin{equation}
\left\Vert p_{\varepsilon ,\lambda }(t)\right\Vert _{L^{2}(\Omega )/\mathbb{R%
}}\leq C(\Omega )\left\Vert \nabla p_{\varepsilon ,\lambda }(t)\right\Vert
_{H^{-1}(\Omega )^{N}}  \label{r25}
\end{equation}%
for almost all $t\in \lbrack 0,T]$.

\subsection{Solvability of (\protect\ref{r1})-(\protect\ref{r5}) with $%
\boldsymbol{w}=\boldsymbol{u}_{\protect\varepsilon ,\protect\lambda }$}

Here, we prove via Schauder's fixed point theorem that by letting $%
\boldsymbol{w}=\boldsymbol{u}_{\varepsilon ,\lambda }$, the regularized
problem is still solvable. In order to do that, we consider the mapping 
\begin{equation*}
S:L^{2}(0,T;V)\rightarrow L^{2}(0,T;V) \mbox{ with } \boldsymbol{w}\mapsto 
\boldsymbol{u}_{\varepsilon ,\lambda },
\end{equation*}%
where $\boldsymbol{u}_{\varepsilon ,\lambda }$ is the unique solution of
Eqs. (\ref{r21})-(\ref{r22}). The mapping $S$ is well-defined because of the
previous step . We need to show that $S$ has a fixed point as asserted in
the next result.

\begin{proposition}
\label{p2.1}There exists a function $\boldsymbol{u}_{\varepsilon ,\lambda }$
in $L^{2}(0,T;V)$ such that $S\boldsymbol{u}_{\varepsilon ,\lambda }=%
\boldsymbol{u}_{\varepsilon ,\lambda }$.
\end{proposition}

\begin{proof}
The mapping $S$ is not linear. However, we can check that it is Lipschitz
continuous. Indeed, let $\boldsymbol{w}_{1},\boldsymbol{w}_{2}\in
L^{2}(0,T;V)$ and set $\boldsymbol{u}_{\varepsilon ,\lambda }^{i}=S%
\boldsymbol{w}_{i}$ ($i=1,2$), $\boldsymbol{w}=\boldsymbol{w}_{1}-%
\boldsymbol{w}_{2}$. Let us also denote by $p_{\varepsilon ,\lambda }^{i}$ ($%
i=1,2$) the associated pressures. Then, $\boldsymbol{u}_{\varepsilon
,\lambda }=\boldsymbol{u}_{\varepsilon ,\lambda }^{1}-\boldsymbol{u}%
_{\varepsilon ,\lambda }^{2}$ and $p_{\varepsilon ,\lambda }=p_{\varepsilon
,\lambda }^{1}-p_{\varepsilon ,\lambda }^{2}$ solve the following Stokes
system 
\begin{eqnarray*}
\frac{\partial \boldsymbol{u}_{\varepsilon ,\lambda }}{\partial t}-\Div%
\left( A_{0}^{\varepsilon }\nabla \boldsymbol{u}_{\varepsilon ,\lambda
}\right) +\nabla p_{\varepsilon ,\lambda } &=&\Div\left( A_{1}^{\varepsilon
}\ast \nabla \boldsymbol{w}\right) -\int_{\mathbb{R}^{N}}(\boldsymbol{w}\ast
\theta _{\lambda })\gamma _{\lambda }(v)f_{\varepsilon ,\lambda }\;dv%
\mbox{ in
}Q \\
\Div\boldsymbol{u}_{\varepsilon ,\lambda } &=&0\mbox{ in }Q \\
\boldsymbol{u}_{\varepsilon ,\lambda } &=&0\mbox{ on }(0,T)\times \partial
\Omega \\
\boldsymbol{u}_{\varepsilon ,\lambda }(0,x) &=&0\mbox{, }x\in \Omega .
\end{eqnarray*}%
Multiplying the leading equation above by $\boldsymbol{u}_{\varepsilon
,\lambda }$, we find after integrating over $\Omega $ that 
\begin{equation*}
\frac{1}{2}\frac{d}{dt}\int_{\Omega }\left\vert \boldsymbol{u}_{\varepsilon
,\lambda }\right\vert ^{2}dx+\int_{\Omega }(A_{0}^{\varepsilon }\nabla 
\boldsymbol{u}_{\varepsilon ,\lambda }+A_{1}^{\varepsilon }\ast \nabla 
\boldsymbol{w})\cdot \nabla \boldsymbol{u}_{\varepsilon ,\lambda
}dx+\int_{\Omega \times \mathbb{R}^{N}}\gamma _{\lambda }(v)f_{\varepsilon
,\lambda }(\boldsymbol{w}\ast \theta _{\lambda })\cdot \boldsymbol{u}%
_{\varepsilon ,\lambda }dvdx=0.
\end{equation*}%
Integrating the above equation over $(0,t)$ and using assumption (\textbf{A1}%
), we arrive at the inequality 
\begin{equation*}
\begin{array}{l}
\int_{\Omega }\left\vert \boldsymbol{u}_{\varepsilon ,\lambda }\right\vert
^{2}dx+2\alpha \int_{0}^{t}\int_{\Omega }\left\vert \nabla \boldsymbol{u}%
_{\varepsilon ,\lambda }\right\vert ^{2}dxd\tau \leq
-2\int_{0}^{t}\int_{\Omega }(A_{1}^{\varepsilon }\ast \nabla \boldsymbol{w}%
)\cdot \nabla \boldsymbol{u}_{\varepsilon ,\lambda }dxd\tau \\ 
\ \ \ -2\int_{0}^{t}\int_{\Omega \times \mathbb{R}^{N}}\gamma _{\lambda
}(v)f_{\varepsilon ,\lambda }(\boldsymbol{w}\ast \theta _{\lambda })\cdot 
\boldsymbol{u}_{\varepsilon ,\lambda }dvdxd\tau .%
\end{array}%
\end{equation*}%
Using Young's inequality yields 
\begin{equation*}
\begin{array}{l}
2\left\vert \int_{0}^{t}\int_{\Omega }(A_{1}^{\varepsilon }\ast \nabla 
\boldsymbol{w})\cdot \nabla \boldsymbol{u}_{\varepsilon ,\lambda }dxd\tau
\right\vert =2\left\vert \int_{0}^{t}\left( \int_{0}^{\tau }\left(
\int_{\Omega }A_{1}^{\varepsilon }(\tau -s)\nabla \boldsymbol{w}(s)\cdot
\nabla \boldsymbol{u}_{\varepsilon ,\lambda }(\tau )dx\right) ds\right)
d\tau \right\vert \\ 
\ \ \ \ \ \ \leq \int_{0}^{t}\left( \int_{0}^{\tau }\frac{\tau }{\alpha }%
\left\Vert A_{1}^{\varepsilon }(\tau -s)\nabla \boldsymbol{w}(s)\right\Vert
_{L^{2}(\Omega )}^{2}ds+\int_{0}^{\tau }\frac{\alpha }{\tau }\left\Vert
\nabla \boldsymbol{u}_{\varepsilon ,\lambda }(\tau )\right\Vert
_{L^{2}(\Omega )}^{2}ds\right) d\tau \\ 
\ \ \ \ \ \ \leq \int_{0}^{t}\left( \int_{0}^{\tau }\frac{c_{1}}{\alpha }%
\tau \left\Vert \nabla \boldsymbol{w}(s)\right\Vert _{L^{2}(\Omega
)}^{2}ds+\int_{0}^{\tau }\frac{\alpha }{\tau }\left\Vert \nabla \boldsymbol{u%
}_{\varepsilon ,\lambda }(\tau )\right\Vert _{L^{2}(\Omega )}^{2}ds\right)
d\tau \\ 
\ \ \ \ \ \ \ =\alpha \int_{0}^{t}\left\Vert \nabla \boldsymbol{u}%
_{\varepsilon ,\lambda }(\tau )\right\Vert _{L^{2}(\Omega )}^{2}d\tau
+\int_{0}^{t}\left( \frac{c_{1}}{\alpha }\tau \int_{0}^{\tau }\left\Vert
\nabla \boldsymbol{w}(s)\right\Vert _{L^{2}(\Omega )}^{2}ds\right) d\tau \\ 
\ \ \ \ \ \ \ \leq \alpha \int_{0}^{t}\left\Vert \nabla \boldsymbol{u}%
_{\varepsilon ,\lambda }(\tau )\right\Vert _{L^{2}(\Omega )}^{2}d\tau
+Ct\int_{0}^{t}\left( \int_{0}^{\tau }\left\Vert \nabla \boldsymbol{w}%
(s)\right\Vert _{L^{2}(\Omega )}^{2}ds\right) d\tau \mbox{ since }0\leq \tau
\leq t \\ 
\ \ \ \ \ \ \ \leq \alpha \int_{0}^{t}\left\Vert \nabla \boldsymbol{u}%
_{\varepsilon ,\lambda }(\tau )\right\Vert _{L^{2}(\Omega )}^{2}d\tau
+CT\int_{0}^{t}\left( \int_{0}^{\tau }\left\Vert \nabla \boldsymbol{w}%
(s)\right\Vert _{L^{2}(\Omega )}^{2}ds\right) d\tau .%
\end{array}%
\end{equation*}%
Also, it can be verified that 
\begin{eqnarray*}
2\left\vert \int_{0}^{t}\int_{\Omega \times \mathbb{R}^{N}}\gamma _{\lambda
}(v)f_{\varepsilon ,\lambda }(\boldsymbol{w}\ast \theta _{\lambda })\cdot 
\boldsymbol{u}_{\varepsilon ,\lambda }dvdxd\tau \right\vert &\leq
&C(N,T)\left\Vert f^{0}\right\Vert _{L^{\infty }(\Omega \times \mathbb{R}%
^{N})}\left\vert B(0,2)\right\vert \int_{0}^{t}\int_{\Omega }\left\vert 
\boldsymbol{w}\ast \theta _{\lambda }\right\vert \left\vert \boldsymbol{u}%
_{\varepsilon ,\lambda }\right\vert dxd\tau \\
&\leq &C\int_{0}^{t}\int_{\Omega }\left( \left\vert \boldsymbol{w}\ast
\theta _{\lambda }\right\vert ^{2}+\left\vert \boldsymbol{u}_{\varepsilon
,\lambda }\right\vert ^{2}\right) \;dxd\tau \\
&\leq &C\left\Vert \boldsymbol{w}\ast \theta _{\lambda }\right\Vert
_{L^{2}(Q)}^{2}+C\int_{0}^{t}\left\Vert \boldsymbol{u}_{\varepsilon ,\lambda
}\right\Vert _{L^{2}(\Omega )}^{2}d\tau \\
&\leq &C\left\Vert \boldsymbol{w}\right\Vert
_{L^{2}(Q)}^{2}+C\int_{0}^{t}\left\Vert \boldsymbol{u}_{\varepsilon ,\lambda
}\right\Vert _{L^{2}(\Omega )}^{2}\;d\tau ,
\end{eqnarray*}%
where $\left\vert B(0,2)\right\vert $ stands for the Lebesgue measure of $%
B(0,2)$. Thus, 
\begin{eqnarray*}
\int_{\Omega }\left\vert \boldsymbol{u}_{\varepsilon ,\lambda }\right\vert
^{2}dx+\alpha \int_{0}^{t}\left\Vert \nabla \boldsymbol{u}_{\varepsilon
,\lambda }\right\Vert _{L^{2}(\Omega )}^{2}d\tau &\leq &C\left\Vert 
\boldsymbol{w}\right\Vert _{L^{2}(Q)}^{2}+C\int_{0}^{t}\left( \int_{0}^{\tau
}\left\Vert \nabla \boldsymbol{w}(s)\right\Vert _{L^{2}(\Omega
)}^{2}ds\right) d\tau \\
&&+C\int_{0}^{t}\left\Vert \boldsymbol{u}_{\varepsilon ,\lambda }\right\Vert
_{L^{2}(\Omega )}^{2}d\tau
\end{eqnarray*}%
and 
\begin{eqnarray*}
C\left\Vert \boldsymbol{w}\right\Vert _{L^{2}(Q)}^{2}+C\int_{0}^{t}\left(
\int_{0}^{\tau }\left\Vert \nabla \boldsymbol{w}(s)\right\Vert
_{L^{2}(\Omega )}^{2}ds\right) d\tau &\leq &C\left\Vert \boldsymbol{w}%
\right\Vert _{L^{2}(Q)}^{2}+C\int_{0}^{T}\left( \int_{0}^{T}\left\Vert
\nabla \boldsymbol{w}(s)\right\Vert _{L^{2}(\Omega )}^{2}ds\right) d\tau \\
&\leq &C\left\Vert \boldsymbol{w}\right\Vert
_{L^{2}(Q)}^{2}+CT\int_{0}^{T}\left\Vert \nabla \boldsymbol{w}(s)\right\Vert
_{L^{2}(\Omega )}^{2}ds \\
&\leq &C\left\Vert \boldsymbol{w}\right\Vert _{L^{2}(0,T;V)}^{2}.
\end{eqnarray*}%
By conflating the previous estimates, we obtain the inequality 
\begin{equation*}
\int_{\Omega }\left\vert \boldsymbol{u}_{\varepsilon ,\lambda }\right\vert
^{2}dx+\alpha \int_{0}^{t}\left\Vert \nabla \boldsymbol{u}_{\varepsilon
,\lambda }\right\Vert _{L^{2}(\Omega )}^{2}d\tau \leq C\left\Vert 
\boldsymbol{w}\right\Vert _{L^{2}(0,T;V)}^{2}+C\int_{0}^{t}\left\Vert 
\boldsymbol{u}_{\varepsilon ,\lambda }\right\Vert _{L^{2}(\Omega )}^{2}d\tau
.
\end{equation*}%
Then, Gronwall's Lemma implies the inequality 
\begin{equation*}
\int_{0}^{t}\left\Vert \boldsymbol{u}_{\varepsilon ,\lambda }\right\Vert
_{L^{2}(\Omega )}^{2}d\tau \leq C\left\Vert \boldsymbol{w}\right\Vert
_{L^{2}(0,T;V)}^{2}\mbox{ for all }0\leq t\leq T,
\end{equation*}%
from which we infer that 
\begin{equation*}
\int_{0}^{T}\left\Vert \nabla \boldsymbol{u}_{\varepsilon ,\lambda
}\right\Vert _{L^{2}(\Omega )}^{2}d\tau \leq C\left\Vert \boldsymbol{w}%
\right\Vert _{L^{2}(0,T;V)}^{2}
\end{equation*}%
or equivalently, 
\begin{equation}
\left\Vert \boldsymbol{u}_{\varepsilon ,\lambda }\right\Vert
_{L^{2}(0,T;V)}\leq C\left\Vert \boldsymbol{w}\right\Vert _{L^{2}(0,T;V)}.
\label{6.1}
\end{equation}

Now, let $\varphi \in \mathcal{C}_{0}^{\infty }(0,T)\otimes \mathcal{V}$.
Then, 
\begin{eqnarray*}
\left\langle \frac{\partial \boldsymbol{u}_{\varepsilon ,\lambda }}{\partial
t},\varphi \right\rangle  &=&-\int_{Q}A_{0}^{\varepsilon }\nabla \boldsymbol{%
u}_{\varepsilon ,\lambda }\cdot \nabla \varphi
dxdt-\int_{Q}(A_{1}^{\varepsilon }\ast \nabla \boldsymbol{w})\cdot \nabla
\varphi dxdt \\
&&-\int_{Q}\left( \int_{\mathbb{R}^{N}}f_{\varepsilon ,\lambda }\gamma
_{\lambda }(v)dv\right) (\boldsymbol{w}\ast \theta _{\lambda })\cdot \varphi
dxdt
\end{eqnarray*}%
and 
\begin{eqnarray*}
\left\vert \left\langle \frac{\partial \boldsymbol{u}_{\varepsilon ,\lambda }%
}{\partial t},\varphi \right\rangle \right\vert  &\leq &C\left\Vert \nabla 
\boldsymbol{u}_{\varepsilon ,\lambda }\right\Vert _{L^{2}(Q)}\left\Vert
\nabla \varphi \right\Vert _{L^{2}(Q)}+\left\Vert A_{1}^{\varepsilon }\ast
\nabla \boldsymbol{w}\right\Vert _{L^{2}(Q)}\left\Vert \nabla \varphi
\right\Vert _{L^{2}(Q)}+C\left\Vert \boldsymbol{w}\ast \theta _{\lambda
}\right\Vert _{L^{2}(Q)}\left\Vert \varphi \right\Vert _{L^{2}(Q)} \\
&\leq &C\left\Vert \boldsymbol{w}\right\Vert _{L^{2}(0,T;V)}\left\Vert
\varphi \right\Vert _{L^{2}(0,T;V)}\mbox{ because of (\ref{6.1}),}
\end{eqnarray*}%
$C$ being a positive constant that does not depend on $\varphi $. Therefore,
it follows from the density of $\mathcal{C}_{0}^{\infty }(0,T)\otimes 
\mathcal{V}$ in $L^{2}(0,T;V)$ that $\partial \boldsymbol{u}_{\varepsilon
,\lambda }/\partial t\in L^{2}(0,T;V^{\prime })$ with 
\begin{equation}
\left\Vert \frac{\partial \boldsymbol{u}_{\varepsilon ,\lambda }}{\partial t}%
\right\Vert _{L^{2}(0,T;V^{\prime })}\leq C\left\Vert \boldsymbol{w}%
\right\Vert _{L^{2}(0,T;V)}.  \label{6.2}
\end{equation}

The inequality (\ref{6.1}) implies that $S$ sends continuously $L^{2}(0,T;V)$
into itself. Moreover (\ref{6.1}) and (\ref{6.2}) show that $S$ transforms
bounded sets in $L^{2}(0,T;V)$ into bounded sets in $W(0,T)=\{\boldsymbol{w}%
\in L^{2}(0,T;V):\partial \boldsymbol{w}/\partial t\in L^{2}(0,T;V^{\prime
})\}$ ($W(0,T)$ being endowed with the norm $\left\Vert \boldsymbol{w}%
\right\Vert _{W(0,T)}=[\left\Vert \boldsymbol{w}\right\Vert
_{L^{2}(0,T;V)}^{2}+\left\Vert \partial \boldsymbol{w}/\partial t\right\Vert
_{L^{2}(0,T;V^{\prime })}^{2}]^{1/2}$ which makes it a Hilbert space).
Furthermore, the range of $S$ is contained in $W(0,T)$ which is compact in $%
L^{2}(0,T;H)$ because of the Aubin-Lions lemma. Thus, the range of $S$ is
relatively compact in $L^{2}(0,T;H)$ and hence in $L^{2}(0,T;V)$ since the
latter space is closed in the former. Hence, by Schauder's fixed point
theorem, $S$ admits a fixed point.
\end{proof}

We have just proved the following result.

\begin{proposition}
\label{p2.2}For any fixed $\varepsilon >0$ and $\lambda >0$, the problem 
\emph{(\ref{r1'})-(\ref{r5'})} below%
\begin{equation}
\frac{\partial f_{\varepsilon ,\lambda }}{\partial t}+\varepsilon v\cdot
\nabla f_{\varepsilon ,\lambda }+\Div_{v}\left( (\boldsymbol{w}\ast \theta
_{\lambda }-v)f_{\varepsilon ,\lambda }\right) =0\mbox{ in }Q\times \mathbb{R%
}^{N}  \label{r1'}
\end{equation}%
\begin{equation}
\frac{\partial \boldsymbol{u}_{\varepsilon ,\lambda }}{\partial t}-\Div%
\left( A_{0}^{\varepsilon }\nabla \boldsymbol{u}_{\varepsilon ,\lambda
}+\int_{0}^{t}A_{1}^{\varepsilon }(t-\tau ,x)\nabla \boldsymbol{u}%
_{\varepsilon ,\lambda }(\tau ,x)d\tau \right) +\nabla p_{\varepsilon
,\lambda }=-\int_{\mathbb{R}^{N}}(\boldsymbol{u}_{\varepsilon ,\lambda }\ast
\theta _{\lambda }-v)\gamma _{\lambda }(v)f_{\varepsilon ,\lambda }dv%
\mbox{
in }Q,  \label{r2'}
\end{equation}%
\begin{equation}
\Div\boldsymbol{u}_{\varepsilon ,\lambda }=0\mbox{ in }Q,  \label{r3'}
\end{equation}%
\begin{equation}
\boldsymbol{u}_{\varepsilon ,\lambda }(0,x)=u^{0}(x),\ f_{\varepsilon
,\lambda }(0,x,v):=f_{\lambda }^{0}(x,v)=\gamma _{\lambda }(v)(f^{0}\ast
\Theta _{\lambda })(x,v),\ x\in \Omega ,v\in \mathbb{R}^{N},  \label{r4'}
\end{equation}%
\begin{equation}
\boldsymbol{u}_{\varepsilon ,\lambda }=0\mbox{ on }\partial \Omega 
\mbox{ and
}f_{\varepsilon ,\lambda }(t,x,v)=f_{\varepsilon ,\lambda }(t,x,v^{\ast })%
\mbox{ for }x\in \partial \Omega \mbox{ with }v\cdot \nu (x)<0  \label{r5'}
\end{equation}%
admits a unique solution $(\boldsymbol{u}_{\varepsilon ,\lambda
},f_{\varepsilon ,\lambda },p_{\varepsilon ,\lambda })$ such that $%
\boldsymbol{u}_{\varepsilon ,\lambda }\in L^{2}(0,T;V)$ with $\partial 
\boldsymbol{u}_{\varepsilon ,\lambda }/\partial t\in L^{2}(0,T;V^{\prime })$%
, $f_{\varepsilon ,\lambda }\in \mathcal{C}^{1}(Q\times \mathbb{R}^{N})$ and 
$p_{\varepsilon ,\lambda }\in L^{\infty }(0,T;L^{2}(\Omega )/\mathbb{R})$.
\end{proposition}

The following uniform estimates hold true.

\begin{lemma}
\label{l2.1}Let $(\boldsymbol{u}_{\varepsilon ,\lambda },f_{\varepsilon
,\lambda },p_{\varepsilon ,\lambda })$ be the solution to \emph{(\ref{r1'})-(%
\ref{r5'})}. Then, 
\begin{equation}
\int_{\Omega \times \mathbb{R}^{N}}(1+\left\vert v\right\vert
^{2})f_{\varepsilon ,\lambda }dxdv+\int_{\Omega }\left\vert \boldsymbol{u}%
_{\varepsilon ,\lambda }\right\vert ^{2}dx+2\int_{0}^{t}\int_{\Omega \times 
\mathbb{R}^{N}}f_{\varepsilon ,\lambda }\left\vert \boldsymbol{u}%
_{\varepsilon ,\lambda }\ast \theta _{\lambda }-v\right\vert ^{2}dxdvd\tau
+\int_{0}^{t}\left\Vert \nabla \boldsymbol{u}_{\varepsilon ,\lambda }(\tau
)\right\Vert _{L^{2}(\Omega )}^{2}d\tau \leq C  \label{2.3}
\end{equation}%
for any $0\leq t\leq T$,\ $\varepsilon >0$ and $\lambda >0$, where $C>0$ is
independent of both $\lambda $ and $\varepsilon $. Moreover if $f^{0}\in
L^{p}(\Omega \times \mathbb{R}^{N})$, \emph{(}$1\leq p\leq \infty $\emph{)},
then 
\begin{equation}
\left\Vert f_{\varepsilon ,\lambda }\right\Vert _{L^{\infty
}(0,T;L^{p}(\Omega \times \mathbb{R}^{N}))}\leq \exp (NT)\left\Vert
f^{0}\right\Vert _{L^{p}(\Omega \times \mathbb{R}^{N})}\mbox{ for any }%
\lambda ,\varepsilon >0.  \label{2.4}
\end{equation}%
It also holds that 
\begin{equation}
\left\Vert \frac{\partial \boldsymbol{u}_{\varepsilon ,\lambda }}{\partial t}%
\right\Vert _{L^{2}(0,T;H^{-1}(\Omega )^{N})}\leq C  \label{2.4'}
\end{equation}%
and%
\begin{equation}
\sup_{\lambda ,\varepsilon >0}\left\Vert p_{\varepsilon ,\lambda
}\right\Vert _{L^{2}(0,T;L^{2}(\Omega ))}\leq C.  \label{2.5}
\end{equation}
\end{lemma}

\begin{proof}
The inequality (\ref{2.4}) has already been obtained (see Eq. (\ref{r9})).
Let us now check (\ref{2.3}). We multiply (\ref{1.1}) by $\frac{1}{2}%
\left\vert v\right\vert ^{2}$ and (\ref{1.2}) by $\boldsymbol{u}%
_{\varepsilon ,\lambda }$, and we get 
\begin{equation}
\frac{1}{2}\frac{d}{dt}\left[ \int_{\Omega \times \mathbb{R}^{N}}\left\vert
v\right\vert ^{2}f_{\varepsilon ,\lambda }dxdv+\int_{\Omega }\left\vert 
\boldsymbol{u}_{\varepsilon ,\lambda }\right\vert ^{2}dx\right]
+\int_{\Omega }(A_{0}^{\varepsilon }\nabla \boldsymbol{u}_{\varepsilon
,\lambda }+A_{1}^{\varepsilon }\ast \nabla \boldsymbol{u}_{\varepsilon
,\lambda })\cdot \nabla \boldsymbol{u}_{\varepsilon ,\lambda
}dx+\int_{\Omega }E_{\varepsilon ,\lambda }(t,x)dx=0,  \label{2.7}
\end{equation}%
where we set 
\begin{equation*}
E_{\varepsilon ,\lambda }(t,x)=\int_{\mathbb{R}^{N}}f_{\varepsilon ,\lambda
}(\boldsymbol{u}_{\varepsilon ,\lambda }\ast \theta _{\lambda }-v)\cdot 
\boldsymbol{u}_{\varepsilon ,\lambda }dv+\frac{\varepsilon }{2}\int_{\mathbb{%
R}^{N}}(v\cdot \nabla f_{\varepsilon ,\lambda })\left\vert v\right\vert
^{2}dv+\frac{1}{2}\int_{\mathbb{R}^{N}}\left\vert v\right\vert ^{2}\Div_{v}((%
\boldsymbol{u}_{\varepsilon ,\lambda }\ast \theta _{\lambda
}-v)f_{\varepsilon ,\lambda })dv.
\end{equation*}%
But 
\begin{eqnarray*}
\int_{\Omega }\int_{\mathbb{R}^{N}}(v\cdot \nabla f_{\varepsilon ,\lambda
})\left\vert v\right\vert ^{2}dvdx &=&\int_{\mathbb{R}^{N}}\int_{\partial
\Omega }f_{\varepsilon ,\lambda }\left\vert v\right\vert ^{2}(v\cdot \nu
)d\sigma dv \\
&=&\int_{\{v\cdot \nu >0\}}f_{\varepsilon ,\lambda }\left\vert v\right\vert
^{2}(v\cdot \nu )d\sigma dv+\int_{\{v\cdot \nu <0\}}f_{\varepsilon ,\lambda
}\left\vert v\right\vert ^{2}(v\cdot \nu )d\sigma dv.
\end{eqnarray*}%
Since $v^{\ast }=v-2(v\cdot \nu )\nu $, it holds that $v^{\ast }\cdot \nu
=-v\cdot \nu $, $\left\vert v^{\ast }\right\vert ^{2}=\left\vert
v\right\vert ^{2}$ and $dv^{\ast }=dv$. Thus, because of the reflection
condition (\ref{r5'}) on $f_{\varepsilon ,\lambda }$, we have 
\begin{equation*}
\int_{\{v\cdot \nu <0\}}f_{\varepsilon ,\lambda }\left\vert v\right\vert
^{2}(v\cdot \nu )d\sigma dv=-\int_{\{v^{\ast }\cdot \nu >0\}}f_{\varepsilon
,\lambda }(t,x,v^{\ast })\left\vert v^{\ast }\right\vert ^{2}(v^{\ast }\cdot
\nu )d\sigma dv^{\ast },
\end{equation*}%
so that 
\begin{equation*}
\int_{\mathbb{R}^{N}}\int_{\partial \Omega }f_{\varepsilon ,\lambda
}\left\vert v\right\vert ^{2}(v\cdot \nu )d\sigma dv=0.
\end{equation*}%
Also, the following identity holds 
\begin{equation*}
\int_{\mathbb{R}^{N}}\left\vert v\right\vert ^{2}\Div_{v}((\boldsymbol{u}%
_{\varepsilon ,\lambda }\ast \theta _{\lambda }-v)f_{\varepsilon ,\lambda
})dv=-2\int_{\mathbb{R}^{N}}f_{\varepsilon ,\lambda }(\boldsymbol{u}%
_{\varepsilon ,\lambda }\ast \theta _{\lambda }-v)\cdot vdv.
\end{equation*}%
It therefore follows that 
\begin{equation*}
\int_{\Omega }E_{\varepsilon ,\lambda }(t,x)dx=\int_{\Omega \times \mathbb{R}%
^{N}}f_{\varepsilon ,\lambda }\left\vert \boldsymbol{u}_{\varepsilon
,\lambda }\ast \theta _{\lambda }-v\right\vert ^{2}dxdv.
\end{equation*}%
Integrating Eq.(\ref{2.7}) over $(0,t)$ and using the assumption (\textbf{A1}%
), we are lead to 
\begin{equation*}
\begin{array}{l}
\int_{\Omega \times \mathbb{R}^{N}}\left\vert v\right\vert
^{2}f_{\varepsilon ,\lambda }dxdv+2\int_{0}^{t}\int_{\Omega \times \mathbb{R}%
^{N}}f_{\varepsilon ,\lambda }\left\vert \boldsymbol{u}_{\varepsilon
,\lambda }\ast \theta _{\lambda }-v\right\vert ^{2}dxdvd\tau +\int_{\Omega
}\left\vert \boldsymbol{u}_{\varepsilon ,\lambda }\right\vert ^{2}dx+2\alpha
\int_{0}^{t}\int_{\Omega }\left\vert \nabla \boldsymbol{u}_{\varepsilon
,\lambda }\right\vert ^{2}dxd\tau \\ 
\ \ \ \leq -2\int_{0}^{t}\int_{\Omega }(A_{1}^{\varepsilon }\ast \nabla 
\boldsymbol{u}_{\varepsilon ,\lambda })\cdot \nabla \boldsymbol{u}%
_{\varepsilon ,\lambda }dxd\tau +\int_{\Omega \times \mathbb{R}%
^{N}}\left\vert v\right\vert ^{2}f_{\lambda }^{0}dxdv+\int_{\Omega
}\left\vert \boldsymbol{u}^{0}\right\vert ^{2}dx.%
\end{array}%
\end{equation*}%
Now, using Young's inequality, we infer that 
\begin{equation*}
\begin{array}{l}
2\int_{0}^{t}\int_{\Omega }(A_{1}^{\varepsilon }\ast \nabla \boldsymbol{u}%
_{\varepsilon ,\lambda })\cdot \nabla \boldsymbol{u}_{\varepsilon ,\lambda
}dxd\tau =2\int_{0}^{t}\left( \int_{0}^{\tau }\left( \int_{\Omega
}A_{1}^{\varepsilon }(\tau -s)\nabla \boldsymbol{u}_{\varepsilon ,\lambda
}(s)\cdot \nabla \boldsymbol{u}_{\varepsilon ,\lambda }(\tau )dx)\right)
ds\right) d\tau \\ 
\ \ \ \ \ \ \leq \int_{0}^{t}\left( \int_{0}^{\tau }\frac{\tau }{\alpha }%
\left\Vert A_{1}^{\varepsilon }(\tau -s)\nabla \boldsymbol{u}_{\varepsilon
,\lambda }(s)\right\Vert _{L^{2}(\Omega )}^{2}ds+\int_{0}^{\tau }\frac{%
\alpha }{\tau }\left\Vert \nabla \boldsymbol{u}_{\varepsilon ,\lambda }(\tau
)\right\Vert _{L^{2}(\Omega )}^{2}ds\right) d\tau \\ 
\ \ \ \ \ \ \leq \int_{0}^{t}\left( \int_{0}^{\tau }\frac{c_{1}}{\alpha }%
\tau \left\Vert \nabla \boldsymbol{u}_{\varepsilon ,\lambda }(s)\right\Vert
_{L^{2}(\Omega )}^{2}ds+\int_{0}^{\tau }\frac{\alpha }{\tau }\left\Vert
\nabla \boldsymbol{u}_{\varepsilon ,\lambda }(\tau )\right\Vert
_{L^{2}(\Omega )}^{2}ds\right) d\tau \\ 
\ \ \ \ \ \ \ =\alpha \int_{0}^{t}\left\Vert \nabla \boldsymbol{u}%
_{\varepsilon ,\lambda }(\tau )\right\Vert _{L^{2}(\Omega )}^{2}d\tau
+\int_{0}^{t}\left( \frac{c_{1}}{\alpha }\tau \int_{0}^{\tau }\left\Vert
\nabla \boldsymbol{u}_{\varepsilon ,\lambda }(s)\right\Vert _{L^{2}(\Omega
)}^{2}ds\right) d\tau \\ 
\ \ \ \ \ \ \ \leq \alpha \int_{0}^{t}\left\Vert \nabla \boldsymbol{u}%
_{\varepsilon ,\lambda }(\tau )\right\Vert _{L^{2}(\Omega )}^{2}d\tau +\frac{%
c_{1}}{\alpha }t\int_{0}^{t}\left( \int_{0}^{\tau }\left\Vert \nabla 
\boldsymbol{u}_{\varepsilon ,\lambda }(s)\right\Vert _{L^{2}(\Omega
)}^{2}ds\right) d\tau \mbox{ since }0\leq \tau \leq t,%
\end{array}%
\end{equation*}%
where $c_{1}=\sup_{(t,x)\overline{Q}}\left\Vert A_{1}(t,x,\cdot ,\cdot
)\right\Vert _{L^{\infty }(\mathbb{R}_{y,\tau }^{N+1})^{N^{2}}}^{2}<\infty $%
. Thus 
\begin{equation}
\begin{array}{l}
\int_{\Omega \times \mathbb{R}^{N}}\left\vert v\right\vert
^{2}f_{\varepsilon ,\lambda }dxdv+2\int_{0}^{t}\int_{\Omega \times \mathbb{R}%
^{N}}f_{\varepsilon ,\lambda }\left\vert \boldsymbol{u}_{\varepsilon
,\lambda }\ast \theta _{\lambda }-v\right\vert ^{2}\;dxdvd\tau +\int_{\Omega
}\left\vert \boldsymbol{u}_{\varepsilon ,\lambda }\right\vert ^{2}dx+\alpha
\int_{0}^{t}\left\Vert \nabla \boldsymbol{u}_{\varepsilon ,\lambda
}\right\Vert _{L^{2}(\Omega )}^{2}d\tau \\ 
\ \ \ \leq \int_{\Omega \times \mathbb{R}^{N}}\left\vert v\right\vert
^{2}f_{\lambda }^{0}dxdv+\int_{\Omega }\left\vert \boldsymbol{u}%
^{0}\right\vert ^{2}dx+\frac{c_{1}}{\alpha }t\int_{0}^{t}\left(
\int_{0}^{\tau }\left\Vert \nabla \boldsymbol{u}_{\varepsilon ,\lambda
}(s)\right\Vert _{L^{2}(\Omega )}^{2}ds\right) d\tau .%
\end{array}
\label{2.8}
\end{equation}%
We infer from Eq.(\ref{2.8}) that 
\begin{equation*}
\alpha \int_{0}^{t}\left\Vert \nabla \boldsymbol{u}_{\varepsilon ,\lambda
}\right\Vert _{L^{2}(\Omega )}^{2}d\tau \leq c_{2}+\frac{c_{1}}{\alpha }%
t\int_{0}^{t}\left( \int_{0}^{\tau }\left\Vert \nabla \boldsymbol{u}%
_{\varepsilon ,\lambda }(s)\right\Vert _{L^{2}(\Omega )}^{2}ds\right) d\tau
\end{equation*}%
where 
\begin{equation*}
c_{2}=\int_{\Omega \times \mathbb{R}^{N}}\left\vert v\right\vert
^{2}f_{\lambda }^{0}dxdv+\int_{\Omega }\left\vert \boldsymbol{u}%
^{0}\right\vert ^{2}dx\leq C(N,p)\left\Vert f^{0}\right\Vert _{L^{p}(\Omega
\times \mathbb{R}^{N})}+\int_{\Omega }\left\vert \boldsymbol{u}%
^{0}\right\vert ^{2}dx<\infty .
\end{equation*}%
It readily follows from Gronwall's inequality that 
\begin{eqnarray*}
\int_{0}^{t}\left\Vert \nabla \boldsymbol{u}_{\varepsilon ,\lambda
}\right\Vert _{L^{2}(\Omega )}^{2}d\tau &\leq &\exp \left( \int_{0}^{t}\frac{%
c_{1}\tau }{\alpha ^{2}}d\tau \right) \left[ \int_{0}^{t}\frac{c_{2}}{\alpha 
}\exp \left( -\int_{0}^{s}\frac{c_{1}\tau }{\alpha ^{2}}d\tau \right) ds%
\right] \\
&=&\exp \left( \frac{c_{1}t^{2}}{2\alpha ^{2}}\right) \left[ \frac{c_{2}}{%
\alpha }\int_{0}^{t}\exp \left( -\frac{c_{1}}{2\alpha ^{2}}s^{2}\right) ds%
\right] \\
&\leq &\frac{c_{2}}{\alpha }\exp \left( \frac{c_{1}t^{2}}{2\alpha ^{2}}%
\right) \int_{0}^{\infty }\exp \left( -\frac{c_{1}}{2\alpha ^{2}}%
s^{2}\right) ds \\
&=&\frac{c_{2}}{2}\sqrt{\frac{2\pi }{c_{1}}}\exp \left( \frac{c_{1}t^{2}}{%
2\alpha ^{2}}\right) \mbox{ for all }0\leq t\leq T.
\end{eqnarray*}%
Setting 
\begin{equation*}
c_{3}=\frac{c_{2}}{2}\sqrt{\frac{2\pi }{c_{1}}}\exp \left( \frac{c_{1}T^{2}}{%
2\alpha ^{2}}\right) ,
\end{equation*}%
we get 
\begin{equation*}
\int_{0}^{t}\left\Vert \nabla \boldsymbol{u}_{\varepsilon ,\lambda
}\right\Vert _{L^{2}(\Omega )}^{2}d\tau \leq c_{3},
\end{equation*}%
and the above inequality entails 
\begin{equation*}
\int_{\Omega \times \mathbb{R}^{N}}\left\vert v\right\vert
^{2}f_{\varepsilon ,\lambda }dxdv+\int_{\Omega }\left\vert \boldsymbol{u}%
_{\varepsilon ,\lambda }\right\vert ^{2}dx+2\int_{0}^{t}\int_{\Omega \times 
\mathbb{R}^{N}}f_{\varepsilon ,\lambda }\left\vert \boldsymbol{u}%
_{\varepsilon ,\lambda }\ast \theta _{\lambda }-v\right\vert ^{2}dxdvd\tau
\leq c_{2}+c_{3}T.
\end{equation*}%
We deduce Eq. (\ref{2.3}) by letting $C=c_{2}+c_{3}T$.

The uniform estimate (\ref{2.4'}) is obtained as Eq. (\ref{6.2}), and Eq. (%
\ref{2.5}) follows in a trivial manner (see e.g. \cite{JMS}). This concludes
the proof.
\end{proof}

\subsection{Passing to the limit $\protect\lambda \rightarrow 0$}

\label{sol} We wish to pass to the limit as $\lambda \rightarrow 0$ in the
sequence of solutions $(\boldsymbol{u}_{\varepsilon ,\lambda
},f_{\varepsilon ,\lambda })$ in order to prove the existence of the
solution to our initial problem (\ref{1.1})-(\ref{1.5}). Owing to Lemma \ref%
{l2.1}, we have 
\begin{equation*}
\left\Vert f_{\varepsilon ,\lambda }\right\Vert _{L^{\infty
}(0,T;L^{p}(\Omega \times \mathbb{R}^{N}))}\leq C\mbox{ for all }1\leq p\leq
\infty ,
\end{equation*}%
\begin{equation*}
\left\Vert \boldsymbol{u}_{\varepsilon ,\lambda }\right\Vert _{L^{\infty
}(0,T;L^{2}(\Omega )^{N})}\leq C,\ \left\Vert \nabla \boldsymbol{u}%
_{\varepsilon ,\lambda }\right\Vert _{L^{2}(Q)}\leq C\mbox{ and }\
\left\Vert \frac{\partial \boldsymbol{u}_{\varepsilon ,\lambda }}{\partial t}%
\right\Vert _{L^{2}(0,T;H^{-1}(\Omega )^{N})}\leq C.
\end{equation*}%
Using the above uniform estimates (in $\lambda $), we deduce that, given an
ordinary sequence $\lambda =(\lambda _{n})_{n}$ (with $0<\lambda _{n}\leq 1$%
, $\lambda _{n}\rightarrow 0$ when $n\rightarrow \infty $, which we denote
by $\lambda \rightarrow 0$), there exist a subsequence of $\lambda $ (still
denoted by $\lambda $), functions $\boldsymbol{u}_{\varepsilon }\in
L^{2}(0,T;V)\cap L^{\infty }(0,T;H)$, $f_{\varepsilon }\in L^{\infty
}(0,T;L^{p}(\Omega \times \mathbb{R}^{N}))$ and $p_{\varepsilon }\in
L^{2}(0,T;L^{2}(\Omega )/\mathbb{R})$ such that, as $\lambda \rightarrow 0$, 
\begin{equation}
f_{\varepsilon ,\lambda }\rightarrow f_{\varepsilon }\mbox{ in }L^{\infty
}(0,T;L^{p}(\Omega \times \mathbb{R}^{N}))\mbox{-weak }\ast ,  \label{6.3}
\end{equation}%
\begin{equation}
\boldsymbol{u}_{\varepsilon ,\lambda }\rightarrow \boldsymbol{u}%
_{\varepsilon }\mbox{ in }L^{2}(0,T;V)\mbox{-weak},  \label{6.4'}
\end{equation}%
\begin{equation}
\boldsymbol{u}_{\varepsilon ,\lambda }\rightarrow \boldsymbol{u}%
_{\varepsilon }\mbox{ in }L^{2}(0,T;H)\mbox{-strong},  \label{6.4}
\end{equation}%
\begin{equation}
\frac{\partial \boldsymbol{u}_{\varepsilon ,\lambda }}{\partial t}%
\rightarrow \frac{\partial \boldsymbol{u}_{\varepsilon }}{\partial t}%
\mbox{
in }L^{2}(0,T;H^{-1}(\Omega )^{N})\mbox{-weak}  \label{6.5}
\end{equation}%
and 
\begin{equation}
p_{\varepsilon ,\lambda }\rightarrow p_{\varepsilon }\mbox{ in }%
L^{2}(0,T;L^{2}(\Omega )/\mathbb{R})\mbox{-weak.}  \label{6.6}
\end{equation}%
Let $\phi \in \mathcal{C}_{0}^{\infty }(\mathcal{O})$ where $\mathcal{O}%
=Q\times \mathbb{R}_{v}^{N}$. We multiply the Vlasov equation (\ref{r1'}) by 
$\phi $ and integrate by parts to get 
\begin{equation}
-\int_{\mathcal{O}}f_{\varepsilon ,\lambda }\left[ \frac{\partial \phi }{%
\partial t}+\varepsilon v\cdot \nabla \phi +(\boldsymbol{u}_{\varepsilon
,\lambda }\ast \theta _{\lambda }-v)\cdot \nabla _{v}\phi \right] dxdtdv=0.
\label{6.7}
\end{equation}%
We consider the terms in (\ref{6.7}) respectively. It is easy to see that,
as $\lambda \rightarrow 0$, 
\begin{equation*}
\int_{\mathcal{O}}f_{\varepsilon ,\lambda }\frac{\partial \phi }{\partial t}%
dxdtdv\rightarrow \int_{\mathcal{O}}f_{\varepsilon }\frac{\partial \phi }{%
\partial t}dxdtdv.
\end{equation*}%
For the second and fourth terms, since the functions $(t,x,v)\mapsto v\cdot
\nabla \phi $ and $(t,x,v)\mapsto v\cdot \nabla _{v}\phi $ belong to $%
\mathcal{C}_{0}^{\infty }(\mathcal{O})$, we use them as test functions to
get, as $\lambda \rightarrow 0$, 
\begin{equation*}
\int_{\mathcal{O}}f_{\varepsilon ,\lambda }\left[ \varepsilon v\cdot \nabla
\phi -v\cdot \nabla _{v}\phi \right] dxdtdv\rightarrow \int_{\mathcal{O}%
}f_{\varepsilon }\left[ \varepsilon v\cdot \nabla \phi -v\cdot \nabla
_{v}\phi \right] dxdtdv.
\end{equation*}%
Now, as for the term $\int_{\mathcal{O}}f_{\varepsilon ,\lambda }(%
\boldsymbol{u}_{\varepsilon ,\lambda }\ast \theta _{\lambda })\cdot \nabla
_{v}\phi dxdtdv$, we claim that 
\begin{equation}
\int_{\mathcal{O}}f_{\varepsilon ,\lambda }(\boldsymbol{u}_{\varepsilon
,\lambda }\ast \theta _{\lambda })\cdot \nabla _{v}\phi dxdtdv\rightarrow
\int_{\mathcal{O}}f_{\varepsilon }\boldsymbol{u}_{\varepsilon }\cdot \nabla
_{v}\phi dxdtdv.  \label{6.8}
\end{equation}%
Indeed it is sufficient to prove that, under the convergence result (\ref%
{6.4}) and for any $\psi \in \mathcal{C}_{0}^{\infty }(\mathcal{O})^{N}$, 
\begin{equation}
\int_{\mathcal{O}}f_{\varepsilon ,\lambda }\boldsymbol{u}_{\varepsilon
,\lambda }\cdot \psi dxdtdv\rightarrow \int_{\mathcal{O}}f_{\varepsilon }%
\boldsymbol{u}_{\varepsilon }\cdot \psi dxdtdv  \label{6.9}
\end{equation}%
and to apply it with $\psi =\nabla _{v}\phi $. For the proof of (\ref{6.9}),
we refer to the proof of a more involved result, Lemma \ref{l4.1} in Section
4.

Returning to (\ref{6.8}), we have 
\begin{eqnarray*}
\int_{\mathcal{O}}f_{\varepsilon ,\lambda }(\boldsymbol{u}_{\varepsilon
,\lambda }\ast \theta _{\lambda })\cdot \nabla _{v}\phi dxdtdv &=&\int_{%
\mathcal{O}}f_{\varepsilon ,\lambda }\left[ (\boldsymbol{u}_{\varepsilon
,\lambda }-\boldsymbol{u}_{\varepsilon })\ast \theta _{\lambda }\right]
\cdot \nabla _{v}\phi dxdtdv \\
&&+\int_{\mathcal{O}}f_{\varepsilon ,\lambda }(\boldsymbol{u}_{\varepsilon
}\ast \theta _{\lambda })\cdot \nabla _{v}\phi dxdtdv \\
&=&(I)+(II).
\end{eqnarray*}%
For the term $(I)$, we have the estimate 
\begin{eqnarray*}
\left\vert (I)\right\vert &\leq &\left\Vert f_{\varepsilon ,\lambda
}\right\Vert _{L^{\infty }(\mathcal{O})}\left\Vert \nabla _{v}\phi
\right\Vert _{\infty }\left\Vert \boldsymbol{u}_{\varepsilon ,\lambda }-%
\boldsymbol{u}_{\varepsilon }\right\Vert _{L^{2}(Q)}\left\Vert \theta
_{\lambda }\right\Vert _{L^{1}(Q)} \\
&\leq &C\left\Vert \boldsymbol{u}_{\varepsilon ,\lambda }-\boldsymbol{u}%
_{\varepsilon }\right\Vert _{L^{2}(Q)},
\end{eqnarray*}%
in which $C$ is a positive constant independent of $\varepsilon $ and $%
\lambda $. Thus, it follows from (\ref{6.4}) that $(I)\rightarrow 0$.
Regarding $(II)$, we have that $\boldsymbol{u}_{\varepsilon }\ast \theta
_{\lambda }\rightarrow \boldsymbol{u}_{\varepsilon }$ in $L^{2}(Q)$-strong
(use once again (\ref{6.4})), so that by (\ref{6.9}) we arrive at $%
(II)\rightarrow \int_{\mathcal{O}}f_{\varepsilon }\boldsymbol{u}%
_{\varepsilon }\cdot \nabla _{v}\phi \;dxdtdv$. (\ref{6.8}) follows thereby.

Taking into account all the above convergence results and passing to the
limit in (\ref{6.7}) as $\lambda \rightarrow 0$, we obtain 
\begin{equation*}
-\int_{\mathcal{O}}f_{\varepsilon }\left[ \frac{\partial \phi }{\partial t}%
+\varepsilon v\cdot \nabla \phi +(\boldsymbol{u}_{\varepsilon }-v)\cdot
\nabla _{v}\phi \right] dxdtdv=0,
\end{equation*}%
which amounts to 
\begin{equation*}
\frac{\partial f_{\varepsilon }}{\partial t}+\varepsilon v\cdot \nabla
f_{\varepsilon }+\Div_{v}\left( (\boldsymbol{u}_{\varepsilon
}-v)f_{\varepsilon }\right) =0\mbox{ in }\mathcal{D}^{\prime }(\mathcal{O}).
\end{equation*}%
Proceeding as in \cite[Section 4]{SM} we recover the reflection boundary
condition 
\begin{equation*}
f_{\varepsilon }(t,x,v)=f_{\varepsilon }(t,x,v^{\ast })\mbox{ for }x\in
\partial \Omega \mbox{ with }v\cdot \nu (x)<0\mbox{.}
\end{equation*}%
Next using the inequality $1-\gamma _{\lambda }(v)\leq 1_{\{\left\vert
v\right\vert \geq 1/2\lambda \}}$, we get 
\begin{eqnarray}
\left\vert \int_{\Omega \times \mathbb{R}^{N}}(1-\gamma _{\lambda
}(v))(f^{0}\ast \Theta _{\lambda })dxdv\right\vert &\leq &\int_{\Omega
\times \mathbb{R}^{N}}1_{\{\left\vert v\right\vert \geq 1/2\lambda
\}}(f^{0}\ast \Theta _{\lambda })dxdv  \label{6.10} \\
&\leq &4\lambda ^{2}\int_{\Omega \times \mathbb{R}^{N}}\left\vert
v\right\vert ^{2}(f^{0}\ast \Theta _{\lambda })dxdv  \notag \\
&\leq &C\lambda ^{2}\left\Vert f^{0}\right\Vert _{L^{1}(\Omega \times 
\mathbb{R}^{N})};\text{ see (\ref{r10})}.  \notag
\end{eqnarray}%
So, the functions $f^{0}\ast \Theta _{\lambda }$ and $\gamma _{\lambda
}(v)(f^{0}\ast \Theta _{\lambda })$ have the same $L^{1}$-limit $f^{0}$ as $%
\lambda \rightarrow 0$. Hence, letting $\lambda \rightarrow 0$, we arrive at 
\begin{equation*}
f_{\varepsilon }(0,x,v)=f^{0}(x,v)\mbox{ for }(x,v)\in \Omega \times \mathbb{%
R}^{N}.
\end{equation*}

Let us now deal with the Stokes system (\ref{r2'}). We choose $\boldsymbol{%
\Phi }\in \mathcal{C}_{0}^{\infty }(Q)^{N}$ and multiply (\ref{r2'}) by $%
\psi $ and integrate over $Q$; 
\begin{equation}
\begin{array}{l}
-\int_{Q}\boldsymbol{u}_{\varepsilon ,\lambda }\cdot \frac{\partial 
\boldsymbol{\Phi }}{\partial t}dxdt+\int_{Q}A_{0}^{\varepsilon }\nabla 
\boldsymbol{u}_{\varepsilon ,\lambda }\cdot \nabla \boldsymbol{\Phi }%
dxdt+\int_{Q}(A_{1}^{\varepsilon }\ast \nabla \boldsymbol{u}_{\varepsilon
,\lambda })\cdot \nabla \boldsymbol{\Phi }dxdt \\ 
\ \ \ \ \ \ -\int_{Q}p_{\varepsilon ,\lambda }\Div\boldsymbol{\Phi }%
dxdt=-\int_{\mathcal{O}}\gamma _{\lambda }(v)f_{\varepsilon ,\lambda }(%
\boldsymbol{u}_{\varepsilon ,\lambda }-v)\cdot \boldsymbol{\Phi }dxdtdv.%
\end{array}
\label{6.11}
\end{equation}%
In Eq.(\ref{6.11}), only the right-hand side is more involved. However,
proceeding as in (\ref{6.10}), one can check that 
\begin{equation*}
\int_{\mathcal{O}}\gamma _{\lambda }(v)f_{\varepsilon ,\lambda }(\boldsymbol{%
u}_{\varepsilon ,\lambda }-v)\cdot \boldsymbol{\Phi }dxdtdv
\end{equation*}%
and 
\begin{equation*}
\int_{\mathcal{O}}f_{\varepsilon ,\lambda }(\boldsymbol{u}_{\varepsilon
,\lambda }-v)\cdot \boldsymbol{\Phi }dxdtdv
\end{equation*}%
have the same limit, which is, using (\ref{6.9}) and the convergence results
(\ref{6.3})-(\ref{6.6}), nothing else but 
\begin{equation*}
\int_{\mathcal{O}}f_{\varepsilon }(\boldsymbol{u}_{\varepsilon }-v)\cdot 
\boldsymbol{\Phi }dxdtdv.
\end{equation*}%
Thus, passing to the limit in (\ref{6.11}), we realize that $\boldsymbol{u}%
_{\varepsilon }$ solves the equation 
\begin{equation*}
\frac{\partial \boldsymbol{u}_{\varepsilon }}{\partial t}-\Div\left(
A_{0}^{\varepsilon }\nabla \boldsymbol{u}_{\varepsilon
}+\int_{0}^{t}A_{1}^{\varepsilon }(t-\tau ,x)\nabla \boldsymbol{u}%
_{\varepsilon }(\tau ,x)d\tau \right) +\nabla p_{\varepsilon }=-\int_{%
\mathbb{R}^{N}}(\boldsymbol{u}_{\varepsilon }-v)f_{\varepsilon }dv%
\mbox{
in }Q.
\end{equation*}%
We also obtain the initial condition $\boldsymbol{u}_{\varepsilon }(0,x)=%
\boldsymbol{u}^{0}(x)$, $x\in \Omega $.

We have just shown that $(\boldsymbol{u}_{\varepsilon },f_{\varepsilon
},p_{\varepsilon })$ solves the system (\ref{1.1})-(\ref{1.5}). This
concludes the proof of Theorem \ref{t2.1}

It remains to check that the above triple verifies the same estimates as in
Lemma \ref{l2.1}. As we are going to see below, this is a mere consequence
of the following well known result:

\begin{itemize}
\item If $B$ is a Banach space with norm $\left\Vert \cdot \right\Vert $ and 
$f_{n}\rightarrow f$ in $B$-weak or weak$\ast $, then $\left\Vert
f\right\Vert \leq \lim \inf \left\Vert f_{n}\right\Vert $.
\end{itemize}

We can therefore state the counterpart of Lemma \ref{l2.1}.

\begin{lemma}
\label{l2.1'}Let $(\boldsymbol{u}_{\varepsilon },f_{\varepsilon
},p_{\varepsilon })$ be the solution to \emph{(\ref{1.1})-(\ref{1.5})}
constructed in Subsection \emph{\ref{sol}}. Then, 
\begin{equation}
\int_{\Omega \times \mathbb{R}^{N}}(1+\left\vert v\right\vert
^{2})f_{\varepsilon }dxdv+\int_{\Omega }\left\vert \boldsymbol{u}%
_{\varepsilon }\right\vert ^{2}dx+2\int_{0}^{t}\int_{\Omega \times \mathbb{R}%
^{N}}f_{\varepsilon }\left\vert \boldsymbol{u}_{\varepsilon }-v\right\vert
^{2}dxdvd\tau +\int_{0}^{t}\left\Vert \nabla \boldsymbol{u}_{\varepsilon
}(\tau )\right\Vert _{L^{2}(\Omega )}^{2}d\tau \leq C  \label{2.3''}
\end{equation}%
for any $0\leq t\leq T$\ and $\varepsilon >0$, where $C>0$ is independent of 
$\varepsilon $. Moreover if $f^{0}\in L^{p}(\Omega \times \mathbb{R}^{N})$, 
\emph{(}$1\leq p\leq \infty $\emph{)}, then 
\begin{equation}
\left\Vert f_{\varepsilon }\right\Vert _{L^{\infty }(0,T;L^{p}(\Omega \times 
\mathbb{R}^{N}))}\leq \exp (NT)\left\Vert f^{0}\right\Vert _{L^{p}(\Omega
\times \mathbb{R}^{N})}\mbox{ for any }\varepsilon >0.  \label{2.4''}
\end{equation}%
It also holds that 
\begin{equation}
\left\Vert \frac{\partial \boldsymbol{u}_{\varepsilon }}{\partial t}%
\right\Vert _{L^{2}(0,T;H^{-1}(\Omega )^{N})}\leq C  \label{2.4'''}
\end{equation}%
and 
\begin{equation}
\sup_{\varepsilon >0}\left\Vert p_{\varepsilon }\right\Vert
_{L^{2}(0,T;L^{2}(\Omega ))}\leq C.  \label{2.5''}
\end{equation}
\end{lemma}

\begin{proof}
We follow arguments similar to those in \cite{Boudin}. First and foremost,
we have by (\ref{2.3}) that 
\begin{equation}
\int_{\Omega \times \mathbb{R}^{N}}(1+\left\vert v\right\vert
^{2})f_{\varepsilon ,\lambda }dxdv+\int_{\Omega }\left\vert \boldsymbol{u}%
_{\varepsilon ,\lambda }\right\vert ^{2}dx+2\int_{0}^{t}\int_{\Omega \times 
\mathbb{R}^{N}}f_{\varepsilon ,\lambda }\left\vert \boldsymbol{u}%
_{\varepsilon ,\lambda }\ast \theta _{\lambda }-v\right\vert ^{2}dxdvd\tau
+\int_{0}^{t}\left\Vert \nabla \boldsymbol{u}_{\varepsilon ,\lambda }(\tau
)\right\Vert _{L^{2}(\Omega )}^{2}d\tau \leq C.  \label{e1}
\end{equation}%
The only term to deal with is actually $\int_{0}^{t}\int_{\Omega \times 
\mathbb{R}^{N}}f_{\varepsilon ,\lambda }\left\vert \boldsymbol{u}%
_{\varepsilon ,\lambda }\ast \theta _{\lambda }-v\right\vert ^{2}dxdvd\tau $
which we write as 
\begin{eqnarray*}
\int_{0}^{t}\int_{\Omega \times \mathbb{R}^{N}}f_{\varepsilon ,\lambda
}\left\vert \boldsymbol{u}_{\varepsilon ,\lambda }\ast \theta _{\lambda
}-v\right\vert ^{2}dxdvd\tau &=&\int_{0}^{t}\int_{\Omega \times \mathbb{R}%
^{N}}f_{\varepsilon ,\lambda }\left\vert \boldsymbol{u}_{\varepsilon
,\lambda }\ast \theta _{\lambda }\right\vert ^{2}dxdvd\tau \\
&&-2\int_{0}^{t}\int_{\Omega \times \mathbb{R}^{N}}f_{\varepsilon ,\lambda }(%
\boldsymbol{u}_{\varepsilon ,\lambda }\ast \theta _{\lambda })\cdot
vdxdvd\tau \\
&&+\int_{0}^{t}\int_{\Omega \times \mathbb{R}^{N}}f_{\varepsilon ,\lambda
}\left\vert v\right\vert ^{2}dxdvd\tau \\
&=&(I)-2(II)+(III).
\end{eqnarray*}%
Concerning $(III)$, we know that $f_{\varepsilon ,\lambda }\rightarrow
f_{\varepsilon }$ in $L^{\infty }(0,T;L^{\infty }(\Omega \times \mathbb{R}%
^{N}))$-weak$\ast $. Let $0<\eta <1$. Then because of (\ref{2.3}) we have 
\begin{equation*}
\int_{\Omega \times \mathbb{R}^{N}}\left\vert v\right\vert
^{2}f_{\varepsilon ,\lambda }\gamma _{\eta }(v)dxdv\leq \int_{\Omega \times 
\mathbb{R}^{N}}\left\vert v\right\vert ^{2}f_{\varepsilon ,\lambda }dxdv\leq
C.
\end{equation*}%
Hence there exists a function $g_{\eta }\in L^{\infty }([0,t])$ such that,
up to a subsequence of $\lambda \rightarrow 0$, setting $M_{2}(f_{%
\varepsilon ,\lambda }\gamma _{\eta })(\tau )=\int_{\Omega \times \mathbb{R}%
^{N}}\left\vert v\right\vert ^{2}f_{\varepsilon ,\lambda }\gamma _{\eta
}(v)dxdv$, 
\begin{equation*}
M_{2}(f_{\varepsilon ,\lambda }\gamma _{\eta })\rightarrow g_{\eta }\text{
in }L^{\infty }([0,t])\text{-weak}\ast ;
\end{equation*}%
thus 
\begin{equation*}
\left\Vert g_{\eta }\right\Vert _{L^{\infty }([0,t])}\leq ~\underset{\lambda
\rightarrow 0}{\lim \inf }\int_{0}^{t}M_{2}(f_{\varepsilon ,\lambda }\gamma
_{\eta })(\tau )d\tau \leq ~\underset{\lambda \rightarrow 0}{\lim \inf }%
\int_{0}^{t}\int_{\Omega \times \mathbb{R}^{N}}\left\vert v\right\vert
^{2}f_{\varepsilon ,\lambda }dxdvd\tau .
\end{equation*}%
On the other hand, the weak$\ast $ convergence $f_{\varepsilon ,\lambda
}\rightarrow f_{\varepsilon }$ in $L^{\infty }(0,T;L^{\infty }(\Omega \times 
\mathbb{R}^{N}))$ implies 
\begin{equation*}
M_{2}(f_{\varepsilon ,\lambda }\gamma _{\eta })\rightarrow
M_{2}(f_{\varepsilon ,\lambda }\gamma _{\eta })=\int_{\Omega \times \mathbb{R%
}^{N}}\left\vert v\right\vert ^{2}f_{\varepsilon }\gamma _{\eta }(v)dxdv%
\text{ in }L^{\infty }([0,t])\text{-weak}\ast
\end{equation*}%
since the product of a function $\chi \in L^{1}([0,t])$ by $\left\vert
v\right\vert ^{2}\gamma _{\eta }$ lies in $L^{1}((0,t)\times \Omega \times 
\mathbb{R}^{N})$. The uniqueness of the weak$\ast $-limit yields $%
M_{2}(f_{\varepsilon ,\lambda }\gamma _{\eta })(\tau )=g_{\eta }(\tau )$
a.e. $\tau \in (0,t)$. It therefore follows from the Fatou's lemma and from
the fact that $\left\vert v\right\vert ^{2}f_{\varepsilon }\gamma _{\eta
}(v)\rightarrow \left\vert v\right\vert ^{2}f_{\varepsilon }$ as $\eta
\rightarrow 0$, that 
\begin{equation*}
\int_{\Omega \times \mathbb{R}^{N}}\left\vert v\right\vert
^{2}f_{\varepsilon }dxdv\leq ~\underset{\eta \rightarrow 0}{\lim \inf }%
M_{2}(f_{\varepsilon ,\lambda }\gamma _{\eta })(\tau )=~\underset{\eta
\rightarrow 0}{\lim \inf }g_{\eta }(\tau )\leq ~\underset{\lambda
\rightarrow 0}{\lim \inf }\left\Vert M_{2}(f_{\varepsilon ,\lambda
})\right\Vert _{L^{\infty }(0,t)},
\end{equation*}%
i.e. 
\begin{equation*}
\int_{\Omega \times \mathbb{R}^{N}}\left\vert v\right\vert
^{2}f_{\varepsilon }dxdv\leq ~\underset{\lambda \rightarrow 0}{\lim \inf }%
\int_{\Omega \times \mathbb{R}^{N}}\left\vert v\right\vert
^{2}f_{\varepsilon ,\lambda }dxdv,
\end{equation*}%
whence 
\begin{equation*}
\int_{0}^{t}\int_{\Omega \times \mathbb{R}^{N}}\left\vert v\right\vert
^{2}f_{\varepsilon }dxdvd\tau \leq ~\underset{\lambda \rightarrow 0}{\lim
\inf }\int_{0}^{t}\int_{\Omega \times \mathbb{R}^{N}}\left\vert v\right\vert
^{2}f_{\varepsilon ,\lambda }dxdvd\tau .
\end{equation*}%
As for the first term $(I)$, one has 
\begin{eqnarray*}
\int_{0}^{t}\int_{\Omega \times \mathbb{R}^{N}}f_{\varepsilon ,\lambda
}\left\vert \boldsymbol{u}_{\varepsilon ,\lambda }\ast \theta _{\lambda
}\right\vert ^{2}dxdvd\tau &=&\int_{0}^{t}\int_{\Omega \times \mathbb{R}%
^{N}}\left\vert \boldsymbol{u}_{\varepsilon ,\lambda }\ast \theta _{\lambda
}\right\vert ^{2}f_{\varepsilon ,\lambda }(1-\gamma _{\eta }(v))dxdvd\tau \\
&&+\int_{0}^{t}\int_{\Omega \times \mathbb{R}^{N}}\left\vert \boldsymbol{u}%
_{\varepsilon ,\lambda }\ast \theta _{\lambda }\right\vert
^{2}f_{\varepsilon ,\lambda }\gamma _{\eta }(v)dxdvd\tau \\
&\geq &\int_{0}^{t}\int_{\Omega \times \mathbb{R}^{N}}\left\vert \boldsymbol{%
u}_{\varepsilon ,\lambda }\ast \theta _{\lambda }\right\vert
^{2}f_{\varepsilon ,\lambda }\gamma _{\eta }(v)dxdvd\tau .
\end{eqnarray*}%
For any fixed $\eta $, 
\begin{equation*}
\int_{0}^{t}\int_{\Omega \times \mathbb{R}^{N}}\left\vert \boldsymbol{u}%
_{\varepsilon ,\lambda }\ast \theta _{\lambda }\right\vert
^{2}f_{\varepsilon ,\lambda }\gamma _{\eta }(v)dxdvd\tau \rightarrow
\int_{0}^{t}\int_{\Omega \times \mathbb{R}^{N}}\left\vert \boldsymbol{u}%
_{\varepsilon }\right\vert ^{2}f_{\varepsilon }\gamma _{\eta }(v)dxdvd\tau
\end{equation*}%
when $\lambda \rightarrow 0$. Indeed, it is easy to see that $\left\vert 
\boldsymbol{u}_{\varepsilon ,\lambda }\ast \theta _{\lambda }\right\vert
^{2}\gamma _{\eta }\rightarrow \left\vert \boldsymbol{u}_{\varepsilon
}\right\vert ^{2}f_{\varepsilon }\gamma _{\eta }$ in $L^{1}((0,t)\times
\Omega \times \mathbb{R}^{N})$-strong as $\lambda \rightarrow 0$, so that
combining this with (\ref{6.3}) (for $p=\infty $) we get our result. Thus,
using once again Fatou's lemma, 
\begin{eqnarray*}
\int_{0}^{t}\int_{\Omega \times \mathbb{R}^{N}}\left\vert \boldsymbol{u}%
_{\varepsilon }\right\vert ^{2}f_{\varepsilon }dxdvd\tau &\leq &~\underset{%
\eta \rightarrow 0}{\lim \inf }\int_{0}^{t}\int_{\Omega \times \mathbb{R}%
^{N}}\left\vert \boldsymbol{u}_{\varepsilon }\right\vert ^{2}f_{\varepsilon
}\gamma _{\eta }(v)dxdvd\tau \\
&=&~\underset{\eta \rightarrow 0}{\lim \inf }\underset{\lambda \rightarrow 0}%
{\lim \inf }\int_{0}^{t}\int_{\Omega \times \mathbb{R}^{N}}\left\vert 
\boldsymbol{u}_{\varepsilon ,\lambda }\ast \theta _{\lambda }\right\vert
^{2}f_{\varepsilon ,\lambda }\gamma _{\eta }(v)dxdvd\tau \\
&\leq &~\underset{\lambda \rightarrow 0}{\lim \inf }\int_{0}^{t}\int_{\Omega
\times \mathbb{R}^{N}}\left\vert \boldsymbol{u}_{\varepsilon ,\lambda }\ast
\theta _{\lambda }\right\vert ^{2}f_{\varepsilon ,\lambda }dxdvd\tau .
\end{eqnarray*}%
Finally, for $(II)$, we have 
\begin{eqnarray*}
\int_{0}^{t}\int_{\Omega \times \mathbb{R}^{N}}f_{\varepsilon ,\lambda }(%
\boldsymbol{u}_{\varepsilon ,\lambda }\ast \theta _{\lambda })\cdot
vdxdvd\tau &=&\int_{0}^{t}\int_{\Omega \times \mathbb{R}^{N}}f_{\varepsilon
,\lambda }(\boldsymbol{u}_{\varepsilon ,\lambda }\ast \theta _{\lambda }-%
\boldsymbol{u}_{\varepsilon })\cdot vdxdvd\tau \\
&&+\int_{0}^{t}\int_{\Omega \times \mathbb{R}^{N}}f_{\varepsilon ,\lambda }%
\boldsymbol{u}_{\varepsilon }\cdot vdxdvd\tau \\
&=&(A)+(B).
\end{eqnarray*}%
Dealing with $(A)$, we have, by setting $\boldsymbol{v}_{\varepsilon
,\lambda }=\boldsymbol{u}_{\varepsilon ,\lambda }\ast \theta _{\lambda }-%
\boldsymbol{u}_{\varepsilon }$, 
\begin{equation*}
(A)=\int_{0}^{t}\int_{\Omega \times \mathbb{R}^{N}}f_{\varepsilon ,\lambda
}(1-\gamma _{\eta }(v))\boldsymbol{v}_{\varepsilon ,\lambda }\cdot
vdxdvd\tau +\int_{0}^{t}\int_{\Omega \times \mathbb{R}^{N}}f_{\varepsilon
,\lambda }\gamma _{\eta }(v)\boldsymbol{v}_{\varepsilon ,\lambda }\cdot
vdxdvd\tau ,
\end{equation*}%
and using the inequality $1-\gamma _{\eta }(v)\leq 1_{\{\left\vert
v\right\vert \geq 1/2\eta \}}$,%
\begin{eqnarray*}
\left\vert \int_{0}^{t}\int_{\Omega \times \mathbb{R}^{N}}f_{\varepsilon
,\lambda }(1-\gamma _{\eta }(v))\boldsymbol{v}_{\varepsilon ,\lambda }\cdot
vdxdvd\tau \right\vert &\leq &\int_{0}^{t}\int_{\Omega \times \mathbb{R}%
^{N}}f_{\varepsilon ,\lambda }1_{\{\left\vert v\right\vert \geq 1/2\eta
\}}\left\vert \boldsymbol{v}_{\varepsilon ,\lambda }\right\vert \left\vert
v\right\vert dxdvd\tau \\
&\leq &\int_{0}^{t}\int_{\Omega }\left[ \int_{\mathbb{R}^{N}}\left\vert
v\right\vert f_{\varepsilon ,\lambda }1_{\{\left\vert v\right\vert \geq
1/2\eta \}}dv\right] \left\vert \boldsymbol{v}_{\varepsilon ,\lambda
}\right\vert dxd\tau \\
&\leq &\int_{0}^{t}\left\{ \left( \int_{\Omega }\left( \int_{\mathbb{R}%
^{N}}\left\vert v\right\vert f_{\varepsilon ,\lambda }1_{\{\left\vert
v\right\vert \geq 1/2\eta \}}dv\right) ^{\frac{6}{5}}dx\right) ^{\frac{5}{6}%
}\left( \int_{\Omega }\left\vert \boldsymbol{v}_{\varepsilon ,\lambda
}\right\vert ^{6}dx\right) ^{\frac{1}{6}}\right\} d\tau .
\end{eqnarray*}%
But 
\begin{eqnarray*}
\int_{\Omega }\left( \int_{\mathbb{R}^{N}}\left\vert v\right\vert
f_{\varepsilon ,\lambda }1_{\{\left\vert v\right\vert \geq 1/2\eta
\}}dv\right) ^{\frac{6}{5}}dx &\leq &C(\Omega )\left( \int_{\Omega }\int_{%
\mathbb{R}^{N}}\left\vert v\right\vert f_{\varepsilon ,\lambda
}1_{\{\left\vert v\right\vert \geq 1/2\eta \}}dvdx\right) ^{\frac{6}{5}} \\
&\leq &C(\Omega )\lambda ^{\frac{6}{5}}\left( \int_{\Omega \times \mathbb{R}%
^{N}}\left\vert v\right\vert ^{2}f_{\varepsilon ,\lambda }1_{\{\left\vert
v\right\vert \geq 1/2\eta \}}dvdx\right) ^{\frac{6}{5}} \\
&\leq &C\lambda ^{\frac{6}{5}}\text{ because of (\ref{2.3}).}
\end{eqnarray*}%
Recalling that $\boldsymbol{v}_{\varepsilon ,\lambda }\in
L^{2}(0,T;H_{0}^{1}(\Omega )^{N})\hookrightarrow L^{1}(0,T;L^{6}(\Omega
)^{N})$, it follows from (\ref{2.3}) that 
\begin{equation*}
\int_{0}^{t}\left( \int_{\Omega }\left\vert \boldsymbol{v}_{\varepsilon
,\lambda }\right\vert ^{6}dx\right) ^{\frac{1}{6}}d\tau \leq C,
\end{equation*}%
so that 
\begin{equation*}
\left\vert \int_{0}^{t}\int_{\Omega \times \mathbb{R}^{N}}f_{\varepsilon
,\lambda }(1-\gamma _{\eta }(v))\boldsymbol{v}_{\varepsilon ,\lambda }\cdot
vdxdvd\tau \right\vert \leq C\lambda .
\end{equation*}%
It follows that $\int_{0}^{t}\int_{\Omega \times \mathbb{R}%
^{N}}f_{\varepsilon ,\lambda }(1-\gamma _{\eta }(v))\boldsymbol{v}%
_{\varepsilon ,\lambda }\cdot vdxdvd\tau \rightarrow 0$ as $\lambda
\rightarrow 0$.

We claim that 
\begin{equation*}
\int_{0}^{t}\int_{\Omega \times \mathbb{R}^{N}}f_{\varepsilon ,\lambda
}\gamma _{\eta }(v)\boldsymbol{v}_{\varepsilon ,\lambda }\cdot vdxdvd\tau
\rightarrow 0\text{ as }\lambda \rightarrow 0.
\end{equation*}%
Indeed 
\begin{eqnarray*}
\left\vert \int_{0}^{t}\int_{\Omega \times \mathbb{R}^{N}}f_{\varepsilon
,\lambda }\gamma _{\eta }(v)\boldsymbol{v}_{\varepsilon ,\lambda }\cdot
vdxdvd\tau \right\vert &\leq &\left\Vert f_{\varepsilon ,\lambda
}\right\Vert _{L^{\infty }(\mathcal{O})}\int_{0}^{t}\int_{\Omega
}\int_{B(0,2)}\left\vert v\right\vert \left\vert \boldsymbol{v}_{\varepsilon
,\lambda }\right\vert dvdxd\tau \\
&\leq &C\int_{0}^{t}\int_{\Omega }\int_{B(0,2)}\left\vert v\right\vert
\left\vert \boldsymbol{v}_{\varepsilon ,\lambda }\right\vert dvdxd\tau
\end{eqnarray*}%
and 
\begin{equation*}
\int_{0}^{t}\int_{\Omega }\int_{B(0,2)}\left\vert v\right\vert \left\vert 
\boldsymbol{v}_{\varepsilon ,\lambda }\right\vert dvdxd\tau \leq 2\left\vert
B(0,2)\right\vert \left\vert \Omega \right\vert ^{\frac{1}{2}}\left\Vert 
\boldsymbol{v}_{\varepsilon ,\lambda }\right\Vert _{L^{2}(Q)}\rightarrow 0%
\text{ as }\lambda \rightarrow 0
\end{equation*}%
since $\boldsymbol{v}_{\varepsilon ,\lambda }\rightarrow 0$ in $L^{2}(Q)$ as 
$\lambda \rightarrow 0$. It follows that $(A)\rightarrow 0$ as $\lambda
\rightarrow 0$. We use the same kind of arguments to show that $%
(B)\rightarrow \int_{0}^{t}\int_{\Omega \times \mathbb{R}^{N}}f_{\varepsilon
}\boldsymbol{u}_{\varepsilon }\cdot vdxdvd\tau $, that is, 
\begin{equation*}
(II)\rightarrow \int_{0}^{t}\int_{\Omega \times \mathbb{R}%
^{N}}f_{\varepsilon }\boldsymbol{u}_{\varepsilon }\cdot vdxdvd\tau .
\end{equation*}%
Coming back to (\ref{e1}) and taking there the $\lim \inf $ as $\lambda
\rightarrow 0$, we get at once (\ref{2.3''}). The lemma follows thereby.
\end{proof}

\begin{remark}
\bigskip \label{r2.3}\emph{We observe that the sequence }$\left( \int_{%
\mathbb{R}^{N}}f_{\varepsilon }(\boldsymbol{u}_{\varepsilon }-v)dv\right)
_{\varepsilon >0}$\emph{\ is bounded in }$L^{1}(Q)^{N}$\emph{. Indeed} 
\begin{eqnarray*}
\int_{Q}\left\vert \int_{\mathbb{R}^{N}}f_{\varepsilon }(\boldsymbol{u}%
_{\varepsilon }-v)dv\right\vert dxdt &\leq &\int_{Q}\int_{\mathbb{R}%
^{N}}f_{\varepsilon }\left\vert \boldsymbol{u}_{\varepsilon }-v\right\vert
dvdxdt \\
&=&\int_{\mathcal{O}}\sqrt{f_{\varepsilon }}(1+\left\vert v\right\vert )%
\sqrt{f_{\varepsilon }}\frac{\left\vert \boldsymbol{u}_{\varepsilon
}-v\right\vert }{(1+\left\vert v\right\vert )}dvdxdt \\
&\leq &\sqrt{2}\left( \int_{\mathcal{O}}f_{\varepsilon }(1+\left\vert
v\right\vert ^{2})dxdvdt\right) ^{\frac{1}{2}}\left( \int_{\mathcal{O}}\frac{%
f_{\varepsilon }\left\vert \boldsymbol{u}_{\varepsilon }-v\right\vert ^{2}}{%
(1+\left\vert v\right\vert )^{2}}dvdxdt\right) ^{\frac{1}{2}} \\
&\leq &C(\mbox{\emph{\ see estimate (\ref{2.3''}) of  Lemma \ref{l2.1'}.}})
\end{eqnarray*}
\end{remark}

%
%\bigskip
%
%To conclude with the proof of Theorem \ref{t2.1}, arguing exactly as in the
%proof of Theorem 2.7 in \cite{PG}, we get the global existence and
%uniqueness of the solution of our problem.

\section{Brief introduction to $\Sigma $-convergence}

This section is far from being a comprehensive introduction to $\Sigma $%
-convergence. It is rather a pretext for fixing notations and recalling
fundamental results pertaining to $\Sigma $-convergence. We shall restrict
ourselves to concepts relevant to our context.

\subsection{Algebras with mean value - An overview}

We refer the reader to \cite{Casado, Hom1, Deterhom, Zhikov4} for an
extensive presentation of the concept of algebras with mean value (algebras
wmv, in short).

Let $A$ be an algebra wmv on $\mathbb{R}^{N}$, that is, a closed subalgebra
of the $\mathcal{C}^{\ast }$-algebra of bounded uniformly continuous
functions on $\mathbb{R}^{N}$, $\mathrm{BUC}(\mathbb{R}^{N})$, which
contains the constants, is translation invariant and is such that any of its
elements possesses a mean value in the following sense: for any $u\in A$,
the sequence $(u^{\varepsilon })_{\varepsilon >0}$ (defined by $%
u^{\varepsilon }(x)=u(x/\varepsilon )$, $x\in \mathbb{R}^{N}$) weakly$\ast $%
-converges in $L^{\infty }(\mathbb{R}^{N})$ to some constant real function $%
M(u)$ (called the mean value of $u$) as $\varepsilon \rightarrow 0$. We
denote by $\Delta (A)$ the spectrum of $A$ and by $\mathcal{G}$ the Gelfand
transformation on $A$. Let $B_{A}^{p}(\mathbb{R}^{N})$ ($1\leq p<\infty $)
denote the Besicovitch space associated to $A$, that is, the closure of $A$
with respect to the Besicovitch seminorm 
\begin{equation*}
\left\Vert u\right\Vert _{p}=\left( \underset{r\rightarrow +\infty }{\lim
\sup }\frac{1}{\left\vert B_{r}\right\vert }\int_{B_{r}}\left\vert
u(y)\right\vert ^{p}dy\right) ^{1/p}
\end{equation*}%
where $B_{r}$ is the open ball of $\mathbb{R}^{N}$ centered at the origin
and of radius $r>0$. We set 
\begin{equation*}
B_{A}^{\infty }(\mathbb{R}^{N})=\{f\in \cap _{1\leq p<\infty }B_{A}^{p}(%
\mathbb{R}^{N}):\sup_{1\leq p<\infty }\left\Vert f\right\Vert _{p}<\infty
\}\;\;\;\;\;\;\;\;\;
\end{equation*}%
and we endow it with the seminorm $\left[ f\right] _{\infty }=\sup_{1\leq
p<\infty }\left\Vert f\right\Vert _{p}$. So topologized, the spaces $%
B_{A}^{p}(\mathbb{R}^{N})$ ($1\leq p\leq \infty $) are complete seminormed
vector spaces which are not in general Fr\'{e}chet spaces since they are not
separated in general. We denote by $\mathcal{B}_{A}^{p}(\mathbb{R}^{N})$ the
completion of $B_{A}^{p}(\mathbb{R}^{N})$ with respect to $\left\Vert \cdot
\right\Vert _{p}$ for $1\leq p<\infty $, and with respect to $\left[ \cdot %
\right] _{\infty }$ for $p=\infty $. The following hold true \cite{CMP, NA}:

\begin{itemize}
\item[(\textbf{1)}] The Gelfand transformation $\mathcal{G}:A\rightarrow 
\mathcal{C}(\Delta (A))$ extends by continuity to a unique continuous linear
mapping (still denoted by $\mathcal{G}$) of $B_{A}^{p}(\mathbb{R}^{N})$ into 
$L^{p}(\Delta (A))$, which in turn induces an isometric isomorphism $%
\mathcal{G}_{1}$ of $B_{A}^{p}(\mathbb{R}^{N})/\mathcal{N}=\mathcal{B}%
_{A}^{p}(\mathbb{R}^{N})$ onto $L^{p}(\Delta (A))$ (where $\mathcal{N}%
=\{u\in B_{A}^{p}(\mathbb{R}^{N}):\mathcal{G}(u)=0\}$). Moreover if $u\in
B_{A}^{p}(\mathbb{R}^{N})\cap L^{\infty }(\mathbb{R}^{N})$ then $\mathcal{G}%
(u)\in L^{\infty }(\Delta (A))$ and $\left\Vert \mathcal{G}(u)\right\Vert
_{L^{\infty }(\Delta (A))}\leq \left\Vert u\right\Vert _{L^{\infty }(\mathbb{%
R}^{N})}$.

\item[(\textbf{2)}] The mean value $M$ defined on $A$, extends by continuity
to a positive continuous linear form (still denoted by $M$) on $B_{A}^{p}(%
\mathbb{R}^{N})$ satisfying $M(u)=\int_{\Delta (A)}\mathcal{G}(u)d\beta $ ($%
u\in B_{A}^{p}(\mathbb{R}^{N})$). Furthermore, $M(\tau _{a}u)=M(u)$ for each 
$u\in B_{A}^{p}(\mathbb{R}^{N})$ and all $a\in \mathbb{R}^{N}$, where $\tau
_{a}u=u(\cdot +a)$. Moreover for $u\in B_{A}^{p}(\mathbb{R}^{N})$ we have $%
\left\Vert u\right\Vert _{p}=\left[ M(\left\vert u\right\vert ^{p})\right]
^{1/p}$, and for $u+\mathcal{N}\in \mathcal{B}_{A}^{p}(\mathbb{R}^{N})$ we
may still define its mean value once again denoted by $M$, as $M(u+\mathcal{N%
})=M(u)$.
\end{itemize}

For $u=v+\mathcal{N}\in \mathcal{B}_{A}^{p}(\mathbb{R}^{N})$ ($1\leq p\leq
\infty $) and $y\in \mathbb{R}^{N}$, we define in a natural way the
translate $\tau _{y}u=v(\cdot +y)+\mathcal{N}$ of $u$, and as it can be seen
in \cite{NA2014, Deterhom}, this is well defined and induces a strongly
continuous $N$-parameter group of isometries $T(y):\mathcal{B}_{A}^{p}(%
\mathbb{R}^{N})\rightarrow \mathcal{B}_{A}^{p}(\mathbb{R}^{N})$ defined by $%
T(y)u=\tau _{y}u$. We denote by $\overline{\partial }/\partial y_{i}$ ($%
1\leq i\leq N$) the infinitesimal generator of $T(y)$ along the $i$th
coordinate direction. We refer the reader to \cite{NA2014, Deterhom} for the
properties of $\overline{\partial }/\partial y_{i}$ as well as for those of
the associated Sobolev-type spaces $\mathcal{B}_{A}^{1,p}(\mathbb{R}^{N})$
and $\mathcal{B}_{\#A}^{1,p}(\mathbb{R}^{N})$.

Now, let $A$ be an algebra wmv on $\mathbb{R}^{N}$. For $\mu \in \Delta (A)$
and $f\in A$, define $T_{\mu }f$ by $T_{\mu }f(y)=\mu (\tau _{y}f)$, $y\in 
\mathbb{R}^{N}$. $T_{\mu }f$ is well defined as an element of $\mathrm{BUC}(%
\mathbb{R}^{N})$ since $A$ is translation invariant. Whence a bounded linear
operator $T_{\mu }:A\rightarrow \mathrm{BUC}(\mathbb{R}^{N})$.

\begin{definition}
\label{d3.1}\emph{The algebra wmv }$A$\emph{\ is said to be }introverted%
\emph{\ if }$T_{\mu }(A)\subset A$\emph{\ for any }$\mu \in \Delta (A)$\emph{%
.}
\end{definition}

Let $A$ be an introverted algebra wmv on $\mathbb{R}^{N}$. Then \cite[%
Theorem 3.2]{NA2014} its spectrum $\Delta (A)$ is a compact topological
semigroup. In order to simplify the notations, the semigroup operation in $%
\Delta (A)$ is additively written. With this in mind, set 
\begin{equation*}
K(A)=\cap _{s\in \Delta (A)}(s+\Delta (A))\mbox{, the \emph{kernel} of }%
\Delta (A).
\end{equation*}%
The following result provides us with the structure of $K(A)$.

\begin{theorem}[{\protect\cite[Theorem 3.4]{NA2014}}]
\label{t3.1}Let $A$ be an introverted algebra wmv on $\mathbb{R}^{N}$. Then

\begin{itemize}
\item[(i)] $K(A)$ is a compact topological group.

\item[(ii)] The mean value $M$ on $A$ can be identified as the Haar integral
over $K(A)$.
\end{itemize}
\end{theorem}

With the help of Theorem \ref{t3.1}, we can define the convolution over $%
\Delta (A)$ in terms of its kernel $K(A)$. Indeed, as proved in \cite{NA2014}%
, we have $r+s\in K(A)$ whenever $r\in \Delta (A)$ and $s\in K(A)$. Thus,
let $p,q,m\geq 1$ be real numbers satisfying $\frac{1}{p}+\frac{1}{q}=1+%
\frac{1}{m}$. For $u\in L^{p}(\Delta (A))$ and $v\in L^{q}(\Delta (A))$ we
define the convolution product $u\widehat{\ast }v$ as follows: 
\begin{equation*}
(u\widehat{\ast }v)(s)=\int_{K(A)}u(r)v(s-r)d\beta (r)\mbox{, \ a.e. }s\in
\Delta (A),
\end{equation*}%
where $-r$ stands for the inverse of $r\in K(A)$ (recall that $K(A)$ is an
Abelian group). Then $\widehat{\ast }$ is well defined since $K(A)$ is an
ideal of $\Delta (A)$, and we have that $\int_{K(A)}u(r)v(s-r)d\beta
(r)=\int_{\Delta (A)}u(r)v(s-r)d\beta (r)$ since $\beta $ is supported by $%
K(A)$. Indeed for $s\in \Delta (A)$ and $r\in K(A)$, $-r$ exists in $K(A)$
and $s-r\in K(\Delta (A))$. It holds that $u\widehat{\ast }v\in L^{m}(\Delta
(A))$ and further: 
\begin{equation*}
\left\Vert u\widehat{\ast }v\right\Vert _{L^{m}(\Delta (A))}\leq \left\Vert
u\right\Vert _{L^{p}(\Delta (A))}\left\Vert v\right\Vert _{L^{q}(\Delta
(A))}.
\end{equation*}%
Now let $u\in L^{p}(\mathbb{R}^{N};L^{p}(\Delta (A)))$ and $v\in L^{q}(%
\mathbb{R}^{N};L^{q}(\Delta (A)))$. We define the double convolution $u\ast
\ast v$ as follows: 
\begin{eqnarray*}
(u\ast \ast v)(x,s) &=&\int_{\mathbb{R}^{N}}\left[ \left( u(t,\cdot )%
\widehat{\ast }v(x-t,\cdot )\right) (s)\right] dt \\
&\equiv &\int_{\mathbb{R}^{N}}\int_{K(A)}u(t,r)v(x-t,s-r)d\beta (r)\,dt%
\mbox{,
a.e. }(x,s)\in \mathbb{R}^{N}\times \Delta (A).
\end{eqnarray*}%
Then $\ast \ast $ is well defined as an element of $L^{m}(\mathbb{R}%
^{N}\times \Delta (A))$ and satisfies 
\begin{equation*}
\left\Vert u\ast \ast v\right\Vert _{L^{m}(\mathbb{R}^{N}\times \Delta
(A))}\leq \left\Vert u\right\Vert _{L^{p}(\mathbb{R}^{N}\times \Delta
(A))}\left\Vert v\right\Vert _{L^{q}(\mathbb{R}^{N}\times \Delta (A))}.
\end{equation*}%
It is to be noted that if $u\in L^{p}(\Omega ;L^{p}(\Delta (A)))$, and $v\in
L^{q}(\mathbb{R}^{N};L^{q}(\Delta (A)))$, we still define $u\ast \ast v$ by
replacing $u$ by its zero extension over $\mathbb{R}^{N}$.

Finally, for $u\in L^{p}(\mathbb{R}^{N};\mathcal{B}_{A}^{p}(\mathbb{R}^{N}))$
and $v\in L^{q}(\mathbb{R}^{N};\mathcal{B}_{A}^{q}(\mathbb{R}^{N}))$ we
define the double convolution still denoted by $\ast \ast $ as follows: $%
u\ast \ast v$ is that element of $L^{m}(\mathbb{R}^{N};\mathcal{B}_{A}^{m}(%
\mathbb{R}^{N}))$ defined by 
\begin{equation*}
\mathcal{G}_{1}(u\ast \ast v)=\widehat{u}\ast \ast \widehat{v}.
\end{equation*}

\subsection{ $\Sigma $-convergence method}

Throughout this section, $\Omega $ is an open subset of $\mathbb{R}^{N}$,
and unless otherwise specified, $A$ is an algebra with mean value on $%
\mathbb{R}^{N}$.

\begin{definition}
\label{d3.2}\emph{(1) A sequence }$\left( u_{\varepsilon }\right)
_{\varepsilon >0}\subset L^{p}\left( \Omega \right) $\emph{\ }$(1\leq
p<\infty )$\emph{\ is said to }weakly $\Sigma $-converge\emph{\ in }$%
L^{p}\left( \Omega \right) $\emph{\ to some }$u_{0}\in L^{p}(\Omega ;%
\mathcal{B}_{A}^{p}(\mathbb{R}^{N}))$\emph{\ if as }$\varepsilon \rightarrow
0$\emph{, }%
\begin{equation}
\int_{\Omega }u_{\varepsilon }\left( x\right) f^{\varepsilon }\left(
x\right) dx\rightarrow \iint_{\Omega \times \Delta (A)}\widehat{u}_{0}\left(
x,s\right) \widehat{f}\left( x,s\right)\; dxd\beta \left( s\right)
\label{2.3'}
\end{equation}%
\emph{for all }$f\in L^{p^{\prime }}\left( \Omega ;A\right) $\emph{\ }$%
\left( 1/p^{\prime }=1-1/p\right) $\emph{\ where }$f^{\varepsilon }\left(
x\right) =f\left( x,x/\varepsilon \right) $\emph{\ and }$\widehat{f}\left(
x,\cdot \right) =\mathcal{G}(f\left( x,\cdot \right) )$\emph{\ }$a.e.$\emph{%
\ in }$x\in \Omega $\emph{. We denote this by }$u_{\varepsilon }\rightarrow
u_{0}$\emph{\ in }$L^{p}(\Omega )$\emph{-weak }$\Sigma $\emph{.}

\noindent \emph{(2) A sequence }$(u_{\varepsilon })_{\varepsilon >0}\subset
L^{p}(\Omega )$\emph{\ }$(1\leq p<\infty )$\emph{\ is said to }strongly $%
\Sigma $-converge\emph{\ in }$L^{p}(\Omega )$\emph{\ to some }$u_{0}\in
L^{p}(\Omega ;\mathcal{B}_{A}^{p}(\mathbb{R}^{N}))$\emph{\ if it is weakly }$%
\Sigma $\emph{-convergent and further satisfies the following condition: }%
\begin{equation*}
\left\Vert u_{\varepsilon }\right\Vert _{L^{p}(\Omega )}\rightarrow
\left\Vert \widehat{u}_{0}\right\Vert _{L^{p}(\Omega \times \Delta (A))}.
\end{equation*}%
\emph{We denote this by }$u_{\varepsilon }\rightarrow u_{0}$\emph{\ in }$%
L^{p}(\Omega )$\emph{-strong }$\Sigma $\emph{.}
\end{definition}

We recall here that $\widehat{u}_{0}=\mathcal{G}_{1}\circ u_{0}$ and $%
\widehat{f}=\mathcal{G}\circ f$, $\mathcal{G}_{1}$ being the isometric
isomorphism sending $\mathcal{B}_{A}^{p}(\mathbb{R}^{N})$ onto $L^{p}(\Delta
(A))$ and $\mathcal{G}$, the Gelfand transformation on $A$.

In the sequel the letter $E$ will throughout denote any ordinary sequence $%
(\varepsilon _{n})_{n}$ (integers $n\geq 0$) with $0<\varepsilon _{n}\leq 1$
and $\varepsilon _{n}\rightarrow 0$ as $n\rightarrow \infty $. The following
two results hold (see e.g. \cite{Casado, CMP, NA} for their justification).

\begin{theorem}
\label{t3.2}\emph{(i)} Any bounded sequence $(u_{\varepsilon })_{\varepsilon
\in E}$ in $L^{p}(\Omega )$ (for $1<p<\infty $) admits a subsequence which
is weakly $\Sigma $-convergent in $L^{p}(\Omega )$.

\noindent \emph{(ii)} Any uniformly integrable sequence $(u_{\varepsilon
})_{\varepsilon \in E}$ in $L^{1}(\Omega )$ admits a subsequence which is
weakly $\Sigma $-convergent in $L^{1}(\Omega )$.
\end{theorem}

\begin{theorem}
\label{t3.3}Let $1<p<\infty $. Let $(u_{\varepsilon })_{\varepsilon \in E}$
be a bounded sequence in $W^{1,p}(\Omega )$. Then there exist a subsequence $%
E^{\prime }$ of $E$, and a couple $(u_{0},u_{1})\in W^{1,p}(\Omega
;I_{A}^{p}(\mathbb{R}^{N}))\times L^{p}(\Omega ;\mathcal{B}_{A}^{1,p}(%
\mathbb{R}^{N}))$ such that, as $E^{\prime }\ni \varepsilon \rightarrow 0$, 
\begin{equation*}
u_{\varepsilon }\rightarrow u_{0}\ \mbox{in }L^{p}(\Omega )\mbox{-weak }%
\Sigma \mbox{;}
\end{equation*}%
\begin{equation*}
\frac{\partial u_{\varepsilon }}{\partial x_{i}}\rightarrow \frac{\partial
u_{0}}{\partial x_{i}}+\frac{\overline{\partial }u_{1}}{\partial y_{i}}%
\mbox{\ in }L^{p}(\Omega )\mbox{-weak }\Sigma \mbox{, }1\leq i\leq N.
\end{equation*}
\end{theorem}

\begin{remark}
\label{r2.1}\emph{In the above result, }$I_{A}^{p}(\mathbb{R}^{N})$\emph{\
stands for the space of invariant functions in }$\mathcal{B}_{A}^{p}(\mathbb{%
R}^{N})$\emph{\ under the group of transformation }$T(y)$\emph{\ of the
preceding subsection: }$u\in I_{A}^{p}(\mathbb{R}^{N})$\emph{\ if and only
if }$\overline{\nabla }_{y}u=0$\emph{. If we assume the algebra }$A$\emph{\
to be ergodic, then }$I_{A}^{p}(\mathbb{R}^{N})$\emph{\ consists of constant
functions, so that the function }$u_{0}$\emph{\ in Theorem \ref{t3.3} does
not depend on }$y$\emph{, that is, }$u_{0}\in W^{1,p}(\Omega )$\emph{. We
thus recover the already known result proved in \cite{NA} in the case of
ergodic algebras.}
\end{remark}

The next result deals with the $\Sigma $-convergence of a product of
sequences.

\begin{theorem}[{\protect\cite[Theorem 6]{DPDE}}]
\label{t3.5}Let $1<p,q<\infty $ and $r\geq 1$ be such that $1/r=1/p+1/q\leq
1 $. Assume $(u_{\varepsilon })_{\varepsilon \in E}\subset L^{q}(\Omega )$
is weakly $\Sigma $-convergent in $L^{q}(\Omega )$ to some $u_{0}\in
L^{q}(\Omega ;\mathcal{B}_{A}^{q}(\mathbb{R}^{N}))$, and $(v_{\varepsilon
})_{\varepsilon \in E}\subset L^{p}(\Omega )$ is strongly $\Sigma $%
-convergent in $L^{p}(\Omega )$ to some $v_{0}\in L^{p}(\Omega ;\mathcal{B}%
_{A}^{p}(\mathbb{R}^{N}))$. Then the sequence $(u_{\varepsilon
}v_{\varepsilon })_{\varepsilon \in E}$ is weakly $\Sigma $-convergent in $%
L^{r}(\Omega )$ to $u_{0}v_{0}$.
\end{theorem}

As a consequence of the above theorem the following holds.

\begin{corollary}
\label{c3.1}Let $(u_{\varepsilon })_{\varepsilon \in E}\subset L^{p}(\Omega
) $ and $(v_{\varepsilon })_{\varepsilon \in E}\subset L^{p^{\prime
}}(\Omega )\cap L^{\infty }(\Omega )$ ($1<p<\infty $ and $p^{\prime
}=p/(p-1) $) be two sequences such that: \emph{(i)} $u_{\varepsilon
}\rightarrow u_{0}$ in $L^{p}(\Omega )$-weak $\Sigma $; \emph{(ii)} $%
v_{\varepsilon }\rightarrow v_{0}$ in $L^{p^{\prime }}(\Omega )$-strong $%
\Sigma $; \emph{(iii)} $(v_{\varepsilon })_{\varepsilon \in E}$ is bounded
in $L^{\infty }(\Omega )$. Then $u_{\varepsilon }v_{\varepsilon }\rightarrow
u_{0}v_{0}$ in $L^{p}(\Omega )$-weak $\Sigma $.
\end{corollary}

Now, assume that the algebra $A$ is introverted. Then its spectrum is a
compact topological semigroup whose kernel is a compact topological group,
so that we can define, as in the preceding subsection, the convolution over $%
\Delta (A)$. Our aim in the next result is to link the $\Sigma $-convergence
concept to the convolution over the spectrum $\Delta (A)$ of $A$. To see
this, let $p,q,m\geq 1$ be real numbers such that $\frac{1}{p}+\frac{1}{q}=1+%
\frac{1}{m}$. Let $(u_{\varepsilon })_{\varepsilon >0}\subset L^{p}(\Omega )$
and $(v_{\varepsilon })_{\varepsilon >0}\subset L^{q}(\mathbb{R}^{N})$ be
two sequences. One may view $u_{\varepsilon }$ as defined in the whole $%
\mathbb{R}^{N}$ by taking its extension by zero outside $\Omega $. Define 
\begin{equation*}
(u_{\varepsilon }\ast v_{\varepsilon })(x)=\int_{\mathbb{R}%
^{N}}u_{\varepsilon }(t)v_{\varepsilon }(x-t)\;dt\ \ (x\in \mathbb{R}^{N}),
\end{equation*}%
which lies in $L^{m}(\mathbb{R}^{N})$. Then

\begin{theorem}[{\protect\cite[Theorem 6.2]{NA2014}}]
\label{t3.4}Let $(u_{\varepsilon })_{\varepsilon >0}$ and $(v_{\varepsilon
})_{\varepsilon >0}$ be as above. Assume that, as $\varepsilon \rightarrow 0$%
, $u_{\varepsilon }\rightarrow u_{0}$ in $L^{p}(\Omega )$-weak $\Sigma $ and 
$v_{\varepsilon }\rightarrow v_{0}$ in $L^{q}(\mathbb{R}^{N})$-strong $%
\Sigma $, where $u_{0}$ and $v_{0}$ are in $L^{p}(\Omega ;\mathcal{B}%
_{A}^{p}(\mathbb{R}^{N}))$ and $L^{q}(\mathbb{R}^{N};\mathcal{B}_{A}^{q}(%
\mathbb{R}^{N}))$ respectively. Assume further that the algebra wmv $A$ is
introverted. Then, as $\varepsilon \rightarrow 0$, 
\begin{equation*}
u_{\varepsilon }\ast v_{\varepsilon }\rightarrow u_{0}\ast \ast v_{0}%
\mbox{
in }L^{m}(\Omega )\mbox{-weak }\Sigma \mbox{.}
\end{equation*}
\end{theorem}

In practice, one deals with the evolutionary version of the concept of $%
\Sigma $-convergence. Such concept requires some further notions such as
those related to the product of algebras with mean value. Let $A_{y}$ (resp. 
$A_{\tau }$) be an algebra with mean value on $\mathbb{R}_{y}^{N}$ (resp. $%
\mathbb{R}_{\tau }$). We define their product denoted by $A=A_{\tau }\odot
A_{y}$ as the closure in $BUC(\mathbb{R}_{y,\tau }^{N+1})$ of the tensor
product $A_{\tau }\otimes A_{y}=\{\sum_{\mbox{finite}}u_{i}\otimes
v_{i}:u_{i}\in A_{\tau },\;v_{i}\in A_{y}\}$. It is a well known fact that $%
A_{y}\odot A_{\tau }$ is an algebra with mean value on $\mathbb{R}^{N+1}$ (
see e.g. \cite{Hom1, CMP}).

With this in mind, let $A=A_{\tau }\odot A_{y}$ be as above. The same letter 
$\mathcal{G}$ will denote the Gelfand transformation on $A_{y}$, $A_{\tau }$
and $A$, as well. Points in $\Delta (A_{y})$ (resp. $\Delta (A_{\tau })$)
are denoted by $s$ (resp. $s_{0}$). The compact space $\Delta (A_{y})$
(resp. $\Delta (A_{\tau })$) is equipped with the $M$-measure $\beta _{y}$
(resp. $\beta _{\tau }$), for $A_{y}$ (resp. $A_{\tau }$). We have $\Delta
(A)=\Delta (A_{\tau })\times \Delta (A_{y})$ (Cartesian product) and the $M$%
-measure for $A$, with which $\Delta (A)$ is equipped, is precisely the
product measure $\beta =\beta _{\tau }\otimes \beta _{y}$ (see \cite{Hom1}).
Finally, let $0<T<\infty $. We set $Q=\left( 0,T\right) \times \Omega $ as
in Section 1 (an open cylinder in $\mathbb{R}^{N+1}$) and $\mathcal{O}%
=Q\times \mathbb{R}^{N}$.

This being so, a sequence $\left( u_{\varepsilon }\right) _{\varepsilon
>0}\subset L^{p}\left( Q\right) $\ $(1\leq p<\infty )$\ is said to weakly $%
\Sigma $-converge\ in $L^{p}\left( Q\right) $\ to some $u_{0}\in L^{p}(Q;%
\mathcal{B}_{A}^{p}(\mathbb{R}^{N+1}))$\ if as $\varepsilon \rightarrow 0$, 
\begin{equation*}
\int_{Q}u_{\varepsilon }\left( t,x\right) f\left( t,x,\frac{t}{\varepsilon },%
\frac{x}{\varepsilon }\right) dxdt\rightarrow \iint_{Q\times \Delta (A)}%
\widehat{u}_{0}\left( t,x,s_{0},s\right) \widehat{f}\left(
t,x,s_{0},s\right)\; dxdtd\beta
\end{equation*}%
for all $f\in L^{p^{\prime }}\left( Q;A\right) $. We may also define the
weak $\Sigma $-convergence in $L^{p}\left( \mathcal{O}\right) $ as follows: $%
\left( u_{\varepsilon }\right) _{\varepsilon >0}\subset L^{p}\left( \mathcal{%
O}\right) $ weakly $\Sigma $-converges to $u_{0}\in L^{p}(\mathcal{O};%
\mathcal{B}_{A}^{p}(\mathbb{R}^{N+1}))$ if 
\begin{equation*}
\int_{\mathcal{O}}u_{\varepsilon }\left( t,x,v\right) f\left( t,x,\frac{t}{%
\varepsilon },\frac{x}{\varepsilon },v\right)\; dxdtdv\rightarrow \iint_{%
\mathcal{O}\times \Delta (A)}\widehat{u}_{0}\left( t,x,s_{0},s,v\right) 
\widehat{f}\left( t,x,s_{0},s,v\right)\; dxdtdvd\beta
\end{equation*}%
for any $f\in L^{p^{\prime }}\left( \mathcal{O};A\right) $.

\begin{remark}
\label{r3.1}\emph{The conclusions of Theorems \ref{t3.2}-\ref{t3.4} are
still valid mutatis mutandis in the present context (change }$\Omega $\emph{%
\ into }$Q$\emph{\ in Theorem \ref{t3.2}, }$W^{1,p}(\Omega )$\emph{\ into }$%
L^{p}(0,T;W^{1,p}(\Omega ))$\emph{, }$W^{1,p}(\Omega ;I_{A}^{p}(\mathbb{R}%
^{N}))\times L^{p}(\Omega ;\mathcal{B}_{A}^{1,p}(\mathbb{R}^{N}))$\emph{\
into }$L^{p}(0,T;W^{1,p}(\Omega ;I_{A}^{p}(\mathbb{R}^{N})))\times L^{p}(Q;%
\mathcal{B}_{A_{\tau }}^{p}(\mathbb{R}_{\tau };\mathcal{B}_{A_{y}}^{1,p}(%
\mathbb{R}^{N})))$\emph{), provided }$A$\emph{\ is introverted in Theorem %
\ref{t3.4}.}
\end{remark}

\section{Homogenization results}

Throughout this section, we consider the algebras wmv $A_{y}$ and $A_{\tau }$
to be as in the end of the preceding section. We further assume that $%
A_{\tau }$ is introverted.

With this in mind, let $(\boldsymbol{u}_{\varepsilon },f_{\varepsilon
})_{\varepsilon >0}$ be the sequence of solutions to (\ref{1.1})-(\ref{1.5}%
). In view of Lemma \ref{l2.1'}, there is a positive constant $C$
independent of $\varepsilon >0$ such that 
\begin{equation*}
\sup_{\varepsilon >0}\left\Vert \frac{\partial \boldsymbol{u}_{\varepsilon }%
}{\partial t}\right\Vert _{L^{2}(0,T;V^{\prime })}\leq C.
\end{equation*}%
This, together with the inequality (\ref{2.3''}) in Lemma \ref{l2.1'} entail
the precompactness of the sequence $(\boldsymbol{u}_{\varepsilon
})_{\varepsilon >0}$ in $L^{2}(0,T;H)$. Thus, given an ordinary sequence $E$%
, there are a subsequence $E^{\prime }$ of $E$ and a function $\boldsymbol{u}%
_{0}\in L^{2}(Q)^{N}$ such that, as $E^{\prime }\ni \varepsilon \rightarrow
0 $ 
\begin{equation}
\boldsymbol{u}_{\varepsilon }\rightarrow \boldsymbol{u}_{0}\mbox{ in }%
L^{2}(Q)^{N}\mbox{.}  \label{4.1}
\end{equation}%
In view of (\ref{2.3''}) and by the diagonal process, one can find a
subsequence of $(\boldsymbol{u}_{\varepsilon })_{\varepsilon \in E^{\prime
}} $ (not relabeled) which weakly converges in $L^{2}(0,T;V)$ to the
function $\boldsymbol{u}_{0}$ (this means that $\boldsymbol{u}_{0}\in
L^{2}(0,T;V)$). From Theorem \ref{t3.3}, we infer the existence of a
function $\boldsymbol{u}_{1}=(u_{1}^{k})_{1\leq k\leq N}\in L^{2}(Q;\mathcal{%
B}_{A_{\tau }}^{2}(\mathbb{R}_{\tau };\mathcal{B}_{\#A_{y}}^{1,2}(\mathbb{R}%
_{y}^{N}))^{N})$ such that the convergence result 
\begin{equation}
\frac{\partial \boldsymbol{u}_{\varepsilon }}{\partial x_{i}}\rightarrow 
\frac{\partial \boldsymbol{u}_{0}}{\partial x_{i}}+\frac{\overline{\partial }%
\boldsymbol{u}_{1}}{\partial y_{i}}\mbox{ in }L^{2}(Q)^{N}\mbox{-weak }%
\Sigma \mbox{ }(1\leq i\leq N)  \label{4.2}
\end{equation}%
holds when $E^{\prime }\ni \varepsilon \rightarrow 0$. Still from Lemma \ref%
{l2.1'} (see (\ref{2.4''}) for $m=2$ and (\ref{2.5''}) therein) there exist
a subsequence of $E^{\prime }$ (still denoted by $E^{\prime }$) and two
functions $f_{0}\in L^{\infty }(0,T;L^{2}(\Omega \times \mathbb{R}^{N};%
\mathcal{B}_{A}^{2}(\mathbb{R}^{N+1})))$, $p\in L^{2}(Q;\mathcal{B}_{A}^{2}(%
\mathbb{R}^{N+1}))$ with $\int_{\Omega }pdx=0$ such that, as $E^{\prime }\ni
\varepsilon \rightarrow 0$, 
\begin{equation}
f_{\varepsilon }\rightarrow f_{0}\mbox{ in }L^{2}(Q\times \mathbb{R}^{N})%
\mbox{-weak }\Sigma  \label{4.3}
\end{equation}%
and 
\begin{equation}
p_{\varepsilon }\rightarrow p\mbox{ in }L^{2}(Q)\mbox{-weak }\Sigma .
\label{4.4}
\end{equation}%
We recall that $\frac{\partial \boldsymbol{u}_{0}}{\partial x_{i}}=\left( 
\frac{\partial u_{0}^{k}}{\partial x_{i}}\right) _{1\leq k\leq N}$ ($%
\boldsymbol{u}_{0}=(u_{0}^{k})_{1\leq k\leq N}$) and $\frac{\overline{%
\partial }\boldsymbol{u}_{1}}{\partial y_{i}}=\left( \frac{\overline{%
\partial }u_{1}^{k}}{\partial y_{i}}\right) _{1\leq k\leq N}$.

Our goal in this section is the study of the asymptotics (as $\varepsilon
\rightarrow 0$) of $(\boldsymbol{u}_{\varepsilon },f_{\varepsilon
},p_{\varepsilon })_{\varepsilon >0}$ under the following additional
assumption

\begin{itemize}
\item[(\textbf{A3})] $A_{i}(t,x,\cdot ,\cdot )\in \left[ B_{A}^{2}(\mathbb{R}%
_{y,\tau }^{N+1})\right] ^{N^{2}}$ for $i=0,1$ and for all $(t,x)\in 
\overline{Q}$.
\end{itemize}

\subsection{Passing to the limit $\protect\varepsilon \rightarrow 0$}

Let us first find the equation satisfied by $f_{0}$. To that end, let $\phi
\in \mathcal{C}_{0}^{\infty }(\mathcal{O})\otimes A^{\infty }$ (where we
recall that $\mathcal{O}=Q\times \mathbb{R}_{v}^{N}$ with $Q=(0,T)\times
\Omega $) and define $\phi ^{\varepsilon }\in \mathcal{C}_{0}^{\infty }(%
\mathcal{O})$ by $\phi ^{\varepsilon }(t,x,v)=\phi (t,x,t/\varepsilon
,x/\varepsilon ,v)$ for $(t,x,v)\in \mathcal{O}$. Multiplying the Vlasov
equation (\ref{1.1}) by $\phi ^{\varepsilon }$ and integrating by parts, 
\begin{equation*}
-\int_{\mathcal{O}}f_{\varepsilon }\left[ \frac{\partial \phi ^{\varepsilon }%
}{\partial t}+\varepsilon v\cdot \nabla \phi ^{\varepsilon }+(\boldsymbol{u}%
_{\varepsilon }-v)\cdot \nabla _{v}\phi ^{\varepsilon }\right] \;dxdtdv=0.
\end{equation*}%
The above equation is equivalent to the following one 
\begin{equation}
-\int_{\mathcal{O}}f_{\varepsilon }\left[ \left( \frac{\partial \phi }{%
\partial t}\right) ^{\varepsilon }+\frac{1}{\varepsilon }\left( \frac{%
\partial \phi }{\partial \tau }\right) ^{\varepsilon }+\varepsilon v\cdot
\left( \nabla \phi \right) ^{\varepsilon }+v\cdot \left( \nabla _{y}\phi
\right) ^{\varepsilon }+(\boldsymbol{u}_{\varepsilon }-v)\cdot \left( \nabla
_{v}\phi \right) ^{\varepsilon }\right]\; dxdtdv=0.  \label{4.5}
\end{equation}%
Multiplying the above equation by $\varepsilon $ and letting $E^{\prime }\ni
\varepsilon \rightarrow 0$ (where $E^{\prime }$ is as above), we end up with
(using the fact that the sequence $(f_{\varepsilon }(\boldsymbol{u}%
_{\varepsilon }-v))_{\varepsilon >0}$ is bounded in $L^{1}(\mathcal{O})$;
see Remark \ref{r2.3}) 
\begin{equation*}
\iint_{\mathcal{O}\times \Delta (A)}\widehat{f}_{0}\partial _{0}\widehat{%
\phi }\;dxdtdvd\beta =0
\end{equation*}%
where $\partial _{0}=\mathcal{G}_{1}\circ \frac{\overline{\partial }}{%
\partial \tau }$. It follows that $\frac{\overline{\partial }f_{0}}{\partial
\tau }=0$, which amounts to say that $f_{0}$ does not depends on $\tau $.
Indeed this is equivalent to $f_{0}\in L^{\infty }(0,T;L^{2}(\Omega \times 
\mathbb{R}^{N};\mathcal{B}_{A_{y}}^{2}(\mathbb{R}^{N};I_{A_{\tau }}^{2}(%
\mathbb{R}_{\tau }))))$, and since $A_{\tau }$ is introverted, it is ergodic 
\cite[Remark 3.3]{NA2014}, so that $I_{A_{\tau }}^{2}(\mathbb{R}_{\tau })$
consists of constants. This means that the test functions $\phi $ may be
chosen independent of $\tau \in \mathbb{R}$, that is, $\phi \in \mathcal{C}%
_{0}^{\infty }(\mathcal{O})\otimes A_{y}^{\infty }$ and so, $\phi
^{\varepsilon }(t,x,v)=\phi (t,x,x/\varepsilon ,v)$ for $(t,x,v)\in \mathcal{%
O}$. Before we can pass to the limit in (\ref{4.5}), we notice that the
function $(t,x,v,y)\mapsto v\cdot \nabla _{v}\phi $ lies in $L^{\infty }(%
\mathcal{O};A_{y})$ since it trivially lies in $\mathcal{C}_{0}^{\infty }(%
\mathcal{O})\otimes A_{y}^{\infty }$. Thus, 
\begin{equation}
\int_{\mathcal{O}}f_{\varepsilon }v\cdot \left( \nabla _{v}\phi \right)
^{\varepsilon }dxdtdv\rightarrow \iint_{\mathcal{O}\times \Delta (A_{y})}%
\widehat{f}_{0}v\cdot \nabla _{v}\widehat{\phi }\;dxdtdvd\beta _{y}.
\label{4.0}
\end{equation}

In order to pass to the limit in the term $\int_{\mathcal{O}}f_{\varepsilon }%
\boldsymbol{u}_{\varepsilon }\cdot \left( \nabla _{v}\phi \right)
^{\varepsilon }\;dxdtdv$, we need the following

\begin{lemma}
\label{l4.1}Let $\boldsymbol{u}_{0}$ and $f_{0}$ be as in \emph{(\ref{4.1})}
and \emph{(\ref{4.3})}, respectively. Then for any $\boldsymbol{\psi }\in (%
\mathcal{C}_{0}^{\infty }(\mathcal{O})\otimes A_{y}^{\infty })^{N}$, 
\begin{equation}
\int_{\mathcal{O}}f_{\varepsilon }\boldsymbol{u}_{\varepsilon }\cdot 
\boldsymbol{\psi }^{\varepsilon }dxdtdv\rightarrow \iint_{\mathcal{O}\times
\Delta (A_{y})}\widehat{f}_{0}\boldsymbol{u}_{0}\cdot \widehat{\boldsymbol{%
\psi }}\;dxdtdvd\beta _{y}  \label{4.0'}
\end{equation}%
as $E^{\prime }\ni \varepsilon \rightarrow 0$.
\end{lemma}

\begin{proof}
First assume $\boldsymbol{\psi }=(\psi _{i})_{1\leq i\leq N}$ with $\psi
_{i}=\varphi _{i}\otimes \chi _{i}\otimes w_{i}$ with $\varphi _{i}\in 
\mathcal{C}_{0}^{\infty }(Q)$, $\chi _{i}\in \mathcal{C}_{0}^{\infty }(%
\mathbb{R}_{v}^{N})$ and $w_{i}\in A_{y}^{\infty }$. Then 
\begin{equation*}
\int_{\mathcal{O}}f_{\varepsilon }\boldsymbol{u}_{\varepsilon }\cdot 
\boldsymbol{\psi }^{\varepsilon }dxdtdv=\sum_{i=1}^{N}\int_{\mathcal{O}%
}f_{\varepsilon }(t,x,v)u_{\varepsilon }^{i}(t,x)\chi _{i}(v)\varphi
_{i}(t,x)w_{i}\left( \frac{x}{\varepsilon }\right) \;dxdtdv.
\end{equation*}%
Set $U_{\varepsilon }^{i}(t,x,v)=u_{\varepsilon }^{i}(t,x)\chi _{i}(v)$ for $%
(t,x,v)\in \mathcal{O}$. Then 
\begin{equation*}
U_{\varepsilon }^{i}\rightarrow U_{0}^{i}\equiv u_{0}^{i}\otimes \chi _{i}%
\mbox{ in }L^{2}(\mathcal{O})\mbox{-strong as }E^{\prime }\ni \varepsilon
\rightarrow 0.
\end{equation*}%
Indeed, since 
\begin{equation*}
\int_{\mathcal{O}}\left\vert U_{\varepsilon }^{i}-U_{0}^{i}\right\vert
^{2}dxdtdv\leq \int_{\mathbb{R}_{v}^{N}}\left\vert \chi _{i}(v)\right\vert
^{2}\;dv\int_{Q}\left\vert u_{\varepsilon }^{i}-u_{0}^{i}\right\vert
^{2}\;dxdt,
\end{equation*}%
the claim follows from Eq. (\ref{4.1}). Thus, we infer from Eq. (\ref{4.3})
and Theorem \ref{t3.5} that 
\begin{equation*}
f_{\varepsilon }U_{\varepsilon }^{i}\rightarrow f_{0}U_{0}^{i}\mbox{ in }%
L^{1}(\mathcal{O})\mbox{-weak }\Sigma \mbox{ as }E^{\prime }\ni \varepsilon
\rightarrow 0.
\end{equation*}%
Hence, by choosing the special test function $\varphi _{i}\otimes
w_{i}\otimes 1_{\mathbb{R}_{v}^{N}}\in L^{\infty }(\mathcal{O};A_{y})$, we
are led to 
\begin{equation*}
\int_{\mathcal{O}}f_{\varepsilon }(t,x,v)u_{\varepsilon }^{i}(t,x)\chi
_{i}(v)\varphi _{i}(t,x)w_{i}\left( \frac{x}{\varepsilon }\right)\;
dxdtdv\rightarrow \iint_{\mathcal{O}\times \Delta (A_{y})}\widehat{f}%
_{0}u_{0}^{i}\chi _{i}\varphi _{i}\widehat{w}_{i}\;dxdtdvd\beta _{y},
\end{equation*}%
or, 
\begin{equation*}
\int_{\mathcal{O}}f_{\varepsilon }\boldsymbol{u}_{\varepsilon }\cdot 
\boldsymbol{\psi }^{\varepsilon }dxdtdv\rightarrow \iint_{\mathcal{O}\times
\Delta (A_{y})}\widehat{f}_{0}\boldsymbol{u}_{0}\cdot \widehat{\boldsymbol{%
\psi }}\;dxdtdvd\beta _{y}.
\end{equation*}%
Now, by some routine computations, the result follows at once from the
density of $\mathcal{C}_{0}^{\infty }(Q)\otimes \mathcal{C}_{0}^{\infty }(%
\mathbb{R}_{v}^{N})\otimes A_{y}^{\infty }$ in $\mathcal{C}_{0}^{\infty }(%
\mathcal{O})\otimes A_{y}^{\infty }$.
\end{proof}

As a consequence of Lemma \ref{l4.1}, we have 
\begin{equation*}
\int_{\mathcal{O}}f_{\varepsilon }\boldsymbol{u}_{\varepsilon }\cdot \left(
\nabla _{v}\phi \right) ^{\varepsilon }\;dxdtdv\rightarrow \iint_{\mathcal{O}%
\times \Delta (A_{y})}\widehat{f}_{0}\boldsymbol{u}_{0}\cdot \widehat{\nabla
_{v}\phi }\;dxdtdvd\beta _{y}.
\end{equation*}

Returning to (\ref{4.5}) and taking the limit when $E^{\prime }\ni
\varepsilon \rightarrow 0$, we arrive at 
\begin{equation*}
-\iint_{\mathcal{O}\times \Delta (A_{y})}\widehat{f}_{0}\left[ \frac{%
\partial \widehat{\phi }}{\partial t}+v\cdot \partial \widehat{\phi }+(%
\boldsymbol{u}_{0}-v)\cdot \nabla _{v}\widehat{\phi }\right] \;dxdtdvd\beta
_{y}=0,
\end{equation*}%
where $\partial \widehat{\phi }=\mathcal{G}\circ \nabla _{y}\phi $. This
gives rise to the following equation satisfied by $f_{0}$: 
\begin{equation}
\frac{\partial f_{0}}{\partial t}+v\cdot \overline{\nabla }_{y}f_{0}+\Div%
_{v}\left( (\boldsymbol{u}_{0}-v)f_{0}\right) =0\mbox{ in }\mathcal{O}\times 
\mathbb{R}_{y}^{N}.  \label{4.6}
\end{equation}%
Following the lines of \cite[Section 4]{SM} we prove, by choosing suitable
test functions, that the function $f_{0}$ satisfies the following reflection
boundary and initial conditions 
\begin{equation}
f_{0}(t,x,y,v)=f_{0}(t,x,y,v^{\ast })\mbox{ for }x\in \partial \Omega 
\mbox{
with }v\cdot \nu (x)<0  \label{4.6'}
\end{equation}%
where $v^{\ast }=v-2(v\cdot \nu (x))\nu (x)$, 
\begin{equation}
f_{0}(0,x,y,v)=f^{0}(x,v)\mbox{ for }(x,y,v)\in \Omega \times \mathbb{R}%
_{y}^{N}\times \mathbb{R}_{v}^{N}  \label{4.6''}
\end{equation}

We consider now the Stokes system (\ref{1.2})-(\ref{1.3}). Choosing $%
\boldsymbol{\psi }_{0}=(\psi _{0}^{k})_{1\leq k\leq N}\in \mathcal{C}%
_{0}^{\infty }(Q)^{N}$ and $\boldsymbol{\psi }_{1}=(\psi _{1}^{k})_{1\leq
k\leq N}\in \lbrack \mathcal{C}_{0}^{\infty }(Q)\otimes A^{\infty }]^{N}$,
we set $\boldsymbol{\Phi }=(\boldsymbol{\psi }_{0},\boldsymbol{\psi }_{1})$
and define $\boldsymbol{\Phi }_{\varepsilon }=\boldsymbol{\psi }%
_{0}+\varepsilon \boldsymbol{\psi }_{1}^{\varepsilon }$ by 
\begin{equation*}
\boldsymbol{\Phi }_{\varepsilon }(t,x)=\boldsymbol{\psi }_{0}(t,x)+%
\varepsilon \boldsymbol{\psi }_{1}\left( t,x,\frac{t}{\varepsilon },\frac{x}{%
\varepsilon }\right) \mbox{ for\ }(t,x)\in Q\mbox{.}
\end{equation*}%
It can be checked that $\boldsymbol{\Phi }_{\varepsilon }\in \mathcal{C}%
_{0}^{\infty }(Q)^{N}$. By plugging $\boldsymbol{\Phi }_{\varepsilon }$ into
the variational formulation of (\ref{1.2}), we obtain 
\begin{equation}
\begin{array}{l}
-\int_{Q}\boldsymbol{u}_{\varepsilon }\cdot \frac{\partial \boldsymbol{\Phi }%
_{\varepsilon }}{\partial t}dxdt+\int_{Q}A_{0}^{\varepsilon }\nabla 
\boldsymbol{u}_{\varepsilon }\cdot \nabla \boldsymbol{\Phi }_{\varepsilon
}dxdt+\int_{Q}(A_{1}^{\varepsilon }\ast \nabla \boldsymbol{u}_{\varepsilon
})\cdot \nabla \boldsymbol{\Phi }_{\varepsilon }\;dxdt \\ 
\ \ \ \ \ \ -\int_{Q}p_{\varepsilon }\Div\boldsymbol{\Phi }_{\varepsilon
}dxdt=-\int_{\mathcal{O}}f_{\varepsilon }(\boldsymbol{u}_{\varepsilon
}-v)\cdot \boldsymbol{\Phi }_{\varepsilon }\;dxdtdv.%
\end{array}
\label{4.7}
\end{equation}%
Our immediate goal is to pass to the limit in the above equation. We deal
with its constituents in turn. Owing to Eq. (\ref{4.1}), as $E^{\prime }\ni
\varepsilon \rightarrow 0$ in the first term in the left-hand side of Eq. (%
\ref{4.7}), we have, 
\begin{equation*}
\int_{Q}\boldsymbol{u}_{\varepsilon }\cdot \frac{\partial \boldsymbol{\Phi }%
_{\varepsilon }}{\partial t}dxdt\rightarrow \int_{Q}\boldsymbol{u}_{0}\cdot 
\frac{\partial \boldsymbol{\psi }_{0}}{\partial t} \;dxdt.
\end{equation*}%
For the next term, one easily shows that as $\varepsilon \rightarrow 0$, 
\begin{equation*}
\frac{\partial \boldsymbol{\Phi }_{\varepsilon }}{\partial x_{i}}\rightarrow 
\frac{\partial \boldsymbol{\psi }_{0}}{\partial x_{i}}+\frac{\partial 
\boldsymbol{\psi }_{1}}{\partial y_{i}}\mbox{ in }L^{2}(Q)^{N}\mbox{-strong }%
\Sigma \ \ (1\leq i\leq N).
\end{equation*}%
Combining the above convergence result with Eq. (\ref{4.2}), we deduce from
Corollary \ref{c3.1} that, as $E^{\prime }\ni \varepsilon \rightarrow 0$, 
\begin{equation*}
\frac{\partial \boldsymbol{u}_{\varepsilon }}{\partial x_{j}}\cdot \frac{%
\partial \boldsymbol{\Phi }_{\varepsilon }}{\partial x_{i}}\rightarrow
\left( \frac{\partial \boldsymbol{u}_{0}}{\partial x_{j}}+\frac{\overline{%
\partial }\boldsymbol{u}_{1}}{\partial y_{j}}\right) \cdot \left( \frac{%
\partial \boldsymbol{\psi }_{0}}{\partial x_{i}}+\frac{\partial \boldsymbol{%
\psi }_{1}}{\partial y_{i}}\right) \mbox{ in }L^{2}(Q)\mbox{-weak }\Sigma .
\end{equation*}%
Passing to the limit in the above mentioned term using $A_{0}$ as a test
function (recall that $A_{0}\in \mathcal{C}(\overline{Q};[B_{A}^{2}(\mathbb{R%
}^{N+1})\cap L^{\infty }(\mathbb{R}^{N+1})]^{N^{2}})$ by assumption (\textbf{%
A3}) so that in view of \cite[Proposition 8]{DPDE}, it is an admissible test
function in the sense of \cite[Definition 5]{DPDE}), we get 
\begin{equation*}
\int_{Q}A_{0}^{\varepsilon }\nabla \boldsymbol{u}_{\varepsilon }\cdot \nabla 
\boldsymbol{\Phi }_{\varepsilon }dxdt\rightarrow \iint_{Q\times \mathcal{K}}%
\widehat{A}_{0}\mathbb{D}\boldsymbol{u}\cdot \mathbb{D}\boldsymbol{\Phi }%
dxdtd\beta \mbox{ as }E^{\prime }\ni \varepsilon \rightarrow 0
\end{equation*}%
where, setting $\boldsymbol{u}=(\boldsymbol{u}_{0},\boldsymbol{u}_{1})$, we
have $\mathbb{D}\boldsymbol{u}=(\mathbb{D}_{j}\boldsymbol{u})_{1\leq j\leq
N} $ with $\mathbb{D}_{j}\boldsymbol{u}=(\mathbb{D}_{j}\boldsymbol{u}%
^{k})_{1\leq k\leq N}$ and $\mathbb{D}_{j}\boldsymbol{u}^{k}=\frac{\partial
u_{0}^{k}}{\partial x_{j}}+\partial _{j}\widehat{u}_{1}^{k}$ ($\partial _{j}%
\widehat{u}_{1}^{k}=\mathcal{G}_{1}\left( \overline{\partial }%
u_{1}^{k}/\partial y_{j}\right) $), and the same definition for $\mathbb{D}%
\boldsymbol{\Phi }$.

Let us now tackle the term involving convolution. First we know that $%
A_{1}^{\varepsilon }\rightarrow A_{1}$ in $L^{1}(\mathbb{R}^{N+1})$-strong $%
\Sigma $ and $\nabla \boldsymbol{u}_{\varepsilon }\rightarrow \nabla 
\boldsymbol{u}_{0}+\overline{\nabla }_{y}\boldsymbol{u}_{1}$ in $%
L^{2}(Q)^{N^{2}}$-weak $\Sigma $; hence by virtue of Theorem \ref{t3.4}, we
conclude that 
\begin{equation*}
A_{1}^{\varepsilon }\ast \nabla \boldsymbol{u}_{\varepsilon }\rightarrow
A_{1}\ast \ast (\nabla \boldsymbol{u}_{0}+\overline{\nabla }_{y}\boldsymbol{u%
}_{1})\mbox{ in }L^{2}(Q)^{N^{2}}\mbox{-weak }\Sigma \mbox{ as }E^{\prime
}\ni \varepsilon \rightarrow 0
\end{equation*}%
where the double convolution is defined with respect to the time variable as
follows: 
\begin{equation*}
\left( \widehat{A}_{1}\ast \ast (\widehat{\nabla \boldsymbol{u}_{0}+%
\overline{\nabla }_{y}\boldsymbol{u}_{1}})\right)
(t,x,s_{0},s)=\int_{0}^{t}d\tau \int_{K(A_{\tau })}\widehat{A}_{1}(\tau
,x,r_{0},s)(\nabla \boldsymbol{u}_{0}+\partial \widehat{\boldsymbol{u}}%
_{1})(t-\tau ,x,s_{0}-r_{0},s)\;d\beta _{\tau }(r_{0}),
\end{equation*}%
in which the function $\nabla \boldsymbol{u}_{0}$ is assumed to be defined
on the whole of $\mathbb{R}^{N}$ by taking its zero-extension off $\Omega $.
Therefore, repeating the same reasoning as for the preceding term, we arrive
at (as $E^{\prime }\ni \varepsilon \rightarrow 0$) 
\begin{equation*}
\int_{Q}(A_{1}^{\varepsilon }\ast \nabla \boldsymbol{u}_{\varepsilon })\cdot
\nabla \boldsymbol{\Phi }_{\varepsilon }dxdt\rightarrow \iint_{Q\times
\Delta (A_{y})\times K(A_{\tau })}(\widehat{A_{1}}\ast \ast \mathbb{D}%
\boldsymbol{u})\cdot \mathbb{D}\boldsymbol{\Phi }\; dxdtd\beta ,
\end{equation*}%
or, as $\beta _{\tau }$ is supported by $K(A_{\tau })$, 
\begin{equation*}
\int_{Q}(A_{1}^{\varepsilon }\ast \nabla \boldsymbol{u}_{\varepsilon })\cdot
\nabla \boldsymbol{\Phi }_{\varepsilon }dxdt\rightarrow \iint_{Q\times
\Delta (A)}(\widehat{A_{1}}\ast \ast \mathbb{D}\boldsymbol{u})\cdot \mathbb{D%
}\boldsymbol{\Phi }\;dxdtd\beta .
\end{equation*}%
As for the term with the pressure, we have 
\begin{eqnarray}
\int_{Q}p_{\varepsilon }\Div\boldsymbol{\Phi }_{\varepsilon }\;dxdt
&=&\int_{Q}p_{\varepsilon }\Div\boldsymbol{\psi }_{0}\;dxdt+\int_{Q}p_{%
\varepsilon }(\Div_{y}\boldsymbol{\psi }_{1})^{\varepsilon }\;dxdt
\label{4.8} \\
&&+\varepsilon \int_{Q}p_{\varepsilon }(\Div\boldsymbol{\psi }%
_{1})^{\varepsilon }\;dxdt.  \notag
\end{eqnarray}%
Set $p_{0}(x,t)=\int_{\Delta (A)}\widehat{p}(t,x,s_{0},s)\;d\beta $ for a.e. 
$(t,x)\in Q$, where $p$ is as in (\ref{4.4}). Then we know that $%
p_{\varepsilon }\rightarrow p_{0}$ in $L^{2}(Q)$-weak, and, passing to the
limit in (\ref{4.8}) as $E^{\prime }\ni \varepsilon \rightarrow 0$ yields 
\begin{equation*}
\int_{Q}p_{\varepsilon }\Div\boldsymbol{\Phi }_{\varepsilon
}\;dxdt\rightarrow \int_{Q}p_{0}\Div\boldsymbol{\psi }_{0}dxdt+\iint_{Q%
\times \Delta (A)}\widehat{p}\widehat{\Div_{y}\boldsymbol{\psi }_{1}}%
\;dxdtd\beta .
\end{equation*}%
Finally, for the term in the right-hand side of (\ref{4.7}), reasoning as in
the proof of (\ref{4.0}) and (\ref{4.0'}), we get 
\begin{equation*}
\int_{\mathcal{O}}f_{\varepsilon }(\boldsymbol{u}_{\varepsilon }-v)\cdot 
\boldsymbol{\Phi }_{\varepsilon }\;dxdtdv\rightarrow \iint_{\mathcal{O}%
\times \Delta (A)}\widehat{f}_{0}(\boldsymbol{u}_{0}-v)\cdot \boldsymbol{%
\psi }_{0}\;dxdtd\beta dv.
\end{equation*}

Putting together the previous convergence results, we are led to the fact
that the quadruple $(\boldsymbol{u}_{0},\boldsymbol{u}_{1},f_{0},p)$
determined by (\ref{4.1})-(\ref{4.4}) solves the system consisting of
equation (\ref{4.6}) and 
\begin{equation}
\left\{ 
\begin{array}{l}
-\int_{Q}\boldsymbol{u}_{0}\cdot \boldsymbol{\psi }_{0}^{\prime
}\;dxdt+\iint_{Q\times \Delta (A)}\left( \widehat{A}_{1}\mathbb{D}%
\boldsymbol{u}+\widehat{A}_{1}\ast \ast \mathbb{D}\boldsymbol{u}\right)
\cdot \mathbb{D}\boldsymbol{\Phi }\;dxdtd\beta \\ 
\ \ -\int_{Q}p_{0}\Div\boldsymbol{\psi }_{0}dxdt-\iint_{Q\times \Delta (A)}%
\widehat{p}\widehat{\Div_{y}\boldsymbol{\psi }_{1}}dxdtd\beta =-\iint_{%
\mathcal{O}\times \Delta (A)}\widehat{f}_{0}(\boldsymbol{u}_{0}-v)\cdot 
\boldsymbol{\psi }_{0}\;dxdtd\beta dv, \\ 
\mbox{ for all }\boldsymbol{\Phi }=(\boldsymbol{\psi }_{0},\boldsymbol{\psi }%
_{1})\in \mathcal{C}_{0}^{\infty }(Q)^{N}\times \lbrack \mathcal{C}%
_{0}^{\infty }(Q)\otimes A^{\infty }]^{N}.%
\end{array}%
\right.  \label{4.9}
\end{equation}

From the equality $\Div\boldsymbol{u}_{\varepsilon }=0$ we easily deduce
that $\overline{\Div}_{y}\boldsymbol{u}_{1}=0$. Next, we need to uncouple
Eq. (\ref{4.9}), which is equivalent to the system (\ref{3.11})-(\ref{3.12})
below: 
\begin{equation}
\left\{ 
\begin{array}{l}
\iint_{Q\times \Delta (A)}\left( \widehat{A}_{0}\mathbb{D}\boldsymbol{u}+%
\widehat{A}_{1}\ast \ast \mathbb{D}\boldsymbol{u}\right) \cdot \partial 
\widehat{\boldsymbol{\psi }}_{1}\;dxdtd\beta -\iint_{Q\times \Delta (A)}%
\widehat{p}\widehat{\Div_{y}\boldsymbol{\psi }_{1}}\;dxdtd\beta =0 \\ 
\mbox{for all }\boldsymbol{\psi }_{1}\in \lbrack \mathcal{C}_{0}^{\infty
}(Q)\otimes A^{\infty }]^{N}%
\end{array}%
\right.  \label{3.11}
\end{equation}%
and 
\begin{equation}
\left\{ 
\begin{array}{l}
-\int_{Q}\boldsymbol{u}_{0}\cdot \boldsymbol{\psi }_{0}^{\prime
}dxdt+\iint_{Q\times \Delta (A)}\left( \widehat{A}_{0}\mathbb{D}\boldsymbol{u%
}+\widehat{A}_{1}\ast \ast \mathbb{D}\boldsymbol{u}\right) \cdot \nabla 
\boldsymbol{\psi }_{0}\;dxdtd\beta \\ 
\ \ -\int_{Q}p_{0}\Div\boldsymbol{\psi }_{0}dxdt=-\iint_{\mathcal{O}\times
\Delta (A)}\widehat{f}_{0}(\boldsymbol{u}_{0}-v)\cdot \boldsymbol{\psi }%
_{0}dxdtd\beta dv\mbox{ for all }\boldsymbol{\psi }_{0}\in \mathcal{C}%
_{0}^{\infty }(Q)^{N}\mbox{.}%
\end{array}%
\right.  \label{3.12}
\end{equation}%
For Eq. (\ref{3.11}), we choose $\boldsymbol{\psi }_{1}(x,t)=\varphi (x,t)%
\boldsymbol{w}$ with $\varphi \in \mathcal{C}_{0}^{\infty }(Q)$ and $%
\boldsymbol{w}\in (A^{\infty })^{N}$. Then (\ref{3.11}) becomes 
\begin{equation}
\left\{ 
\begin{array}{l}
\int_{\Delta (A)}\left( \widehat{A}_{0}\mathbb{D}\boldsymbol{u}+\widehat{A}%
_{1}\ast \ast \mathbb{D}\boldsymbol{u}\right) \cdot \partial \widehat{%
\boldsymbol{w}}d\beta -\int_{\Delta (A)}\widehat{p}\widehat{\Div}\widehat{%
\boldsymbol{w}}d\beta =0 \\ 
\mbox{for all }\boldsymbol{w}\in (A^{\infty })^{N}.%
\end{array}%
\right.  \label{3.13}
\end{equation}

Now, fix $\xi \in \mathbb{R}^{N\times N}$ and consider the following cell
problem: 
\begin{equation}
\left\{ 
\begin{array}{l}
\mbox{Find }u_{\xi }\in \mathcal{B}_{A_{\tau }}^{2}(\mathbb{R}_{\tau };%
\mathcal{B}_{\Div}^{1,2}(\mathbb{R}_{y}^{N}))\mbox{, }p_{\xi }\in \mathcal{B}%
_{A_{\tau }}^{2}(\mathbb{R}_{\tau };\mathcal{B}_{A_{y}}^{2}(\mathbb{R}%
_{y}^{N})/\mathbb{R})\mbox{ such that} \\ 
\int_{\Delta (A)}\left( \widehat{A}_{0}(\xi +\partial \widehat{u}_{\xi })+%
\widehat{A}_{1}\ast \ast (\xi +\partial \widehat{u}_{\xi })\right) \cdot
\partial \widehat{\boldsymbol{w}}d\beta -\int_{\Delta (A)}\widehat{p}_{\xi }%
\widehat{\Div}\widehat{\boldsymbol{w}}d\beta =0 \\ 
\mbox{for all }\boldsymbol{w}\in (A^{\infty })^{N}%
\end{array}%
\right.  \label{3.14}
\end{equation}%
where $\mathcal{B}_{\Div}^{1,2}(\mathbb{R}_{y}^{N})=\{\boldsymbol{v}\in 
\mathcal{B}_{\#A_{y}}^{1,2}(\mathbb{R}_{y}^{N})^{N}:\overline{\Div}_{y}%
\boldsymbol{v}=0\}$. Then Eq. (\ref{3.14}) is the variational formulation of
the problem 
\begin{equation}
\left\{ 
\begin{array}{l}
-\overline{\Div}_{y}\left( A_{0}\overline{\nabla }_{y}u_{\xi }+A_{1}\ast
\ast \overline{\nabla }_{y}u_{\xi }\right) +\overline{\nabla }_{y}p_{\xi }=%
\Div_{y}(A_{0}\xi +A_{1}\ast \ast \xi )\mbox{ in }\mathbb{R}_{y,\tau }^{N+1}
\\ 
\overline{\Div}_{y}u_{\xi }=0.%
\end{array}%
\right.  \label{3.15}
\end{equation}%
Thanks to the properties of the functions $A_{i}$ ($i=0,1$), the above
problem is classically solved and possesses a unique solution $(u_{\xi
},p_{\xi })\in \mathcal{B}_{A_{\tau }}^{2}(\mathbb{R}_{\tau };\mathcal{B}_{%
\Div}^{1,2}(\mathbb{R}_{y}^{N}))\times \mathcal{B}_{A_{\tau }}^{2}(\mathbb{R}%
_{\tau };\mathcal{B}_{A_{y}}^{2}(\mathbb{R}_{y}^{N})/\mathbb{R})$.

Now, Returning to Eq. (\ref{3.14}) where we set $\xi =\nabla \boldsymbol{u}%
_{0}(t,x)$, we find out that Eqs. (\ref{3.13}) and (\ref{3.14}) are the
variational formulation of the same problem (say Eq. (\ref{3.15})). Owing to
the uniqueness of the solution of Eq. (\ref{3.14}), we have $\boldsymbol{u}%
_{1}=u_{\nabla \boldsymbol{u}_{0}}$ and $p=p_{\nabla \boldsymbol{u}_{0}}$
where $u_{\nabla \boldsymbol{u}_{0}}$ (resp. $p_{\nabla \boldsymbol{u}_{0}}$%
) denotes the function $(t,x)\mapsto u_{\nabla \boldsymbol{u}_{0}(x,t)}$
(resp. $(t,x)\mapsto p_{\nabla \boldsymbol{u}_{0}(x,t)}$) defined from $Q$
into $\mathcal{B}_{A_{\tau }}^{2}(\mathbb{R}_{\tau };\mathcal{B}_{\Div%
}^{1,2}(\mathbb{R}_{y}^{N}))$ (resp. $\mathcal{B}_{A_{\tau }}^{2}(\mathbb{R}%
_{\tau };\mathcal{B}_{A_{y}}^{2}(\mathbb{R}_{y}^{N})/\mathbb{R})$).

\subsection{Homogenization result}

Let us first define the effective coefficients. Let the matrices $\mathcal{C}%
_{k}=(c_{ij}^{k})_{1\leq i,j\leq N}$ ($k=0,1$) be defined as follows: for
any $\xi =(\xi _{ij})_{1\leq i,j\leq N}$, 
\begin{equation}
\mathcal{C}_{0}\xi =\int_{\Delta (A)}\widehat{A}_{0}(\xi +\partial \widehat{u%
}_{\xi })d\beta ,\ \ \mathcal{C}_{1}\xi =\int_{\Delta (A)}\left( \widehat{A}%
_{1}\ast \ast (\xi +\partial \widehat{u}_{\xi })\right) d\beta  \label{4.10}
\end{equation}%
Then, thanks to the uniqueness of $u_{\xi }$ (for a given $\xi $), the
matrices $\mathcal{C}_{k}$ are well defined and are symmetric. It is obvious
that the $c_{ij}^{k}$ are obtained by choosing in Eq. (\ref{4.10}) $\xi
=(\delta _{ij})_{1\leq i,j\leq N}$ (the identity matrix), $\delta _{ij}$
being the Kronecker delta.

The matrix $\mathcal{C}_{k}$ are the \emph{effective homogenized viscosities}
which depend continuously on $(t,x)\in Q$ as seen in the next result whose
classical proof is omitted.

\begin{proposition}
\label{p4.1}It holds that

\begin{itemize}
\item[(i)] $\mathcal{C}_{i}$ $(i=0,1)$ are symmetric and further $\mathcal{C}%
_{i}\in \mathcal{C}(Q)^{N^{2}}$;

\item[(ii)] $\mathcal{C}_{0}\lambda \cdot \lambda \geq \alpha \left\vert
\lambda \right\vert ^{2}$ for all $(x,t)\in Q$ and all $\lambda \in \mathbb{R%
}^{N}$, where $\alpha $ is the same as in assumption \emph{(\textbf{A1})}.
\end{itemize}
\end{proposition}

We can now formulate the homogenized problem. To this end, consider Eq. (\ref%
{3.12}) in which we take $\boldsymbol{u}_{1}=u_{\nabla \boldsymbol{u}_{0}}$.
We get 
\begin{equation*}
\left\{ 
\begin{array}{l}
-\int_{Q}\boldsymbol{u}_{0}\cdot \boldsymbol{\psi }_{0}^{\prime
}dxdt+\int_{Q}\left[ \int_{\Delta (A)}\left( \widehat{A_{0}}(\nabla 
\boldsymbol{u}_{0}+\partial \widehat{u}_{\nabla \boldsymbol{u}_{0}})+%
\widehat{A}_{1}\ast \ast (\nabla \boldsymbol{u}_{0}+\partial \widehat{u}%
_{\nabla \boldsymbol{u}_{0}})\right) d\beta \right] \cdot \nabla \boldsymbol{%
\psi }_{0}\;dxdt \\ 
\ \ -\int_{Q}p_{0}\Div\boldsymbol{\psi }_{0}dxdt=-\int_{Q}\left[ \int_{%
\mathbb{R}^{N}}\left( \int_{\Delta (A)}\widehat{f}_{0}d\beta \right) (%
\boldsymbol{u}_{0}-v)dv\right] \cdot \boldsymbol{\psi }_{0}\;dxdt%
\mbox{ for
all }\boldsymbol{\psi }_{0}\in \mathcal{C}_{0}^{\infty }(Q)^{N}\mbox{,}%
\end{array}%
\right.
\end{equation*}%
which is just the variational formulation (where accounting of $\Div%
\boldsymbol{u}_{0}=0$) of the following anisotropic nonlocal Stokes system 
\begin{equation}
\left\{ 
\begin{array}{l}
\frac{\partial \boldsymbol{u}_{0}}{\partial t}-\Div(\mathcal{C}_{0}\nabla 
\boldsymbol{u}_{0}+\int_{0}^{t}\mathcal{C}_{1}(t-\tau ,x)\nabla \boldsymbol{u%
}_{0}(x,\tau )d\tau )+\nabla p_{0}=-\int_{\mathbb{R}^{N}}f(\boldsymbol{u}%
_{0}-v)\,dv\mbox{ in }Q \\ 
\Div\boldsymbol{u}_{0}=0\mbox{ in }Q \\ 
\boldsymbol{u}_{0}=0\mbox{ on }\partial \Omega \times (0,T) \\ 
\boldsymbol{u}_{0}(x,0)=\boldsymbol{u}^{0}(x)\mbox{ in }\Omega%
\end{array}%
\right.  \label{4.15}
\end{equation}%
where $f=\int_{\Delta (A)}\widehat{f}_{0}d\beta $. Finally, to be more
concise, let us put together Eq. (\ref{4.6})-(\ref{4.6''}): 
\begin{equation}
\left\{ 
\begin{array}{l}
\frac{\partial f_{0}}{\partial t}+v\cdot \overline{\nabla }_{y}f_{0}+\Div%
_{v}\left( (\boldsymbol{u}_{0}-v)f_{0}\right) =0\mbox{ in }\mathcal{O}\times 
\mathbb{R}_{y}^{N} \\ 
f_{0}(t,x,y,v)=f_{0}(t,x,y,v^{\ast })\mbox{ for }x\in \partial \Omega 
\mbox{
with }v\cdot \nu (x)<0 \\ 
f_{0}(0,x,y,v)=f^{0}(x,v)\mbox{ for }(x,y,v)\in \Omega \times \mathbb{R}%
_{y}^{N}\times \mathbb{R}_{v}^{N}%
\end{array}%
\right.  \label{4.16}
\end{equation}%
where $v^{\ast }=v-2(v\cdot \nu (x))\nu (x)$.

In view of what has been done above, we see that the system (\ref{4.15})-(%
\ref{4.16}) possesses at least solution $(\boldsymbol{u}_{0},f_{0},p_{0})$
such that $\boldsymbol{u}_{0}\in L^{2}(0,T;V)\cap \mathcal{C}([0,T];H)$, $%
f_{0}\in L^{\infty }(0,T;L^{\infty }(\Omega \times \mathbb{R}^{N})\cap
L^{2}(\Omega \times \mathbb{R}^{N}))$ and $p_{0}\in L^{2}(0,T;L^{2}(\Omega )/%
\mathbb{R})$. We are therefore led to the following homogenization result
which is the second main result of this work.

\begin{theorem}
\label{t4.1}Assume that \emph{(\textbf{A1})-(\textbf{A3})} hold. For each $%
\varepsilon >0$, let $(\boldsymbol{u}_{\varepsilon },f_{\varepsilon
},p_{\varepsilon })$ be a solution to \emph{(\ref{1.1})-(\ref{1.5})}. Then
up to a subsequence, the sequence $(\boldsymbol{u}_{\varepsilon
})_{\varepsilon >0}$ strongly converges in $L^{2}(Q)^{N}$ to $\boldsymbol{u}%
_{0}$, the sequence $(f_{\varepsilon })_{\varepsilon >0}$ weakly $\Sigma $%
-converges in $L^{2}(Q\times \mathbb{R}_{v}^{N})$ towards $f_{0}$ and the
sequence $(p_{\varepsilon })_{\varepsilon >0}$ weakly converges in $%
L^{2}(0,T;L^{2}(\Omega )/\mathbb{R})$ towards $p_{0}$, where $(\boldsymbol{u}%
_{0},f_{0},p_{0})$ is a solution to the system\emph{\ (\ref{4.15})-(\ref%
{4.16})}. Moreover any weak $\Sigma $-limit point $(\boldsymbol{u}%
_{0},f_{0},p_{0})$ in $L^{2}(0,T;V)\cap \mathcal{C}([0,T];H)\times L^{\infty
}(0,T;L^{\infty }(\Omega \times \mathbb{R}^{N})\cap L^{2}(\Omega \times 
\mathbb{R}^{N}))\times L^{2}(0,T;L^{2}(\Omega )/\mathbb{R})$ of $(%
\boldsymbol{u}_{\varepsilon },f_{\varepsilon },p_{\varepsilon
})_{\varepsilon >0}$ is a solution to Problem \emph{(\ref{4.15})-(\ref{4.16})%
}.
\end{theorem}

\section{Some applications}

A look at the previous section reveals that the homogenization process has
been made possible thanks to Assumption (\textbf{A3}). This assumption is
formulated in a general manner encompassing a variety of concrete behaviors
of the coefficients of the operator involved in (\ref{1.2}). We aim at
providing in this section some natural situations leading to the
homogenization of (\ref{1.1})-(\ref{1.5}). First and foremost, it is an easy
task (using \cite{NA2014}) to see that all the algebras wmv involved in the
following problems are introverted.

\subsection{Problem 1 (Periodic homogenization)}

The homogenization of (\ref{1.1})-(\ref{1.5}) can be achieved under the
periodicity assumption

\begin{itemize}
\item[(\textbf{A3})$_{1}$] The functions $A_{i}(t,x,\cdot ,\cdot )$\ ($i=0,1$%
) are periodic of period $1$\ in each scalar coordinate.
\end{itemize}

\noindent This leads to (\textbf{A3}) with\emph{\ }$A=\mathcal{C}_{\mbox{per}%
}(Z\times Y)=\mathcal{C}_{\mbox{per}}(Z)\odot \mathcal{C}_{\mbox{per}}(Y)$\
(the product algebra, with $Y=(0,1)^{N}$\ and $Z=(0,1)$), and hence $%
B_{A}^{2}(\mathbb{R}_{y,\tau }^{N+1})=L_{\mbox{per}}^{2}(Z\times Y)$.

\subsection{Problem 2 (Almost periodic homogenization)}

The above functions in (\textbf{A3})$_{1}$\ are both almost periodic in $%
(\tau ,y)$ in the sense of Besicovitch \cite{Besicovitch}. This amounts to (%
\textbf{A3}) with\emph{\ }$A=AP(\mathbb{R}_{y,\tau }^{N+1})=AP(\mathbb{R}%
_{\tau })\odot AP(\mathbb{R}_{y}^{N})$\ ($AP(\mathbb{R}_{y}^{N})$\ the Bohr
almost periodic functions on $\mathbb{R}_{y}^{N}$ \cite{Bohr}).

\subsection{Problem 3 (Weak almost periodic homogenization)}

The homogenization problem for (\ref{1.1})-(\ref{1.5}) may also be
considered under the assumption

\begin{itemize}
\item[(\textbf{A3})$_{2}$] $A_{i}(t,x,\cdot ,\cdot )$\ ($i=0,1$) is weakly
almost periodic \cite{Eberlein}. This leads to (\textbf{A3}) with $A=WAP(%
\mathbb{R}_{\tau })\odot WAP(\mathbb{R}_{y}^{N})$\emph{\ }($WAP(\mathbb{R}%
_{y}^{N})$, the algebra of continuous weakly almost periodic functions on $%
\mathbb{R}_{y}^{N}$; see e.g., \cite{Eberlein}).
\end{itemize}

\subsection{Problem 4}

Let $F$ be a Banach subalgebra of \textrm{BUC}$(\mathbb{R}^{m})$. Let $%
\mathcal{B}_{\infty }(\mathbb{R}^{d};F)$ denote the space of all continuous
functions $\psi \in \mathcal{C}(\mathbb{R}^{d};F)$ such that $\psi (\zeta )$
has a limit in $F$ as $\left\vert \zeta \right\vert \rightarrow \infty $. In
particular, it is known that $\mathcal{B}_{\infty }(\mathbb{R}^{d};\mathbb{R}%
)\equiv \mathcal{B}_{\infty }(\mathbb{R}^{d})$.

With this in mind, our goal here is to study the homogenization for problem (%
\ref{1.1})-(\ref{1.5}) under the hypothesis

\begin{itemize}
\item[(\textbf{A3})$_{3}$] $A_{i}(t,x,\cdot ,\cdot )\in \mathcal{B}_{\infty
}(\mathbb{R}_{\tau };L_{\mbox{per}}^{2}(Y))$ for any $(t,x)\in Q$, where $%
Y=(0,1)^{N}$.
\end{itemize}

It is an easy task to see that the appropriate algebra here is the product
algebra $A=\mathcal{B}_{\infty }(\mathbb{R}_{\tau })\odot \mathcal{C}_{%
\mbox{per}}(Y)$.

\section{Conclusion}

In this work, we have constructed weak solutions of a nonlocal Stokes-Vlasov
system without any assumptions on high-order velocity moments of the initial
distribution of particles. Our approach consisted in applying Schauder's
fixed point theorem to a carefully regularized version of the Stokes-Vlasov
system and passing to the limit by means of compactness arguments. Our
investigation culminated with the homogenization of Stokes-Vlasov system
under generous structural assumptions on coefficients encompassing various
forms of classical behaviors.

\end{document}